\documentclass[10pt,a4paper]{scrartcl}

\usepackage[utf8]{inputenc}
\usepackage[T1]{fontenc}
\usepackage[a4paper,margin=2.5cm]{geometry}
\usepackage{amsmath}
\usepackage{amssymb}
\usepackage{amsthm}
\usepackage{amsfonts}
\usepackage{multirow}
\usepackage{array}
\usepackage{appendix}
\usepackage[svgnames]{xcolor}
\usepackage{hyperref}
\usepackage{ulem}
\hypersetup{
    colorlinks,
    citecolor=SeaGreen,
    filecolor=black,
    linkcolor=Indigo,
    urlcolor=black
}

\usepackage{mathtools}
\usepackage{dsfont}
\usepackage{graphicx}
\usepackage{fancyhdr}
\usepackage{geometry}
\usepackage{stmaryrd}
\usepackage{booktabs}

\theoremstyle{plain}
\newtheorem{theorem}{{Theorem}}[section] 
\newtheorem*{theorem*}{{Theorem}}
\newtheorem{proposition}[theorem]{Proposition}
\newtheorem*{proposition*}{Proposition}
\newtheorem{corollary}[theorem]{Corollary}
\newtheorem*{corollary*}{Corollary}
\newtheorem{lemma}[theorem]{Lemma}
\newtheorem*{lemma*}{Lemma}

\newtheorem*{assumption*}{Assumption}
\newtheorem{definition}[theorem]{Definition}
\newtheorem*{definition*}{Definition}

\theoremstyle{remark}
\newtheorem*{notation*}{Notation}
\newtheorem*{remark*}{Remark}
\newtheorem{remark}{Remark}

\newcommand{\Trace}{\mathrm{Tr}}
\newcommand{\E}{\mathbb{E}}
\newcommand{\V}{\mathbb{V}}
\newcommand{\Cov}{\mathrm{Cov}}
\newcommand{\hproj}[2]{p^{#1}_{(#2)}}
\newcommand{\Card}{\mathrm{Card}}
  
\def\projecttitle{
    
}
  
\RequirePackage{fancyhdr}
\title{Hoeffding-type decomposition for $U$-statistics on bipartite networks}
\subtitle{}

\author{Tâm \textsc{Le Minh}, Sophie \textsc{Donnet}, Fran\c{c}ois \textsc{Massol}, and St\'ephane \textsc{Robin}}
  
\graphicspath{ {./figures/} }

\begin{document}

\begin{center}
{\Large
	\textsc {Hoeffding-type decomposition for $U$-statistics on bipartite networks }
}
\bigskip

 T\^am Le Minh $^{1}$ \& Sophie Donnet $^{1}$ \& Fran\c{c}ois Massol $^{2}$ \& St\'ephane Robin $^{3}$
\bigskip

\textit{
$^{1}$ Universit\'e Paris-Saclay, AgroParisTech, INRAE, UMR MIA Paris-Saclay, 91120, Palaiseau, France \\
$^{2}$ Univ. Lille, CNRS, Inserm, CHU Lille, Institut Pasteur de Lille, U1019 - UMR 9017 - CIIL - Center for Infection and Immunity of Lille, F-59000 Lille, France \\
$^{3}$ Sorbonne Universit\'e, CNRS, Laboratoire de Probabilit\'es, Statistique et Mod\'elisation, UMR 8001, 75005 Paris, France
}
\end{center}

\paragraph{Abstract.}
We consider a broad class of random bipartite networks, the distribution of which is invariant under permutation within each type of nodes. We are interested in $U$-statistics defined on the adjacency matrix of such a network, for which we define a new type of Hoeffding decomposition based on the Aldous-Hoover-Kallenberg representation of row-column exchangeable matrices. This decomposition enables us to characterize non-degenerate $U$-statistics -- which are then asymptotically normal -- and provides us with a natural and easy-to-implement estimator of their asymptotic variance. \\
We illustrate the use of this general approach on some typical random graph models and use it to estimate or test some quantities characterizing the topology of the associated network. We also assess the accuracy and the power of the proposed estimates or tests, via a simulation study.

\paragraph{Keywords.} bipartite networks, $U$-statistics, variance estimation, Hoeffding decomposition, row-column exchangeability, Central Limit Theorem

\section{Introduction}

Networks are popular objects to represent a set of interacting entities.  The last decades have witnessed an explosion in the number of network datasets. The fields of application range from sociology to ecology, from economics to computer science.  Understanding the organization of the network is a first step towards a better insight into the system it represents.  Several strategies exist to study or describe the topology of a network. Many of them are based on the calculation of one or several numeric quantities (statistics) such as density, clustering coefficients, or counts of given motifs to name but a few.  These statistics generally rely on several nodes. 

The calculation of these numerical quantities on a given network naturally leads to comparing them to a reference value, or to the value obtained on another network. The concept of hypothesis tests naturally meets this expectation. The challenging step of statistical hypothesis testing is identifying the statistic distribution under the null assumption. In particular, one class of statistics considered is the $U$-statistics which, in the context of network analysis, have complex dependencies.

\paragraph{Networks and dissociated RCE matrices.} In networks, entities are represented by nodes which are linked by edges when they interact. In bipartite networks,  the nodes are divided into two types and the interactions only happen between nodes of the two different types. Some examples of bipartite networks connect users and items in recommender systems \cite{zhou2007bipartite}, papers and scientists in authorship networks \cite{newman2001structure}, or plants and pollinators in ecological interaction networks \cite{dormann2009indices}.  The networks are naturally encoded in matrices.  In the adjacency matrix $Y$ of a bipartite network (sometimes also called incidence matrix), the two types of nodes are represented by rows and columns, so that $Y_{ij}$ encodes the interaction between entity $i$ of the first type and entity $j$ of the second type. In binary networks, $Y_{ij} = 1$ if $i$ and $j$ interact, else $Y_{ij} = 0$. Some networks are weighted, meaning that $Y_{ij}$ represents the intensity of the interaction.

We consider the asymptotic framework where $Y$ is an infinite adjacency matrix and the adjacency matrix of an observed network of size $m \times n$ is the submatrix extracted from the leading $m$ rows and $n$ columns of $Y$. Probabilistic models define a joint distribution on the values of the matrix entries. In random graph models, it is common to assume that the nodes of the networks are exchangeable, i.e. that the distribution of the network does not change if its nodes are permuted. For instance, the stochastic blockmodel \cite{snijders1997estimation}, the random dot product graph model \cite{young2007random} or the latent space model \cite{hoff2002latent} are all node-exchangeable. On the corresponding adjacency matrix of a bipartite network, this assumption implies the row-column exchangeability. $Y$ is said to be row-column exchangeable (RCE) if for any couple $\Phi = (\sigma_1,\sigma_2)$ of finite permutations of $\mathbb{N}$,  
\begin{equation*}
    \Phi Y \overset{\mathcal{D}}{=} Y,
\end{equation*}
where $\Phi Y := (Y_{\sigma_1(i)\sigma_2(j)})_{i \ge 1, j \ge 1}$. Many exchangeable random graph models also have a dissociatedness property, i.e. their adjacency matrices are also dissociated \cite{silverman1976limit, lauritzen2018random}. An RCE matrix is said to be dissociated if for all $m$ and $n$, $(Y_{ij})_{1 \le i \le m, 1 \le j \le n}$ is independent of $(Y_{ij})_{i > m, j > n}$. In the present work, we only consider RCE dissociated matrices.

\paragraph{$U$-statistics and Hoeffding decomposition.} 
$U$-statistics are a generalization of the empirical mean to functions of more than one variable. Many estimators fall under the category of $U$-statistics. Given a sequence of random variables $Y = (Y_i)_{i \ge 1}$ numbered with a unique index a $U$-statistic $U^h_n(Y)$ of order $n$ and kernel function $h$ is defined as the following average
\begin{equation*}
    U^h_n(Y) = \binom{n}{k}^{-1} \sum_{1 \le i_1 < i_2 < ... < i_k \le n} h(Y_{i_1}, Y_{i_2}, ..., Y_{i_k}),
\end{equation*}
where $h : \mathbb{R}^k \rightarrow \mathbb{R}$ is a symmetric function referred to as the kernel. Denote $\llbracket n \rrbracket := \{ 1, ..., n \}$ and for a set $A$, $\mathcal{P}_{k}(A)$ the set of all subsets of cardinal $k$ of $A$. Let $\mathbf{i} = \{i_1, ..., i_k\} \in \mathcal{P}_k(\llbracket n \rrbracket)$, then by symmetry of $h$, $h(Y_{i_1}, ..., Y_{i_k})$ does not depend on the order of the elements of $\mathbf{i}$. Therefore, we will denote $h(Y_{\mathbf{i}}) := h(Y_{i_1}, ..., Y_{i_k})$. Finally, the $U$-statistic $U^h_n(Y)$ can be formulated as follows:
\begin{equation*}
    U^h_n(Y) = \binom{n}{k}^{-1} \sum_{\mathbf{i} \in \mathcal{P}_k(\llbracket n \rrbracket)} h(Y_{\mathbf{i}}).
\end{equation*}
When $Y$ is an exchangeable sequence, $h(Y_{\mathbf{i}})$ has the same distribution for all $\mathbf{i} \in \mathcal{P}_k(\llbracket n \rrbracket)$, therefore $U^h_n(Y)$ is an unbiased estimate of $h(Y_{\llbracket k \rrbracket})$. The case where $Y$ is an i.i.d. sequence is well-studied: the $U$-statistics are known to be asymptotically normal \cite{hoeffding1948class} and can be used for inference tasks such as estimation and hypothesis testing. 

In the i.i.d. case, a useful technique to study the asymptotic behavior of $U$-statistics is the Hoeffding decomposition, formalized for the first time in~\cite{hoeffding1961strong}. For $1 \le c \le k$, define the function $\psi^ch : \mathbb{R}^c \rightarrow \mathbb{R}$ as 
\begin{equation*}
    \psi^ch : (y_1, ..., y_c) \longmapsto \E[h(Y_{1}, ..., Y_{k}) \mid Y_1 = y_{1}, ..., Y_c = y_c].
\end{equation*} 
Again, by symmetry of $h$, for some set $\mathbf{i} \in \mathcal{P}_c(\llbracket n \rrbracket)$, we can denote $\psi^ch(Y_{\mathbf{i}}) := \psi^ch(Y_{i_1}, ..., Y_{i_c})$ since the order of the elements of $\mathbf{i}$ does not matter. Set $p^0h = \E[h(Y_{\llbracket k \rrbracket})]$ and define recursively
\begin{equation*}
    p^ch(Y_{\mathbf{i}}) = \psi^ch(Y_{\mathbf{i}}) - \sum_{c'=0}^{c-1} \sum_{ \mathbf{i}' \in \mathcal{P}_{c'} (\mathbf{i})} p^{c'}h(Y_{\mathbf{i}'}).
\end{equation*}
for all subsets $\mathbf{i} \in \mathcal{P}_c(\llbracket n \rrbracket)$, for all $1 \le c \le k$. Then, for   $\mathbf{i} \in \mathcal{P}_k(\llbracket n \rrbracket)$, $h(Y_{\mathbf{i}})$ can be written
\begin{equation*}
    h(Y_{\mathbf{i}}) = \sum_{0 \le c \le k} \sum_{\mathbf{i}' \in \mathcal{P}_{c}(\mathbf{i})} p^ch(Y_{\mathbf{i}'}).
\end{equation*}
The $U$-statistic $U^h_n$ can be written as
\begin{equation*}
    U^h_n(Y) = \sum_{c=0}^{k} \binom{k}{c} P^c_{n} h(Y),
\end{equation*}
where for $1 \le c \le k$, $P^c_{n}h(Y) = \binom{n}{c}^{-1} \sum_{ \mathbf{i}  \in \mathcal{P}_c(\llbracket n \rrbracket)} p^ch(Y_{\mathbf{i}})$.

This decomposition is interesting as all the quantities $p^ch(Y_{\mathbf{i}})$ are orthogonal. By extension the $U$-statistics $P^c_{n} h$ are also orthogonal. The leading terms of this decomposition have been used by~\cite{hoeffding1948class} to prove the asymptotic normality of $U$-statistics. The decomposition also yields a decomposition of the variance of $U$-statistics.

This framework naturally extends to a network $U$-statistics setting, where the kernel $h$ is defined on submatrices of size $p \times q$. Let $Y$ be an infinite adjacency matrix from which we observe the first $m$ rows and $n$ columns. Let $h : \mathcal{M}_{p,q}(\mathbb{R}) \rightarrow \mathbb{R}$ be a function defined on $p \times q$ matrices, $1 \le p \le m$, $1 \le q \le n$, verifying the following symmetry property: for all $(\sigma_1, \sigma_2) \in \mathbb{S}_p \times \mathbb{S}_q$, 
\begin{equation}\label{eq:sym}
h(Y_{(i_{\sigma_1(1)},...,i_{\sigma_1(p)};j_{\sigma_2(1)},...,j_{\sigma_2(q)})}) = h(Y_{(i_1,i_2,...,i_p;j_1,j_2,...,j_q)}),
\end{equation}
where $Y_{(i_1,...,i_p;j_1,...,j_q)}$ is the $p \times q$ submatrix consisting of the rows and columns of $Y$ indexed by $i_1,...,i_p$ and $j_1,...,j_q$ respectively. Therefore, since the order of the elements of $\mathbf{i} = \{ i_1, ..., i_p \}$ and $\mathbf{j} = \{ j_1, ..., j_q \}$ does not matter, we can denote $h(Y_{\mathbf{i},\mathbf{j}}) := h(Y_{(i_1,i_2,...,i_p;j_1,j_2,...,j_q)})$.
Then the associated $U$-statistic is  
\begin{equation}
    U^h_{m,n}(Y) = \binom{m}{p}^{-1} \binom{n}{q}^{-1} \sum_{\substack{\mathbf{i} \in \mathcal{P}_p(\llbracket m \rrbracket)\\ \mathbf{j} \in \mathcal{P}_q(\llbracket n \rrbracket)}} h(Y_{\mathbf{i},\mathbf{j}}).
    \label{eq:kernel_symmetry}
\end{equation}
Note that the assumption on the symmetry of $h$ can be made without loss of generality. Indeed, if $h^0 : \mathcal{M}_{p,q}(\mathbb{R}) \rightarrow \mathbb{R}$ is not symmetric, then $h : \mathcal{M}_{p,q}(\mathbb{R}) \rightarrow \mathbb{R}$ defined by
\begin{equation}
    h(Y_{(i_1,i_2,...,i_p;j_1,j_2,...,j_q)}) = (p!q!)^{-1} \sum_{(\sigma_1, \sigma_2) \in \mathbb{S}_p \times \mathbb{S}_q} h^0(Y_{(i_{\sigma_1(1)},...,i_{\sigma_1(p)};j_{\sigma_2(1)},...,j_{\sigma_2(q)})})
    \label{eq:kernel_symmetric_version}
\end{equation}
 verifies Equation~\eqref{eq:sym} and leads to the same $U$-statistic ($U^{h^0}_{m,n}(Y) = U^h_{m,n}(Y)$). 

Formally, the network $U$-statistic defined by \eqref{eq:kernel_symmetry} is a weighted $U$-statistic \cite{shapiro1979asymptotic}, and more precisely, an incomplete $U$-statistic \cite{blom1976some}, where the weights are binary and determined by a sampling design. Flattening the matrix $Y$ into a vector in $\mathbb{R}^{m \times n}$, and interpreting $h$ as a function of $p \times q$ variables, clarifies this perspective. The asymptotic properties of weighted $U$-statistics \cite{brown1978reduced, janson1984asymptotic, oneil1993asymptotic} have been studied extensively, although a general distributional characterization has been derived only for the case of linear kernels \cite{bhattacharya2024fluctuations}. Moreover, most existing results rely on the assumption of i.i.d. observations. In contrast, the RCE structure of $Y$ implies a different dependency pattern and suggests a natural sampling design that respects both row and column symmetries.

The asymptotic behavior of network $U$-statistics fundamentally depends on the dependency structure of the data $Y$. Many studies have considered the case where the matrix $Y$ represents an exchangeable network, more precisely through the investigation of subgraph counts \cite{bickel2011method, kaur2021higher, ouadah2022motif, shao2022higher, bhattacharya2023fluctuations}. Some have investigated their asymptotic behavior through the theory of generalized $U$-statistics \cite{janson1991asymptotic}. However, most of them consider \textit{unipartite} networks. These results do not apply to \textit{bipartite} networks, since they have two disjoint sets of nodes which can be of different sizes. \cite{wang2015u} considered bipartite graphs, but the $U$-statistics they have defined are averages over only the pairs of present edges not sharing any node. They do not make assumptions on the distribution of networks, so they do not derive the limit distribution of these $U$-statistics. \cite{davezies2021empirical} investigated averages of separately exchangeable arrays, which corresponds to the network $U$-statistic $U^h_{m,n}$ on $Y$, where $Y$ is an RCE matrix but $h$ is a function of only one variable. Therefore, their results do not directly apply. More closely related, \cite{leminh2023ustatistics} used a martingale approach to obtain a weak convergence result for $U^h_{m,n}$ when $m$ and $n$ grow to infinity at the same rate. Applying this result requires specific development to get the asymptotic variance. 

In this paper, we propose a Hoeffding decomposition-based approach to analyze $U^h_{m,n}$. In our context, the Hoeffding decomposition yields a simpler variance decomposition than for classic generalized $U$-statistics. This strategy also has the advantage of providing a method to estimate the asymptotic variance of network $U$-statistics. Indeed, estimating this variance is a required condition to perform practical inference tasks, such as hypothesis testing. However, it remains a complex problem that has been tackled with various methods in the literature.

\paragraph{Variance estimation of $U$-statistics.}
The standard error of $U$-statistics of i.i.d. random variables is most often computed using resampling techniques such as the jackknife \cite{arvesen1969jackknifing} and bootstrap \cite{efron1979bootstrap, bickel1981some} estimators of variance. \cite{sen1960some, sen1977some, callaert1981order, schucany1989small} suggested various estimators of the asymptotic variance. However, all these estimators are biased for both the variance and the asymptotic variances. \cite{callaert1981order, schucany1989small} also discussed unbiased estimators, but they are computationally more demanding than all the previous estimators and they find them to have a positive probability of being negative, which is undesirable. For $U$-statistics on RCE matrices, no generic variance estimator has been proposed. The bootstrap procedures considered in \cite{davezies2021empirical, menzel2021bootstrap} are proven to be consistent for the sample average of an RCE matrix, but have not been extended to $U$-statistics. Previous studies have estimated the variance by identifying estimable quantities in the problem-specific analytical expressions \cite{leminh2023ustatistics, leminh2024characterization}.

\paragraph{Contribution.}

We show how $U$-statistics of size $p \times q$ can be used for exchangeable network inference. In particular, we propose a Hoeffding-type decomposition to identify the asymptotic distribution of these $U$-statistics. Due to the RCE structure of $Y$, this decomposition uses projections on orthogonal spaces generated by latent variables obtained through the Aldous-Hoover-Kallenberg (AHK) representation. In our asymptotic framework, the total number of rows and columns, denoted by $N$, tends to infinity. To treat this setting, we assume that the matrix dimensions $m_N$ and $n_N$ grow at the same rate, satisfying $N = m_N + n_N$, with $m_N/N \rightarrow \rho$ and $n_N/N \rightarrow 1-\rho$, where $\rho \in (0,1)$. The case in which the dimensions grow at different rates ($\rho = 0$ or $\rho = 1$) is not considered in this paper, although our results and proofs can be generalized to accommodate this scenario. For simplification, we denote $U^h_N(Y) := U^h_{m_N,n_N}(Y)$. First, we show that $\sqrt{N}(U^h_N(Y) - \E[h])$ converges to a Gaussian distribution, except in degenerate cases. Then, for these $U$-statistics, we build a computationally efficient estimator for their variances. This variance estimator relies on the prior estimation of conditional expectations given the AHK latent variables, which makes it unique compared to traditional ways of estimating the variance of $U$-statistics.

\paragraph{Outline.}

In Section~\ref{sec:ahk}, we recall the Aldous-Hoover-Kallenberg (AHK) representation for RCE matrices. In Section~\ref{sec:hoeffding}, we leverage the AHK representation to derive an orthogonal decomposition for $U$-statistics on RCE matrices. Section \ref{sec:asymptotic_normality} exploits this new Hoeffding-type decomposition to establish the asymptotic normality of the $U$-statistics under consideration. In Section~\ref{sec:variance_estimator}, we use this decomposition to construct an efficient estimator for the asymptotic variance. Section~\ref{sec:examples} demonstrates how our results can be applied to define a methodological framework for bipartite network analysis, providing examples of models for RCE matrices and network analysis questions that can be addressed. Sections~\ref{sec:simulations} and~\ref{sec:illustrations} are dedicated to simulation studies and an illustration using a legislature dataset, showcasing this methodology.

\section{Aldous-Hoover-Kallenberg representation}
\label{sec:ahk}
\label{SEC:AHK}

Traditionally, orthogonal decompositions for $U$-statistics on i.i.d. observations \cite{hoeffding1961strong} or generalized $U$-statistics \cite{janson1991asymptotic} are obtained by partitioning the probability space into orthogonal subspaces generated by subsets of observations. However, this approach generally fails for $U$-statistics on RCE matrices, as the dependencies between matrix elements can prevent orthogonality. To ensure orthogonality, the decomposition proposed in this paper uses the Aldous–Hoover–Kallenberg representation of RCE matrices, which expresses the matrix as a function of i.i.d. random variables. Although these variables are unobserved, they provide a natural framework for constructing suitable orthogonal subspaces. Before specifying these subspaces and deriving the decomposition, we first introduce this representation.

\paragraph{Aldous-Hoover-Kallenberg representation.} Corollary 7.23 of~\cite{kallenberg2005probabilistic} states that for any dissociated RCE matrix $Y$, there exists $(\xi_i)_{i \geq 1}$, $(\eta_j)_{j \geq 1}$ and $(\zeta_{ij})_{i,j \geq 1}$ arrays of i.i.d. random variables with uniform distribution over $[0,1]$ and a real measurable function $\phi$ such that for all $1 \le i,j < \infty$, 
\begin{equation*}
    Y_{ij} \overset{a.s.}{=} \phi(\xi_i, \eta_j, \zeta_{ij}).
\end{equation*} 
With such a representation, the kernel function taken on a $p \times q$ submatrix indexed by the rows $\mathbf{i} \in \mathcal{P}_{p}(\mathbb{N})$ and columns $\mathbf{j} \in \mathcal{P}_{q}(\mathbb{N})$ can be written $h(Y_{\mathbf{i}, \mathbf{j}}) \overset{a.s.}{=} h_\phi((\xi_{i})_{i \in \mathbf{i}}; (\eta_{j})_{j \in \mathbf{j}}; (\zeta_{ij})_{\substack{i \in \mathbf{i} \\ j \in \mathbf{j}}})$, where $h_\phi$ is some function depending on $h$ and $\phi$.

\paragraph{Projection sets of AHK variables.} For $\mathbf{i}' \in \mathcal{P}(\mathbb{N})$ and $\mathbf{j}' \in \mathcal{P}(\mathbb{N})$, we define the sets of AHK variables defined as
\begin{equation*}
    A(\mathbf{i}', \mathbf{j}') := ((\xi_{i})_{i \in \mathbf{i}'}, (\eta_{j})_{j \in \mathbf{j}'}, (\zeta_{ij})_{\substack{i \in \mathbf{i}' \\ j \in \mathbf{j}'}}).
\end{equation*}
These sets play a crucial role, as they form the foundation of our $U$-statistic decomposition. More specifically, the decomposition will be defined through projections onto probability spaces generated by these sets of AHK variables. Consequently, we consider the $\sigma$-algebras given by
\begin{equation*}
    \mathcal{A}_{\mathbf{i}', \mathbf{j}'} := \sigma(A(\mathbf{i}', \mathbf{j}')) = \sigma((\xi_{i})_{i \in \mathbf{i}'}, (\eta_{j})_{j \in \mathbf{j}'}, (\zeta_{ij})_{\substack{i \in \mathbf{i}' \\ j \in \mathbf{j}'}}).
\end{equation*}

\paragraph{Notations.} In the rest of the paper, we assume that for each dissociated RCE matrix $Y$, we have picked an AHK representation, i.e. a suitable function $\phi$, and suitable i.i.d. random variables $(\xi_{i})_{i \ge 1}$, $(\eta_{j})_{j \ge 1}$ and $(\zeta_{ij})_{\substack{i \ge 1 \\ j \ge 1}}$. In the rest of the paper, we will write abusively, but without ambiguity, $Y_{ij} = \phi(\xi_i, \eta_j, \zeta_{ij})$ and 
\begin{equation}
  h(Y_{\mathbf{i}, \mathbf{j}}) = h_\phi((\xi_{i})_{i \in \mathbf{i}}; (\eta_{j})_{j \in \mathbf{j}}; (\zeta_{ij})_{\substack{i \in \mathbf{i} \\ j \in \mathbf{j}}}).
  \label{eq:ahk_function}
\end{equation}
Therefore, it follows from our notations that
\begin{equation*}
    \E[h(Y_{\mathbf{i}, \mathbf{j}}) \mid \mathcal{A}_{\mathbf{i}', \mathbf{j}'}] = \E[h_\phi((\xi_{i})_{i \in \mathbf{i}}; (\eta_{j})_{j \in \mathbf{j}}; (\zeta_{ij})_{\substack{i \in \mathbf{i} \\ j \in \mathbf{j}}}) \mid (\xi_{i})_{i \in \mathbf{i}'}; (\eta_{j})_{j \in \mathbf{j}'}; (\zeta_{ij})_{\substack{i \in \mathbf{i}' \\ j \in \mathbf{j}'}}].
\end{equation*}

For fixed sets $\mathbf{i}'$ and $\mathbf{j}'$, the quantity $\E[h(Y_{\mathbf{i}, \mathbf{j}}) \mid \mathcal{A}_{\mathbf{i}', \mathbf{j}'}]$ only depends on the elements shared by $\mathbf{i}$ and $\mathbf{i}'$ and the elements shared by $\mathbf{j}$ and $\mathbf{j}'$, and not on the other elements of $\mathbf{i}$, $\mathbf{i}'$, $\mathbf{j}$ and $\mathbf{j}'$. Suppose $r = \Card(\mathbf{i} \cap \mathbf{i}')$ and $c = \Card(\mathbf{j} \cap \mathbf{j}')$. Without loss of generality, we can assume that $\mathbf{i}' \in \mathcal{P}_r(\mathbf{i})$ and $\mathbf{j}' \in \mathcal{P}_c(\mathbf{j})$ so $\E[h(Y_{\mathbf{i}, \mathbf{j}}) \mid \mathcal{A}_{\mathbf{i}', \mathbf{j}'}]$ only depends on the $r$ elements of $\mathbf{i}'$ and the $c$ elements of $\mathbf{j}'$. Therefore, we can define the quantities $\psi^{r,c} h(Y_{\mathbf{i}', \mathbf{j}'})$ such that 
\begin{equation*}
    \psi^{r,c} h(Y_{\mathbf{i}', \mathbf{j}'}) := \E[h(Y_{\mathbf{i}, \mathbf{j}}) \mid \mathcal{A}_{\mathbf{i}', \mathbf{j}'}],
\end{equation*}
where the choice of $\mathbf{i}$ and $\mathbf{j}$ does not matter as long as $\mathbf{i}' \subset \mathbf{i}$ and $\mathbf{j}' \subset \mathbf{j}$. Note that $\psi^{r,c} h(Y_{\mathbf{i}', \mathbf{j}'})$ is simply a notation and not a function of $Y_{\mathbf{i}', \mathbf{j}'}$. If $\mathbf{i}' = \emptyset$ or $\mathbf{j}' = \emptyset$, we will still use this notation, for example
\begin{equation*}
    \psi^{r,c} h(Y_{\mathbf{i}', \emptyset}) = \E[h(Y_{\mathbf{i}, \mathbf{j}}) \mid \mathcal{A}_{\mathbf{i}', \emptyset}],
\end{equation*}
despite $Y_{\mathbf{i}', \emptyset}$ being undefined.

\section{Decomposition for $U$-statistics on RCE matrices}
\label{sec:hoeffding}
\label{SEC:HOEFFDING}

In this section, we derive a decomposition for $U$-statistics on RCE matrices based on the AHK representation, using the subspaces generated by the specific projection sets that we have defined. This decomposition follows from orthogonal projections on these subspaces. In this section and the following sections, for elements of $\mathbb{N}^2$, $(x,y) \leq (x',y')$ means that both $x \le x'$ and $y \le y'$; $(x,y) < (x',y')$ means that, in addition, $(x,y) \neq (x',y')$.

\paragraph{Hoeffding-type decomposition of the kernel.} 

For all $\mathbf{i} \in \mathcal{P}_{r}(\mathbb{N})$ and $\mathbf{j} \in \mathcal{P}_{c}(\mathbb{N})$, we define by recursion the following quantity:
\begin{equation}
    p^{r,c}h(Y_{\mathbf{i}, \mathbf{j}}) = \psi^{r,c} h(Y_{\mathbf{i}, \mathbf{j}}) - \sum_{(0,0) \le (r',c') < (r,c)} \sum_{\substack{\mathbf{i}' \in \mathcal{P}_{r'}(\mathbf{i}) \\ \mathbf{j}' \in \mathcal{P}_{c'}(\mathbf{j})}} p^{r',c'}h(Y_{\mathbf{i}', \mathbf{j}'}).
    \label{eq:def_projection}
\end{equation}
Since $\psi^{p,q} h(Y_{\mathbf{i}, \mathbf{j}}) = h(Y_{\mathbf{i}, \mathbf{j}})$ for $\mathbf{i} \in \mathcal{P}_{p}(\mathbb{N})$ and $\mathbf{j} \in \mathcal{P}_{q}(\mathbb{N})$,~\eqref{eq:def_projection} yields the following decomposition of the kernel function:
\begin{equation}\label{eq:def_projection_U}
    h(Y_{\mathbf{i}, \mathbf{j}}) = \sum_{(0,0) \le (r,c) \le (p,q)} \sum_{\substack{\mathbf{i}' \in \mathcal{P}_{r}(\mathbf{i}) \\ \mathbf{j}' \in \mathcal{P}_{c}(\mathbf{j})}} p^{r,c} h(Y_{\mathbf{i}', \mathbf{j}'}).
\end{equation}

\begin{remark}
    From this formula, we see that $h(Y_{\mathbf{i}, \mathbf{j}})$ is a linear combination of the projections $p^{r',c'}h(Y_{\mathbf{i}', \mathbf{j}'})$, for $0 \le r' \le p, 0 \le c' \le q, \mathbf{i}' \in \mathcal{P}_{r'}(\mathbf{i}), \mathbf{j}' \in \mathcal{P}_{c'}(\mathbf{j})$. Therefore, this is a linear combination of the respective conditional expectations $\psi^{r',c'} h(Y_{\mathbf{i}', \mathbf{j}'})$.
\end{remark}

\begin{remark}
    In practice, the AHK variables $(\xi_{i})_{i \ge 1}$, $(\eta_{j})_{j \ge 1}$ and $(\zeta_{ij})_{\substack{i \ge 1 \\ j \ge 1}}$ are usually unobserved. However, the existence of such variables is sufficient to define the conditional expectations of the form $\psi^{r,c} h(Y_{\mathbf{i}, \mathbf{j}})$.
\end{remark}

Now, we show that $p^{r,c}h(Y_{\mathbf{i}, \mathbf{j}})$ is the projection of $Y_{\mathbf{i}, \mathbf{j}}$ on the probability space generated by the projection set of AHK variables $A(\mathbf{i}, \mathbf{j})$, orthogonally to all the spaces generated by the sets $A(\mathbf{i}', \mathbf{j}')$, for $\mathbf{i}' \subset \mathbf{i}, \mathbf{j}' \subset \mathbf{j}$. This system of projection is analogous to the Hoeffding decomposition for the kernel functions of usual $U$-statistics on i.i.d. data. The next proposition, proven in Appendix~\ref{app:hoeffding}, states that the following properties hold, ensuring the orthogonality of the projection spaces generated by $A(\mathbf{i}_1, \mathbf{j}_1)$ and $A(\mathbf{i}_2, \mathbf{j}_2)$, if $(\mathbf{i}_1, \mathbf{j}_1) \neq (\mathbf{i}_2, \mathbf{j}_2)$.

\begin{proposition}
    Let $h_1$ and $h_2$ two kernel functions of respective size $p_1 \times q_1$ and $p_2 \times q_2$.
    \begin{enumerate}
    \item Let $(0,0) \le (r_1,c_1) \le (p_1,q_1) $ and $ (0,0) \le (r_2,c_2) \le (p_2,q_2)$ such that $(r_1,c_1) \neq (r_2,c_2)$. Let $(\mathbf{i}_1,\mathbf{j}_1) \in \mathcal{P}_{r_1}(\llbracket m \rrbracket) \times \mathcal{P}_{c_1}(\llbracket n \rrbracket)$ and $(\mathbf{i}_2,\mathbf{j}_2) \in \mathcal{P}_{r_2}(\llbracket m \rrbracket) \times \mathcal{P}_{c_2}(\llbracket n \rrbracket)$, then 
    \begin{equation*}
        \Cov(p^{r_1,c_1}h_1(Y_{\mathbf{i}_1,\mathbf{j}_1}), p^{r_2,c_2}h_2(Y_{\mathbf{i}_2,\mathbf{j}_2})) = 0.
    \end{equation*}
    \item Let $(r,c)$ such that $(0,0) \le (r,c) \le (p_1,q_1)$ and $(0,0) \le (r,c) \le (p_2,q_2)$. Let $(\mathbf{i}_1,\mathbf{j}_1)$ and $(\mathbf{i}_2,\mathbf{j}_2)$ two elements of $\mathcal{P}_r(\llbracket m \rrbracket) \times \mathcal{P}_c(\llbracket n \rrbracket)$. If $(\mathbf{i}_1,\mathbf{j}_1) \neq (\mathbf{i}_2,\mathbf{j}_2)$, then
    \begin{equation*}
        \Cov(p^{r,c}h_1(Y_{\mathbf{i}_1,\mathbf{j}_1}), p^{r,c}h_2(Y_{\mathbf{i}_2,\mathbf{j}_2})) = 0.
    \end{equation*}
    \end{enumerate}
    \label{prop:ortho_proj}
\end{proposition}

\paragraph{Decomposition of $U$-statistics.} Using the Hoeffding-type decomposition of kernel functions \eqref{eq:def_projection_U} the $U$-statistic \eqref{eq:kernel_symmetry} can be reformulated as:
\begin{align}
    U^h_{m,n}(Y) &= \binom{m}{p}^{-1} \binom{n}{q}^{-1} \sum_{\substack{1 \le i_1 < ... < i_p \le m\\1 \le j_1 < ... < j_q \le n}} \sum_{(0,0) \le (r,c) \le (p,q)} \sum_{\substack{\mathbf{i} \in \mathcal{P}_{r}(\{i_1,...,i_p\}) \\ \mathbf{j} \in \mathcal{P}_{c}(\{j_1,...,j_q\})}} p^{r,c} h(Y_{\mathbf{i}, \mathbf{j}}) \nonumber \\
    &= \binom{m}{p}^{-1} \binom{n}{q}^{-1} \sum_{(0,0) \le (r,c) \le (p,q)} \binom{m-r}{p-r} \binom{n-c}{q-c} \sum_{\substack{\mathbf{i} \in \mathcal{P}_{r}(\llbracket m \rrbracket) \\ \mathbf{j} \in \mathcal{P}_{c}(\llbracket n \rrbracket)}} p^{r,c} h(Y_{\mathbf{i}, \mathbf{j}}) \nonumber \\
    &= \sum_{(0,0) \le (r,c) \le (p,q)} \binom{p}{r} \binom{q}{c} P^{r,c}_{m,n}h(Y), \label{eq:decomposition_ustat}
\end{align}
where for all $0 \le r \le p$ and $0 \le c \le q$, $$P^{r,c}_{m,n}h(Y) = \binom{m}{r}^{-1}\binom{n}{c}^{-1} \sum_{\substack{\mathbf{i} \in \mathcal{P}_{r}(\llbracket m \rrbracket) \\ \mathbf{j} \in \mathcal{P}_{c}(\llbracket n \rrbracket)}} p^{r,c} h(Y_{\mathbf{i}, \mathbf{j}})$$ is the $U$-statistic of kernel function $p^{r,c} h$ taken on the first $m \times n$ rows and columns of the matrix $Y$. A consequence of the orthogonality of the projections of $h$ is the orthogonality of these $U$-statistics, as stated by Corollary~\ref{cor:ortho_ustats} provided in Appendix~\ref{app:hoeffding}. The orthogonality between the $P^{r_1,c_1}_{m,n}h_1(Y)$ and $P^{r_2,c_2}_{m,n}h_2(Y)$ allows to decompose the covariance of two $U$-statistics into a few covariance terms, as specified by Corollary~\ref{cor:cov_ustats}. 

\begin{corollary}
   \begin{equation*}
   \begin{split}
       &\Cov(U^{h_1}_{m,n}(Y), U^{h_2}_{m,n}(Y)) \\
       &= \sum_{(0,0) < (r,c) \le (p,q)} \binom{p}{r}^2 \binom{q}{c}^2 \binom{m}{r}^{-1} \binom{n}{c}^{-1} \Cov(p^{r,c}h_1(Y_{\llbracket r \rrbracket,\llbracket c \rrbracket}), p^{r,c}h_2(Y_{\llbracket r \rrbracket,\llbracket c \rrbracket})).
   \end{split}
   \end{equation*}
    \label{cor:cov_ustats}
\end{corollary}
This variance decomposition is particularly noteworthy and useful for analyzing the asymptotic properties of $U$-statistics. In contrast, the classic decomposition of (unipartite) graph $U$-statistics is typically more complex \cite{janson1991asymptotic}. Importantly, with Corollary~\ref{cor:cov_ustats}, each term in the decomposition~\eqref{eq:decomposition_ustat} has a distinct asymptotic order, providing a general understanding of variance contributions. In this regard, our decomposition more closely resembles the traditional decomposition of $U$-statistics on i.i.d. observations, where each covariance term is associated with a distinct binomial coefficient that determines its order of magnitude, rather than the decomposition in \cite{janson1991asymptotic}. In addition, the bipartite setting uniquely allows for separate contributions from the rows and the columns of observations. Later, the simplicity of this decomposition is exploited to introduce a novel variance estimator for these $U$-statistics.

\section{Asymptotic normality}
\label{sec:asymptotic_normality}
\label{SEC:ASYMPTOTIC_NORMALITY}

This section establishes the following Central Limit Theorem for $U^h_N$. In this section and the following sections, we will use simplified notations, summarizing the couple $(m_N, n_N)$ into $N = m_N + n_N$. We recall that $U^h_N(Y) = U^h_{m_N,n_N}(Y)$. We also denote $P^{r,c}_{N}h(Y) := P^{r,c}_{m_N,n_N}h(Y)$. When this is unambiguous, we will omit to mention $Y$, so we will simply write $U^h_N$, $P^{r,c}_N h$, $\hproj{r,c}{\mathbf{i}, \mathbf{j}}h$ and $\psi^{r,c}_{(\mathbf{i},\mathbf{j})} h$ instead of $U^h_N(Y)$, $P^{r,c}_N h(Y)$, $p^{r,c}h(Y_{\mathbf{i}, \mathbf{j}})$ and $\psi^{r,c} h(Y_{\mathbf{i},\mathbf{j}})$. Denote $v^{r,c}_h := \V[\psi^{r,c}_{(\llbracket r \rrbracket, \llbracket c \rrbracket)} h]$. 

\begin{theorem}
    Let $Y$ be a dissociated RCE matrix. Let $h$ be a $p \times q$ kernel function such that $\E[h(Y_{(1,...,p;1,...,q)})^2] < \infty$. Let $(m_N, n_N)_{N \ge 1}$ be a sequence of dimensions for the $U$-statistics, such that $\frac{m_N}{N} \xrightarrow[N \rightarrow \infty]{} \rho$ and $\frac{n_N}{N} \xrightarrow[N \rightarrow \infty]{} 1-\rho$, where $\rho \in (0,1)$. Let $(U^h_{N})_{N \ge 1}$ be the sequence of $U$-statistics associated with $h$ defined by $U^h_N := U^h_{m_N, n_N}$. Set $U^h_\infty = \mathbb{E}[h(Y_{(1,...,p;1,...,q)})]$ and 
    \begin{equation*}
    \begin{split}
        V^h =&~ \frac{p^2}{\rho}v^{1,0}_h + \frac{q^2}{1-\rho} v^{0,1}_h.
    \end{split}
    \end{equation*}
    If $V^h > 0$, then
    \begin{equation*}
        \sqrt{N}(U^h_{N}-U^h_\infty) \xrightarrow[N \rightarrow \infty]{\mathcal{D}} \mathcal{N}(0, V^h).
    \end{equation*}
    \label{th:asymptotic_normality}
\end{theorem}

This theorem comes from the decomposition of $\sqrt{N}(U^h_{N}-U^h_\infty)$ into three different terms, the limits of which are given by the following lemmas (proofs in Appendix~\ref{app:asymptotic_normality}).

\begin{lemma} 
If $v^{1,0}_h>0$, then we have
    \begin{equation*}
        \frac{1}{\sqrt{m}} \sum_{i = 1}^{m} \hproj{1,0}{\{i\},\emptyset}h \xrightarrow[m \rightarrow \infty]{\mathcal{D}} \mathcal{N}(0, v^{1,0}_h),
    \end{equation*}
    and if $v^{0,1}_h>0$, then we have
    \begin{equation*}
        \frac{1}{\sqrt{n}} \sum_{j = 1}^{n} \hproj{0,1}{\emptyset, \{j\}}h \xrightarrow[n \rightarrow \infty]{\mathcal{D}} \mathcal{N}(0, v^{0,1}_h).
    \end{equation*}
    \label{lem:asymptotic_normality_of_main_terms}
\end{lemma}

\begin{lemma}
    Let $A_N := \sqrt{N} \sum_{\substack{(0,0) < (r,c) \le (p,q) \\ (r,c) \neq (1,0) \neq (0,1)}} \binom{p}{r} \binom{q}{c} P^{r,c}_{N}h$. Then $A_N \xrightarrow[N \rightarrow \infty]{\mathbb{P}} 0$.
    \label{lem:clt_residue}
\end{lemma}

\begin{proof}[Proof of Theorem~\ref{th:asymptotic_normality}]
We have
\begin{equation*}
    \begin{split}
    U^h_{N} &= \sum_{(0,0) \le (r,c) \le (p,q)} \binom{p}{r} \binom{q}{c} P^{r,c}_{N}h \\
    &= P^{0,0}_{N}h + p P^{1,0}_{N}h + q P^{0,1}_{N}h + \sum_{\substack{(0,0) < (r,c) \le (p,q) \\ (r,c) \neq (1,0) \neq (0,1)}} \binom{p}{r} \binom{q}{c} P^{r,c}_{N}h.
    \end{split}
\end{equation*}
    First, we see that $P^{0,0}_{N}h = U^h_\infty$.
    Next, $A_{N}$ being defined in Lemma~\ref{lem:clt_residue}, we have
\begin{equation*}
    \sqrt{N} (U^h_{N} - U^h_\infty) = \frac{\sqrt{N} p}{m_N} \sum_{i = 1}^{m_N} \hproj{1,0}{\{i\},\emptyset}h + \frac{\sqrt{N} q }{n_N} \sum_{j = 1}^{n_N} \hproj{0,1}{\emptyset, \{j\}} h + A_N.
\end{equation*}
From Lemma~\ref{lem:clt_residue}, we have $A_N \xrightarrow[N \rightarrow \infty]{\mathbb{P}} 0$. So by Slutsky's theorem, $\sqrt{N} (U^h_{N} - U^h_\infty)$ has the same limiting distribution as the two main terms of this decomposition. From Lemma~\ref{lem:asymptotic_normality_of_main_terms}, this is the sum of two centered Gaussians of respective variance $\frac{p^2}{\rho} v^{1,0}_h$ and $\frac{q^2}{1-\rho} v^{0,1}_h$. Furthermore, $\sum_{i = 1}^{m} \hproj{1,0}{\{i\},\emptyset}h$ and $\sum_{j = 1}^{n} \hproj{0,1}{\{j\},\emptyset}h$ are independent, so the two Gaussians are independent, which concludes the proof.
\end{proof}

\begin{remark}
    The expression of $V^h$ could have been predicted with Corollary~\ref{cor:cov_ustats}. Indeed, this corollary implies that 
    \begin{equation*}
        \V[U^{h}_{N}] = \sum_{(0,0) < (r,c) \le (p,q)} \binom{p}{r}^2 \binom{q}{c}^2 \binom{m_N}{r}^{-1} \binom{n_N}{c}^{-1} \V[\hproj{r,c}{\llbracket r \rrbracket,\llbracket c \rrbracket}h],
    \end{equation*}
    so 
    \begin{equation*}
    \begin{split}
        \lim_{N\rightarrow \infty} N \V[U^{h}_{N}] &= \lim_{N\rightarrow \infty}  \left(\frac{p^2 N}{m_N} \V[\hproj{1,0}{\{1\},\emptyset}h] + \frac{q^2 N}{n_N} \V[\hproj{0,1}{\emptyset,\{1\}}h] \right) \\
        &= \frac{p^2}{c} \V[\psi^{1,0}_{(\{1\},\emptyset)}h] + \frac{q^2}{1-c} \V[\psi^{0,1}_{(\emptyset,\{1\})}h].
    \end{split}
    \end{equation*}
\end{remark}

\begin{remark}
    It is worth noting that, if $V^h = 0$, this convergence result remains valid, despite the theorem excluding this case. In such a scenario, the limit would be trivial: 
    \begin{equation*}
        \sqrt{N}(U^h_{N}-U^h_\infty) \xrightarrow[N \rightarrow \infty]{\mathbb{P}} 0.
    \end{equation*} 
    This case is referred to as \textit{degeneracy}, and we say that $U^h_N$ is a \textit{degenerate} $U$-statistic. From the orthogonal decomposition of the $U$-statistic~\eqref{eq:decomposition_ustat}, each component $P^{r,c}_N$ contributes to the limit distribution with a magnitude determined by $(r, c)$. Lemma~\ref{lem:asymptotic_normality_of_main_terms} establishes that the contributions of the two dominant terms are Gaussian. Degeneracy occurs when these components vanish, implying that the shape of the limit distribution is determined by higher-order terms, which requires a larger scaling than $\sqrt{N}$ to observe. In this paper, we only consider \textit{non-degenerate} $U$-statistics.
\end{remark}

\begin{remark}
    The asymptotic distribution of the $U$-statistic is determined by the first-order projections of the form $\hproj{1,0}{\{i\},\emptyset}h$ and $\hproj{1,0}{\emptyset, \{j\}}h$, through their separate contributions given by Lemma~\ref{lem:asymptotic_normality_of_main_terms}. These projections depend on the AHK latent variables $(\xi_i)$ and $(\eta_j)$, which correspond to the row and column effects in the RCE matrix. Interestingly, the entry-specific AHK variables $(\zeta_{ij})$ do not contribute to the asymptotic distribution. To build intuition for this observation, consider the representation of the $U$-statistic in terms of AHK variables, derived from~\eqref{eq:ahk_function}:
    \begin{equation*}
        U_N^h = \binom{m_N}{r}^{-1}\binom{n_N}{c}^{-1} \sum_{\substack{\mathbf{i} \in \mathcal{P}_p(\llbracket m_N \rrbracket)\\ \mathbf{j} \in \mathcal{P}_q(\llbracket n_N \rrbracket)}} h_\phi((\xi_{i})_{i \in \mathbf{i}}; (\eta_{j})_{j \in \mathbf{j}}; (\zeta_{ij})_{\substack{i \in \mathbf{i} \\ j \in \mathbf{j}}}).
    \end{equation*}
    In this summation over $\mathbf{i}$ and $\mathbf{j}$, the entry-specific variable $\zeta_{ij}$ appears only in terms where both $i \in \mathbf{i}$ and $j \in \mathbf{j}$. In contrast, the row-specific variable $(\xi_i)$ contributes to all terms where $i \in \mathbf{i}$, regardless of the choice of $j \in \mathbf{j}$. Since the number of such terms is $O(N)$ larger, the influence of $\xi_i$ on $U^h_N$ dominates that of $\zeta_{ij}$. An analogous reasoning applies to column-specific variables $(\eta_j)$. Thus, in the asymptotic regime, the contribution of entry-specific fluctuations captured by $(\zeta_{ij})$ is of higher order and is negligible compared to the first-order row-column variability captured by $(\xi_i)$ and $(\eta_j)$.
\end{remark}

\begin{remark}
    Theorem~\ref{th:asymptotic_normality} can be extended to cases where the numbers of rows and columns grow at different rates, meaning that $\rho = 0$ or $\rho = 1$. For instance, consider the case $\rho = 0$, which implies $m_N = o(n_N)$. In this scenario, we can rescale $U^h_N$ by $\sqrt{m_N}$ instead of $\sqrt{N}$ in the proof of Theorem~\ref{th:asymptotic_normality}, leading to 
    \begin{equation*}
        \sqrt{m_N}(U^h_N - U^h_\infty) = \frac{p}{\sqrt{m_N}} \sum_{i = 1}^{m_N} \hproj{1,0}{\{i\},\emptyset}h + B_N,
    \end{equation*}
    where Lemma~\ref{lem:asymptotic_normality_of_main_terms} remains applicable and, analogous to Lemma~\ref{lem:clt_residue}, it can be shown that the sequence $B_N \xrightarrow[N \rightarrow \infty]{\mathbb{P}} 0$. Therefore, we obtain
    \begin{equation*}
        \sqrt{m_N}(U^h_N - U^h_\infty) \xrightarrow[N \rightarrow \infty]{\mathcal{D}} \mathcal{N}(0, V'),
    \end{equation*}
    where $V' = p^2 v_h^{1,0}$.
\end{remark}

In Appendix~\ref{app:asymptotic_normality}, we also demonstrate the asymptotic normality of vectors and functions of $U$-statistics on RCE matrices, as stated by Corollaries~\ref{cor:joint_asymptotic_normality} and~\ref{cor:delta_method}. Specifically, the asymptotic normality of functions of $U$-statistics is very useful to perform statistical inference on bipartite network data. We give more details in Section~\ref{sec:examples}.

\section{Estimation of the asymptotic variance}
\label{sec:variance_estimator}
\label{SEC:VARIANCE_ESTIMATOR}

Theorem~\ref{th:asymptotic_normality} shows the asymptotic normality of RCE submatrix $U$-statistics. To perform statistical inference using these $U$-statistics, one needs to estimate their variances. We see that 
\begin{equation*}
    \V[U^h_N] = \left[\binom{m}{p} \binom{n}{q}\right]^{-2} \sum_{\substack{\mathbf{i} \in \mathcal{P}_p(\llbracket m \rrbracket)\\\mathbf{j} \in \mathcal{P}_q(\llbracket n \rrbracket)}} \sum_{\substack{\mathbf{i}' \in \mathcal{P}_p(\llbracket m \rrbracket)\\\mathbf{j}' \in \mathcal{P}_q(\llbracket n \rrbracket)}} \Cov\left(h(Y_{\mathbf{i},\mathbf{j}}), h(Y_{\mathbf{i}',\mathbf{j}'})\right).
\end{equation*}
By exchangeability, the covariance between the kernels $h(Y_{\mathbf{i},\mathbf{j}})$ and $h(Y_{\mathbf{i}',\mathbf{j}'})$ only depends on the number of row and column indices they share. Denote $\gamma^{r,c}_h := \Cov\left(h(Y_{\mathbf{i},\mathbf{j}}), h(Y_{\mathbf{i}',\mathbf{j}'})\right)$ where $Y_{\mathbf{i},\mathbf{j}}$ and $Y_{\mathbf{i}',\mathbf{j}'}$ share $r$ row indices and $c$ column indices: $\Card(\mathbf{i} \cap \mathbf{i}') = r$ and $\Card(\mathbf{j} \cap \mathbf{j}') = c$. We get
\begin{equation*}
    \V[U^h_N] = \left[\binom{m}{p} \binom{n}{q}\right]^{-1} \sum_{(0,0) \le (r,c) \le (p,q)} \binom{p}{r} \binom{q}{c} \binom{m-p}{p-r} \binom{n-q}{q-c} \gamma^{r,c}_h.
\end{equation*}
Each $\gamma^{r,c}_h$, $1\le r\le p, 1 \le c \le q$, can be estimated using empirical covariance estimators, between kernel terms that share $r$ rows and $c$ columns and in particular. This leads to an unbiased estimator of $\V[U^h_N]$ similar to the one discussed by \cite{schucany1989small} for $U$-statistics of one-dimensional i.i.d. arrays. However, the estimation of these covariances is computationally intensive and the estimators can take negative values, which can lead to a negative variance estimation.

One approach is to estimate the asymptotic variance $V^h$. The asymptotic variance formula given by Theorem \ref{th:asymptotic_normality} is $V^h = \frac{p^2}{\rho} v^{1,0}_h + \frac{q^2}{1-\rho} v^{0,1}_h$. Remark that 
$v^{1,0}_h = \gamma^{1,0}_h$ and $v^{0,1}_h = \gamma^{0,1}_h$. It is often tedious to analytically calculate $V^h$, especially as it depends on $h$ and the distribution of $Y$, see for example Section 3 of~\cite{leminh2023ustatistics}. 

We present here a kernel-free and model-free estimator of $V^h$, taking advantage of the Hoeffding decomposition. Indeed, we will first define estimators for the conditional expectations $(\psi^{1,0}_{(\{i\}, \emptyset)} h)_{1 \le i \le m}$ and $(\psi^{0,1}_{(\emptyset, \{j\})} h)_{1 \le j \le n}$. Then, using the fact that $v^{1,0}_h = \V[\psi^{1,0}_{(\{1\}, \emptyset)} h]$ and $v^{0,1}_h = \V[\psi^{0,1}_{(\emptyset, \{1\})} h]$, we can derive a positive estimator for $V^h$. Afterward, we explain how to generalize it to estimate the variance of functions of $U$-statistics.

\subsection{Notations}
\label{subsec:notations}

First, we introduce further notations and a helpful lemma for this section. For some $N > 0$, the size of the overall RCE matrix is $m_N \times n_N$. $i$ being a row index means that $1 \le i \le m_N$ and $j$ being a column index means that $1 \le j \le n_N$. We denote
\begin{equation*}
    X_{\mathbf{i}, \mathbf{j}} := h(Y_{\mathbf{i},\mathbf{j}}).
\end{equation*} 
For $N$ such that $m_N \ge p$ and $n_N \ge q$, we further denote 
\begin{equation*}
    \mathcal{S}^{p,q}_{N} := \left\{ (\mathbf{i}, \mathbf{j}): \mathbf{i} \in \mathcal{P}_p(\llbracket m_N \rrbracket), \mathbf{j} \in \mathcal{P}_q(\llbracket n_N \rrbracket) \right\},
\end{equation*}
so that the set of the kernels taken on all the $p \times q$ submatrices can be written as $(X_{\mathbf{i}, \mathbf{j}})_{(\mathbf{i}, \mathbf{j}) \in \mathcal{S}^{p,q}_{N}}$.

Let $\mathbf{\underline{i}}$ be a set of row indices of size $\underline{p}$ such that $0 \le \underline{p} \le p$ and $\mathbf{\underline{j}}$ a set of column indices of size $\underline{q}$ such that $0 \le \underline{q} \le q$. The subset of $\mathcal{S}^{p,q}_{N}$ where $\mathbf{\underline{i}}$ is included in the row indices and $\mathbf{\underline{j}}$ is included in the column indices is denoted
\begin{equation*}
    \mathcal{S}^{p,q}_{N , (\mathbf{\underline{i}}, \mathbf{\underline{j}}) } := \left\{ (\mathbf{i}, \mathbf{j}) \in \mathcal{S}^{p,q}_{N} : \mathbf{\underline{i}} \subset \mathbf{i}, \mathbf{\underline{j}} \subset \mathbf{j} \right\}.
\end{equation*}
For example, the set of the kernels taken on all the $p \times q$ submatrices containing the columns $1$ and $2$ can be written $(X_{\mathbf{i}, \mathbf{j}})_{(\mathbf{i}, \mathbf{j}) \in \mathcal{S}^{p,q}_{N, (\emptyset, \{1,2\})}}$.

\subsection{Estimation of the conditional expectations}

In this section, for all $i \in \llbracket m_N \rrbracket$ and $j \in \llbracket n_N \rrbracket$, we define estimators for $\psi^{1,0}_{(\{i\}, \emptyset)} h = \E[X_{\llbracket p \rrbracket, \llbracket q \rrbracket} \mid \xi_i]$ and $\psi^{0,1}_{(\emptyset, \{j\})} h = \E[X_{\llbracket p \rrbracket, \llbracket q \rrbracket} \mid \eta_j]$, where $\xi_i$ and $\eta_j$ have been defined in Section~\ref{sec:ahk}.

Let $\widehat{\mu}^{h,(i)}_N$ be the average of the kernel function applied on the $p \times q$ submatrices containing the row $i$. Symmetrically, let $\widehat{\nu}^{h,(j)}_N$ be the average of the kernel function applied on the $p \times q$ submatrices containing the column $j$. This means
\begin{equation}
    \widehat{\mu}^{h,(i)}_N := \binom{m_N-1}{p-1}^{-1} \binom{n_N}{q}^{-1} \sum_{(\mathbf{i}, \mathbf{j}) \in \mathcal{S}^{p,q}_{N , (\{i\}, \emptyset) }} X_{\mathbf{i}, \mathbf{j}},
\label{eq:mu}
\end{equation}
and
\begin{equation}
    \widehat{\nu}^{h,(j)}_N := \binom{m_N}{p}^{-1} \binom{n_N-1}{q-1}^{-1} \sum_{(\mathbf{i}, \mathbf{j}) \in \mathcal{S}^{p,q}_{N , (\emptyset, \{j\}) }} X_{\mathbf{i}, \mathbf{j}}.
\label{eq:nu}
\end{equation}
The following propositions, proven in Appendix~\ref{app:variance_estimator}, establish guarantees for these estimators.

\begin{proposition}
    If $Y$ is an RCE matrix, then $\widehat{\mu}^{h,(i)}_N$ and $\widehat{\nu}^{h,(i)}_N$  are  both conditionally unbiased $\psi^{1,0}_{(\{i\}, \emptyset)} h$ and $\psi^{0,1}_{(\emptyset, \{j\})} h$, i.e. we have for all $N \in \mathbb{N}$:
    \begin{itemize}
        \item $\E[\widehat{\mu}^{h,(i)}_N \mid \xi_i] = \psi^{1,0}_{(\{i\}, \emptyset)} h$,
        \item $\E[\widehat{\nu}^{h,(j)}_N \mid \eta_j] = \psi^{0,1}_{(\emptyset, \{j\})} h$.
    \end{itemize}
    \label{prop:cond_exp_unbiased}
\end{proposition}

\begin{proposition}
    If $Y$ is an RCE matrix, then
    \begin{itemize}
        \item $\widehat{\mu}^{h,(i)}_N \xrightarrow[N \rightarrow \infty]{a.s., L_1} \psi^{1,0}_{(\{i\}, \emptyset)} h$,
        \item $\widehat{\nu}^{h,(j)}_N \xrightarrow[N \rightarrow \infty]{a.s., L_1} \psi^{0,1}_{(\emptyset, \{j\})} h$.
    \end{itemize}
    As a consequence, $\widehat{\mu}^{h,(i)}_N$ and $\widehat{\nu}^{h,(j)}_N$ are consistent estimators for $\psi^{1,0}_{(\{i\}, \emptyset)} h$ and $\psi^{0,1}_{(\emptyset, \{j\})} h$.
    \label{prop:cond_exp_consistent}
\end{proposition}

\begin{remark}
    Despite the AHK variables $(\xi_i)_{i \ge 1}$ and $(\eta_j)_{j \ge 1}$ being unobserved, Propositions~\ref{prop:cond_exp_unbiased} and~\ref{prop:cond_exp_consistent} ensure that the conditional expectations $\psi^{1,0}_{(\{i\}, \emptyset)} h$ and $\psi^{0,1}_{(\emptyset, \{ j \})} h$ can be estimated using the estimators $\widehat{\mu}^{h,(i)}_N$ and $\widehat{\nu}^{h,(j)}_N$.
\end{remark}

\subsection{Estimation of the asymptotic variance}
\label{sub:variance_estimator}

Using the estimators for the conditional expectations $\psi^{1,0}_{(\{i\}, \emptyset)} h$ and $\psi^{0,1}_{(\emptyset, \{j\})} h$ that we have defined in the previous section, we suggest the following estimators for $v^{1,0}_h = \V[\psi^{1,0}_{(\{1\}, \emptyset)} h]$ and $v^{0,1}_h = \V[\psi^{0,1}_{(\emptyset, \{j\})} h]$:

\begin{equation}
    \widehat{v}^{h;1,0}_N = \binom{m_N}{2}^{-1} \sum_{1 \le i_1 < i_2 \le m_N} \frac{(\widehat{\mu}^{h,(i_1)}_N - \widehat{\mu}^{h,(i_2)}_N)^2}{2},
    \label{eq:est_v10}
\end{equation}

\begin{equation}
    \widehat{v}^{h;0,1}_N = \binom{n_N}{2}^{-1} \sum_{1 \le j_1 < j_2 \le n_N} \frac{(\widehat{\nu}^{h,(j_1)}_N - \widehat{\nu}^{h,(j_2)}_N)^2}{2}.
    \label{eq:est_v01}
\end{equation}
Then, an estimator for $V^h$ is
\begin{equation*}
    \widehat{V}^h_N := \frac{p^2}{\rho}\widehat{v}^{h;1,0}_N + \frac{q^2}{1-\rho} \widehat{v}^{h;0,1}_N.
\end{equation*}

The following theorem shows that $\widehat{V}^h_N$ is a consistent estimator for $V^h$. The proof relies on the fact that $\widehat{v}^{h;1,0}_N$ and $\widehat{v}^{h;0,1}_N$ are both asymptotically unbiased and have vanishing variances. These properties are established by Propositions~\ref{prop:est_var_unbias} and~\ref{prop:est_var_variance} in Appendix~\ref{app:variance_estimator}.

\begin{theorem}
We have $\widehat{v}^{h;1,0}_N \xrightarrow[N \rightarrow \infty]{\mathbb{P}} v^{1,0}_h$ and $\widehat{v}^{h;0,1}_N \xrightarrow[N \rightarrow \infty]{\mathbb{P}} v^{0,1}_h$. As a consequence, $\widehat{V}^h_N  \xrightarrow[N \rightarrow \infty]{\mathbb{P}} V^h$.
\label{th:consistency_est_var}
\end{theorem}

\begin{proof}[Proof of Theorem~\ref{th:consistency_est_var}]
    For some $\epsilon > 0$, it follows from Proposition~\ref{prop:est_var_unbias} that for large enough values of $N$, $\left\lvert\E[\widehat{v}^{h;1,0}_N] - v^{1,0}_h\right\rvert < \epsilon$. The triangular inequality and Chebyshev's inequality state that 
    \begin{equation*}
    \begin{split}
         \mathbb{P}\left(\left\lvert\widehat{v}^{h;1,0}_N - v^{1,0}_h\right\rvert > \epsilon\right) &\le  \mathbb{P}\left(\left\lvert\widehat{v}^{h;1,0}_N - \E[\widehat{v}^{h;1,0}_N]\right\rvert \ge \epsilon - \left\lvert\E[\widehat{v}^{h;1,0}_N] - v^{1,0}_h\right\rvert\right) \\
         &\le \frac{\V[\widehat{v}^{h;1,0}_N]}{\left(\epsilon -\left\lvert\E[\widehat{v}^{h;1,0}_N] - v^{1,0}_h\right\rvert\right)^2}.
    \end{split}
    \end{equation*}
    Applying Propositions~\ref{prop:est_var_unbias} and~\ref{prop:est_var_variance} to the right-hand side of the inequality ensures that $\mathbb{P}\left(\left\lvert\widehat{v}^{h;1,0}_N - v^{1,0}_h\right\rvert > \epsilon\right) \xrightarrow[N \rightarrow \infty]{} 0$ which concludes the proof.
\end{proof}

With Theorem~\ref{th:consistency_est_var}, it is possible to use $\widehat{V}^h_N$ for statistical inference tasks when plugged-in in place of $V^h$, an asymptotic normality result similar to Theorem~\ref{th:asymptotic_normality} holds, by Slutsky's theorem.

\begin{proposition}
    If $V^h > 0$, then
    \begin{equation*}
        \sqrt{\frac{N}{\widehat{V}^h_N}}\left(U^h_N - U^h_\infty\right) \xrightarrow[N \rightarrow \infty]{\mathcal{D}} \mathcal{N}(0, 1).
    \end{equation*}
    \label{prop:asymptotic_normality_plugin}
\end{proposition}

\subsection{Computation of the estimator}
\label{sub:variance_estimator_alternate}

In practice, the computation of the estimators $(\widehat{\mu}^{h,(i)}_N)_{1 \le i \le m_N}$ and $(\widehat{\nu}^{h,(j)}_N)_{1 \le j \le n_N}$ using their definition is pretty tedious, as they are sums of $O(N^{p+q-1})$ terms. The computation cost of $\widehat{V}^h_N$ is then $O(N^{p+q})$ which is of the same order as the computation of $U^h_N$ when naively applying the kernel function to all the $p \times q$ submatrices and averaging. However, computing $U^h_N$ in this manner is generally undesirable. In many problem-specific settings, more efficient methods exist for computing $U$-statistics.  Typically, for simple kernels, $U^h_N$ can be expressed in terms of matrix operations, which are optimized in most computing software (see Section~\ref{sub:simu_variance} for an example). In such cases, it is beneficial to rewrite the proposed variance estimators in the form of those $U$-statistics, to leverage these advantages.

\begin{proposition}
    An alternative form for $\widehat{v}^{h;1,0}_N$ and $\widehat{v}^{h;0,1}_N$ is given by
    \begin{equation}
        \widehat{v}^{h;1,0}_N = \frac{(m_N - p)^2}{p^2(m_N-1)} \sum_{i=1}^{m_N} \left(U^h_N - U_N^{h,(-i,\emptyset)}\right)^2
    \label{eq:est_v10_alternate}
    \end{equation}
    and
    \begin{equation}
        \widehat{v}^{h;0,1}_N = \frac{(n_N - q)^2}{q^2(n_N-1)} \sum_{j=1}^{n_N} \left(U^h_N - U_N^{h,(\emptyset,-j)}\right)^2,
    \label{eq:est_v01_alternate}
    \end{equation}
    where
    \begin{equation*}
        U_N^{h,(-i,\emptyset)} := \left[\binom{m_N-1}{p} \binom{n_N}{q}\right]^{-1} \sum_{\substack{\mathbf{i} \in \mathcal{P}_p(\llbracket m_N \rrbracket \backslash \{i\})\\ \mathbf{j} \in \mathcal{P}_q(\llbracket n_N \rrbracket)}} h(Y_{\mathbf{i},\mathbf{j}})
    \end{equation*}
    and 
    \begin{equation*}
        U_N^{h,(\emptyset,-j)} := \left[\binom{m_N}{p} \binom{n_N-1}{q}\right]^{-1} \sum_{\substack{\mathbf{i} \in \mathcal{P}_p(\llbracket m_N \rrbracket)\\ \mathbf{j} \in \mathcal{P}_q(\llbracket n_N \rrbracket \backslash \{j\})}} h(Y_{\mathbf{i},\mathbf{j}}).
    \end{equation*}
\end{proposition}

\begin{proof}
    The relations are obtained from the definition of $\widehat{v}^{h;1,0}_N$ and $\widehat{v}^{h;0,1}_N$ and noticing that 
    \begin{equation*}
        \binom{m_N-1}{p-1} \binom{n_N}{q} \widehat{\mu}^{h,(i)}_N = \binom{m_N}{p} \binom{n_N}{q} U^h_N - \binom{m_N-1}{p} \binom{n_N}{q} U_N^{h,(-i,\emptyset)}
    \end{equation*}
    and 
    \begin{equation*}
        \binom{m_N}{p} \binom{n_N-1}{q-1} \widehat{\nu}^{h,(j)}_N = \binom{m_N}{p} \binom{n_N}{q} U^h_N - \binom{m_N}{p} \binom{n_N-1}{q} U_N^{h,(\emptyset,-j)}.
    \end{equation*}
\end{proof}

\begin{remark}
    This is one alternative method to compute $\widehat{V}^h_N$, but not necessarily the optimal one. With this form, the computational cost of $\widehat{V}^h_N$ is $O(N \Gamma(N))$, assuming the cost of computing a $U$-statistic is $O(\Gamma(N))$, where $\Gamma$ is an  increasing function. This can be outperformed by the naive method in some situations, for example when $N^{p + q} = O(\Gamma(N))$.
\end{remark}

\begin{remark}
    The alternative form of $\widehat{v}^{h;1,0}_N$ and $\widehat{v}^{h;0,1}_N$ is reminiscent of the jackknife estimator for the variance of $U$-statistics of one-dimensional arrays \cite{arvesen1969jackknifing}, but the two are well distinct. In the case where $Y$ is a one-dimensional array, the $U$-statistic associated to the kernel $h : \mathbb{R}^p \rightarrow \mathbb{R}$ is
    \begin{equation*}
        U^h_N = \binom{N}{p}^{-1} \sum_{\mathbf{i} \in \mathcal{P}_p( \llbracket N \rrbracket)} h(X_{\mathbf{i}}).
    \end{equation*}
    The jackknife estimator of the asymptotic variance of this $U$-statistic is 
    \begin{equation*}
        \widehat{V}^{h,J}_N = (N-1) \sum_{i=1}^N (U^h_N - U_N^{h,(-i)})^2,
    \end{equation*}
    where 
    \begin{equation*}
         U_N^{h,(-i)} = \binom{N}{p}^{-1} \sum_{\mathbf{i} \in \mathcal{P}_p( \llbracket N \rrbracket \backslash \{i\})} h(X_{\mathbf{i}}).
    \end{equation*}
    In fact, our estimator is closer to Sen's estimator of the asymptotic variance of \cite{sen1960some, sen1977some} which depends on the kernel size and is related to the jackknife estimator $\widehat{V}^{h,S}_N = (N-k)^2/(N-1)^2 \widehat{V}^{h,J}_N$.
    However, it is unclear how these estimators could be translated in a two-dimensional setup where $Y$ is a matrix instead of a vector, especially how to define the analog of $U_N^{h,(-i)}$.  
\end{remark}

\subsection{Extension to functions of $U$-statistics}

In the case of a function of $U$-statistics $g(U^{h_1}_N, ..., U^{h_D}_N)$, the asymptotic normality result of Corollary~\ref{cor:delta_method} applies and the asymptotic variance to be estimated is 
\begin{equation*}
    V^\delta = \nabla g(\theta)^T \Sigma^{h_1,...,h_D} \nabla g(\theta) \neq 0,
\end{equation*}
where $\theta = (U^{h_1}_\infty, ..., U^{h_D}_\infty)$. Similar to the estimator for the asymptotic variance of $V^h$, we suggest an estimator for the covariance matrix $\Sigma^{h_1,...,h_D} = \left(C^{h_k, h_\ell} \right)_{1 \le k,\ell \le D}$.

For each kernel $h_k$, $1 \le k \le D$, let $\widehat{\mu}^{h_k,(i)}_N$ and $\widehat{\nu}^{h_k,(j)}_N$ be the respective estimators of the conditional expectations $\psi^{1,0}_{(\{i\}, \emptyset)} h_k$ and $\psi^{0,1}_{(\emptyset, \{j\})} h_k$, as defined in equations~\eqref{eq:mu} and~\eqref{eq:nu}. Define
\begin{equation*}
    \widehat{c}^{h_k, h_\ell; 1,0}_N := \binom{m_N}{2}^{-1} \sum_{1 \le i_1 < i_2 \le m_N} \frac{(\widehat{\mu}^{h_k,(i_1)}_N - \widehat{\mu}^{h_k,(i_2)}_N)(\widehat{\mu}^{h_\ell,(i_1)}_N - \widehat{\mu}^{h_\ell,(i_2)}_N)}{2}
\end{equation*}
and
\begin{equation*}
    \widehat{c}^{h_k, h_\ell; 0,1}_N := \binom{n_N}{2}^{-1} \sum_{1 \le j_1 < j_2 \le n_N} \frac{(\widehat{\nu}^{h_k,(j_1)}_N - \widehat{\nu}^{h_k,(j_2)}_N)(\widehat{\nu}^{h_\ell,(j_1)}_N - \widehat{\nu}^{h_\ell,(j_2)}_N)}{2}.
\end{equation*}
Then, for two kernels $h_k$ and $h_\ell$, 
\begin{equation*}
    \widehat{C}^{h_k, h_\ell}_N := \frac{p^2}{\rho}\widehat{c}^{h_k, h_\ell; 1,0}_N + \frac{q^2}{1-\rho} \widehat{c}^{h_k, h_\ell; 0,1}_N,
\end{equation*}
is an estimator of the asymptotic covariance term $C^{h_k, h_\ell}$. With a similar proof as for Theorem~\ref{th:consistency_est_var}, the following theorem ensures the consistency of this estimator.
\begin{theorem}
    For two kernel functions $h_k$ and $h_\ell$, we have $\widehat{c}^{h_k, h_\ell; 1,0}_N \xrightarrow[N \rightarrow \infty]{\mathbb{P}} c^{1,0}_{h_k, h_\ell}$ and $\widehat{c}^{h_k, h_\ell; 0,1}_N \xrightarrow[N \rightarrow \infty]{\mathbb{P}} c^{0,1}_{h_k, h_\ell}$. As a consequence, $\widehat{C}^{h_k, h_\ell}_N \xrightarrow[N \rightarrow \infty]{\mathbb{P}} C^{h_k, h_\ell}$.
    \label{th:consistency_est_cov}
\end{theorem}

Thus, for linearly independent kernel functions $(h_1, ..., h_D)$, the entries of the matrix $\widehat{\Sigma}^{h_1,...,h_D}_N := \left(\widehat{C}^{h_k, h_\ell}_N\right)_{1 \le k,\ell \le D}$ converge to the entries of $\Sigma^{h_1,...,h_D}$. Set 
\begin{equation*}
    \widehat{V}^\delta_N := \nabla g(U^{h_1}_N, ..., U^{h_D}_N)^T \widehat{\Sigma}^{h_1,...,h_D}_N \nabla g(U^{h_1}_N, ..., U^{h_D}_N),
\end{equation*}
then we have $\widehat{V}^\delta_N \xrightarrow[N \rightarrow \infty]{\mathbb{P}} V^\delta$. Finally, the ``studentized'' statistic obtained by plug-in also stands.
\begin{proposition}
    If $V^\delta > 0$, then
    \begin{equation*}
        \sqrt{\frac{N}{\widehat{V}^\delta_N}}\left(g(U^{h_1}_N, ..., U^{h_D}_N) - g(\theta)\right) \xrightarrow[N \rightarrow \infty]{\mathcal{D}} \mathcal{N}(0, 1).
    \end{equation*}
    \label{prop:delta_method_plugin}
\end{proposition}

\section{Methodology for bipartite network analysis}
\label{sec:examples}
\label{SEC:EXAMPLES}

The previous section introduced a generic and consistent estimator for the asymptotic variance of $U$-statistics on RCE matrices. This estimator can be used to develop a methodology for statistical inference on bipartite network data. In particular, the RCE framework, combined with the versatility provided by Proposition~\ref{prop:delta_method_plugin} in selecting network statistics, supports a broad range of network analysis tasks through model parameter estimation and hypothesis testing. In this section, we first examine the network comparison problem as a statistical inference task. We then illustrate the flexibility of this framework by presenting examples of RCE models and $U$-statistic-based network statistics applicable to various network analysis problems.

\subsection{Example of inference task: network comparison} \label{sec:netComp}

Network comparison has a long history in network analysis literature~\cite{emmert2016fifty, tantardini2019comparing}. Few comparison methods use random network models~\cite{asta2014geometric, maugis2020testing, leminh2023ustatistics}. However, there are several advantages to using model-based approaches. Indeed, random models define a probability distribution of networks that can be used to derive statistical guarantees. Also, models can be used to control the sources of heterogeneity in networks. The comparison can be made with respect to a quantity of interest, which makes it easier to interpret. Now, we show how to define a test statistic for model-based network comparison with our framework. Naturally, $U$-statistics define network statistics characterizing a network. They can be used to analyze a single network, but their use is easily extended to compare different networks.

Let $Y^A$ and $Y^B$ be two independent networks of respective sizes $m^A_N \times n^A_N$ and $m^B_N \times n^B_N$. Define a network quantity of interest $\theta$. These two networks are generated by two models, leading to different values $\theta^A$ and $\theta^B$ of this quantity for the two models. Let $\widehat{\theta}_N$ be an estimator for this quantity of interest. If $\widehat{\theta}_N$ is a $U$-statistic or a function of $U$-statistics, then as previously, we have both 
\begin{equation*}
    \sqrt{N}(\widehat{\theta}_N(Y^A) - \theta^A) \xrightarrow[N \rightarrow \infty]{\mathcal{D}} \mathcal{N}(0, V^A) 
\end{equation*}
and
\begin{equation*}
    \sqrt{N}(\widehat{\theta}_N(Y^B) - \theta^B) \xrightarrow[N \rightarrow \infty]{\mathcal{D}} \mathcal{N}(0, V^B),
\end{equation*}
where $V^A$ and $V^B$ are the asymptotic variances.

To compare $\theta^A$ and $\theta^B$, we can confront the test hypotheses $\mathcal{H}_0: \theta^A = \theta^B$ and $\mathcal{H}_1: \theta^A \neq \theta^B$. Because $Y^A$ and $Y^B$ are independent, the previous convergence results give 
\begin{equation*}
    \sqrt{\frac{N}{V^A + V^B}} \bigg(\widehat{\delta}_N(Y^A,Y^B) - (\theta^A - \theta^B)\bigg) \xrightarrow[N \rightarrow \infty]{\mathcal{D}} \mathcal{N}(0, 1),
\end{equation*}
where $\widehat{\delta}_N(Y^A,Y^B) := \widehat{\theta}_N(Y^A) - \widehat{\theta}_N(Y^B)$. Therefore, using the consistent estimators $\widehat{V}^A_N$ and $\widehat{V}^B_N$ of the asymptotic variances, the test statistic
\begin{equation*}
    Z_N(Y^A, Y^B) := \sqrt{\frac{N}{\widehat{V}^A_N + \widehat{V}^B_N}} \widehat{\delta}_N(Y^A,Y^B)
\end{equation*}
admits the convergence result
\begin{equation*}
    Z_N(Y^A, Y^B) - \sqrt{\frac{N}{\widehat{V}^A_N + \widehat{V}^B_N}} (\theta^A - \theta^B) \xrightarrow[N \rightarrow \infty]{\mathcal{D}}  \mathcal{N}(0,1).
\end{equation*}
Under $\mathcal{H}_0$, we have $\theta^A - \theta^B = 0$, so $Z_N(Y^A, Y^B)$ converges in distribution to a standard Gaussian variable. 

\subsection{Examples of RCE models for networks}

\paragraph{Bipartite Expected Degree Distribution model.} The Bipartite Expected Degree Distribution (BEDD) model, suggested by~\cite{ouadah2022motif}, is a binary graph model characterized by two distributions from which the row and column-nodes draw a weight. The probability of a connection between two nodes is fully determined by the corresponding row and column weight distributions. The distribution of a graph following a BEDD model can be written using latent variables $(\xi_i)_{i \ge 1}$ and $(\eta_j)_{j \ge 1}$ corresponding to the row and column-nodes of the graph:
\begin{equation}
\begin{split}
    \xi_i, \eta_j &\overset{\text{i.i.d.}}{\sim} \mathcal{U}[0,1] \\ 
    Y_{ij}~ \mid ~\xi_i, \eta_j &\sim \mathcal{B}(\lambda f(\xi_i) g(\eta_j)).
\end{split}
\label{eq:wbedd}
\end{equation} 
where
\begin{itemize}
    \item $f$ and $g$ are positive, càdlàg, nondecreasing, bounded, and normalized ($\int f = \int g = 1$) real functions $[0,1] \rightarrow \mathbb{R}_+$,
    \item $\lambda$ is a positive real number such that $\lambda \le \lVert f \rVert^{-1}_\infty \lVert g \rVert^{-1}_\infty$, 
    \item $\mathcal{B}$ is the Bernoulli distribution. 
\end{itemize}
The BEDD model is an RCE version of the Expected Degree Sequence model of~\cite{chung2002average} but in the BEDD, the row weights $f(\xi_i)$ and column weights $g(\eta_j)$ are exchangeable. Indeed, $f$ and $g$ characterize the weight distributions of the row and column-nodes whereas the weights are fixed in the Expected Degree Sequence model. The triplet $(\lambda, f, g)$ is called the BEDD parameters. 

\paragraph{Latent Block model.} The Latent Block model (LBM)~\cite{govaert2003clustering} is a binary graph model characterized by a partition of row and column-nodes in several groups. It can be considered as a bipartite extension of the Stochastic Block model~\cite{nowicki2001estimation}. The probability of interaction between two nodes is fully determined by the groups to which they belong. All the nodes have the same probability of belonging to each group. The distribution of an LBM is most commonly written using independent latent variables for the node attribution in a group $(Z_i)_{i \ge 1}$ and $(W_j)_{j \ge 1}$ corresponding to the row and column-nodes of the graph:
\begin{equation}
\begin{split}
    &Z_i \overset{\text{i.i.d.}}{\sim} \mathcal{M}(1;\boldsymbol{\alpha}) \\ 
    &W_j \overset{\text{i.i.d.}}{\sim} \mathcal{M}(1;\boldsymbol{\beta}) \\ 
    &Y_{ij} \mid Z_i=k, W_j=\ell \sim \mathcal{B}(\pi_{k\ell}),
\end{split}
\label{eq:lbm}
\end{equation} 
where $\boldsymbol{\alpha} = (\alpha_1,...,\alpha_K)$ and $\boldsymbol{\beta} = (\beta_1,...,\beta_L)$ are the probability vectors of the rows and the columns and $\boldsymbol{\pi} = (\pi_{k \ell})_{1 \le k \le K, 1 \le \ell \le L} \in [0,1]^{KL}$ is a matrix of probabilities. The LBM is an RCE model since the group attribution variables of the nodes are exchangeable.

\paragraph{W-graph model.} Let $w$ be a function of $[0,1]^2 \rightarrow [0,1]$. The W-graph model associated to $w$ is defined by
\begin{equation}
  \begin{split}
    &\xi_i, \eta_j \overset{\text{i.i.d.}}{\sim} \mathcal{U}[0,1] \\
    &Y_{ij} \mid \xi_i, \eta_j \sim \mathcal{B}\left(w(\xi_i, \eta_j)\right),
  \end{split}
  \label{eq:graphon}
\end{equation}
$w$ is sometimes referred to as a graphon. For identification reasons, we assume $\int w(\cdot,\eta) d\eta$ and $\int w(\xi,\cdot) d\xi$ to be càdlàg, nondecreasing and bounded. Any RCE model can be written as a W-graph model~\cite{diaconis2008graph} so can the BEDD model and the LBM be expressed with a graphon. For the BEDD model, $w(\xi_i, \eta_j) = \lambda f(\xi_i) g(\eta_j)$ where $f$ and $g$ are those of Equation~\eqref{eq:wbedd}. For the LBM, $w(\xi_i,\eta_j) = \sum_{k,\ell} \pi_{k \ell} \mathds{1}\{ s(\xi_i) = k \} \mathds{1}\{ t(\eta_{j}) = \ell \}$ where $s(\xi_i) = 1 + \sum_{k = 1}^K \mathds{1}\{ \xi_i > \sum_{k'=1}^k \alpha_{k'}\}$, $t(\eta_j) = 1 + \sum_{\ell = 1}^L \mathds{1}\{ \eta_j > \sum_{\ell'=1}^\ell \beta_{\ell'} \}$ and $\boldsymbol{\alpha}$, $\boldsymbol{\beta}$ and $\boldsymbol{\pi}$ are those of Equation~\eqref{eq:lbm}.

\paragraph{Extension to weighted graphs.} All the models presented above are defined for binary graphs. However, one can extend it to weighted graphs by switching the Bernoulli distributions with another. In particular, the Poisson distribution is particularly suitable for count values in $\mathbb{N}$:
\begin{equation}
  \begin{split}
    &\xi_i, \eta_j \overset{\text{i.i.d.}}{\sim} \mathcal{U}[0,1] \\
    &Y_{ij} \mid \xi_i, \eta_j \sim \mathcal{P}\left(w(\xi_i, \eta_j)\right),
  \end{split}
  \label{eq:graphon_weighted}
\end{equation}
where $w : [0,1]^2 \rightarrow \mathbb{R}_+$ is the graphon of this model. Similarly, one can define a Poisson-BEDD model or Poisson-LBM. 

\subsection{Examples of kernel functions for network analysis}

We present three examples of statistics which are $U$-statistics or functions of $U$-statistics. Sometimes, a kernel function $h$ has a long expression, especially when it is the symmetric version of some simpler function $h^0$, as in~\eqref{eq:kernel_symmetric_version}. In this case, we introduce the kernels of interest with $h^0$ instead of $h$, but the $U$-statistics and asymptotic results always apply to the symmetric version $h$.

\paragraph{Motif counts.} 
\label{sub:example_motif}

Motifs are the name given to small subgraphs. Their occurrences in the complete network can be counted. Motif counts are useful statistics for random graphs as they provide information on the network local structure~\cite{shen2002network, kashtan2004topological}. Many random network models hinge on motif frequencies. In the Exponential Random Graph Model \cite{frank1986markov}, motif frequencies are sufficient statistics. The $dk$-random graph model \cite{orsini2015quantifying} also largely relies on motif frequencies. 

Their asymptotic properties are widely studied and a large number of studies use motif counts to perform statistical tests \cite{reinert2010random, bickel2011method, bhattacharyya2015subsampling, coulson2016poisson, gao2017testing, maugis2020testing, naulet2021bootstrap, ouadah2022motif}. Our framework is particularly well adapted to the use of motif counts for statistical tests because frequencies are $U$-statistics with kernel functions of the same size as the motifs. 

In many applications, motifs are viewed as elementary building blocks that can be specifically interpreted. This is the case in molecular biology \cite{shen2002network, prvzulj2004modeling, ali2014alignment}, neurology \cite{zhao2011synchronization}, sociology \cite{bearman2004chains}, evolution \cite{przytycka2006important} and ecology \cite{stouffer2007evidence, baker2015species}. In particular, Figure 3 of~\cite{simmons2019motifs} lists all the bipartite motifs consisting of from 2 to 6 nodes. For example, their motif 6 represents a $2 \times 2$ clique, where every node is connected to all others (Figure~\ref{fig:motif_6_10}). Their motif 14 represents a path between two column-nodes, passing through another column-node and two row-nodes (Figure~\ref{fig:motif_6_10}). The latter indicates an indirect interaction between the row-nodes, hinging on the middle row-node. \cite{lanuza2023non} found that motif 6 is over-represented in plant-pollinator interaction networks, while motif 14 is under-represented, compared to Erd\"{o}s-R\'{e}nyi graphs of the same density.

\begin{figure}[!tb]
\centering
\includegraphics[width=0.2\linewidth]{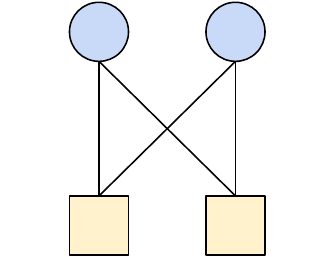}
\includegraphics[width=0.2\linewidth]{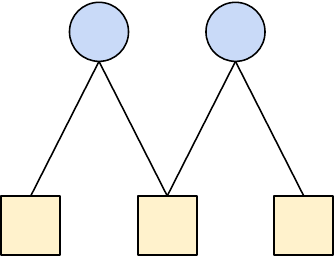}
\caption{Motif 6 and motif 14, counted by $U^{h_6}_N$ and $U^{h_{14}}_N$}
\label{fig:motif_6_10}
\end{figure}

These motifs can be counted with $U^{h_6}_N$ and $U^{h_{14}}_N$ the $U$-statistics associated to the kernels $h_6$ and $h_{14}$, that we introduce as the symmetric version of the following functions:
\begin{align}
    h^0_6(Y_{(i_1, i_2; j_1, j_2)}) &= Y_{i_1j_1}Y_{i_1j_2}Y_{i_2j_1}Y_{i_2j_2}, \nonumber \\
    h^0_{14}(Y_{(i_1, i_2; j_1, j_2, j_3)}) &= Y_{i_1j_1}Y_{i_1j_2}Y_{i_2j_2}Y_{i_2j_3}(1-Y_{i_2j_1})(1-Y_{i_1j_3}).
    \label{eq:motif_kernels}
\end{align}

\begin{remark}
    With this example, it is very apparent that introducing simple, non-symmetric kernel functions and then symmetrizing them is way simpler than introducing symmetric kernel functions. $h^0_6$ is already symmetric, so $h_6 = h^0_6$. However, $h^0_{14}$ is not symmetric and one can sum over all the permutations of the indices to make it symmetric using formula~\eqref{eq:kernel_symmetric_version}. Because of the automorphisms of motif 14, it involves only $6$ permutations instead of $2! 3! = 12$.
    \begin{equation*}
    \begin{split}
        h_{14}(Y_{(i_1, i_2; j_1, j_2, j_3)})
        &= \frac{1}{6} Y_{i_1j_1}Y_{i_1j_2}Y_{i_2j_2}Y_{i_2j_3}(1-Y_{i_2j_1})(1-Y_{i_1j_3}) \\
        &+ \frac{1}{6} Y_{i_1j_2}Y_{i_1j_3}Y_{i_2j_3}Y_{i_2j_1}(1-Y_{i_1j_1})(1-Y_{i_2j_2}) \\
        &+ \frac{1}{6} Y_{i_1j_3}Y_{i_1j_1}Y_{i_2j_1}Y_{i_2j_2}(1-Y_{i_1j_2})(1-Y_{i_2j_3})  \\
        &+ \frac{1}{6} Y_{i_2j_1}Y_{i_2j_2}Y_{i_1j_2}Y_{i_1j_3}(1-Y_{i_1j_1})(1-Y_{i_2j_3})  \\
        &+ \frac{1}{6} Y_{i_2j_2}Y_{i_2j_3}Y_{i_1j_3}Y_{i_1j_1}(1-Y_{i_2j_1})(1-Y_{i_1j_2}) \\
        &+ \frac{1}{6} Y_{i_2j_3}Y_{i_2j_1}Y_{i_1j_1}Y_{i_1j_2}(1-Y_{i_2j_2})(1-Y_{i_1j_3}).
    \end{split}
    \end{equation*}
\end{remark}

Using Proposition~\ref{prop:asymptotic_normality_plugin}, the following studentized statistics converge to a standard normal distribution 
\begin{equation*}
    Z^6_N := \sqrt{\frac{N}{\widehat{V}^{h_6}_N}}\left(U^{h_6}_N - U^{h_6}_\infty\right) \xrightarrow[N \rightarrow \infty]{\mathcal{D}} \mathcal{N}(0, 1),
\end{equation*}
\begin{equation*}
    Z^{14}_N := \sqrt{\frac{N}{\widehat{V}^{h_{14}}_N}}\left(U^{h_{14}}_N - U^{h_{14}}_\infty\right) \xrightarrow[N \rightarrow \infty]{\mathcal{D}} \mathcal{N}(0, 1).
\end{equation*}

\paragraph{Product graphon.}
\label{sub:example_graphon}

The $W$-graph model encompasses all the dissociated RCE models. One can make use of this model to compare the form of the graphon that has generated a graph to a known form. One interesting form is the product form, which corresponds to an absence of specific interaction between row-nodes and column-nodes. For example, one can normalize the graphon in the Poisson $W$-graph model described by~\eqref{eq:graphon_weighted} such that 
\begin{equation*}
    Y_{ij} \mid \xi_i, \eta_j \sim \mathcal{P}(\lambda \bar{w}(\xi_i, \eta_j)),
    \label{eq:wgraph_normalized}
\end{equation*}
where $\lambda > 0$ and $\iint \bar{w} = 1$. Call $\bar{w}$ a normalized graphon. Define $f$ and $g$ as the marginals of $\bar{w}$, i.e.
\begin{align}
    f & = \int \bar{w}(\cdot,\eta) d\eta, & 
    g & = \int \bar{w}(\xi,\cdot) d\xi.
    \label{eq:def_marginals}
\end{align}
A density-free dissimilarity measure between a graphon and its corresponding product form could be the quantity $d(\bar{w})$ defined as 
\begin{equation}
    d(\bar{w}) := \lvert \lvert \bar{w} - fg \rvert \rvert^2_2 = \iint \left(\bar{w}(\xi, \eta) - f(\xi) g(\eta) \right)^2 d\xi d\eta.
\end{equation}

Observe that $d(\bar{w}) = 0$ if and only if $\bar{w}$ is of product form, e.g. $\bar{w}(\xi,\eta) = f(\xi)g(\eta)$ almost everywhere in $[0,1]^2$. In this case, the $W$-graph model~\eqref{eq:wgraph_normalized} is a BEDD model, and $f$ and $g$ have the same role as in~\eqref{eq:wbedd}. 

We suggest an estimator for $d(\bar{w})$. Let $h_A, h_B, h_C$ and $h_D$ be kernel functions of respective size $2 \times 2$, $1 \times 2$, $2 \times 1$ and $1 \times 1$ defined as in Table~\ref{tab:kernels_1}. These kernel functions are linearly independent, so Corollary~\ref{cor:joint_asymptotic_normality} applies to the associated $U$-statistics $(U^{h_A}_N,U^{h_B}_N,U^{h_C}_N,U^{h_D}_N)$, and they are jointly asymptotically normal with asymptotic covariance matrix $\Sigma^{h_A,h_B,h_C,h_D}$.

\begin{table}[tb!]
\centering
\begin{tabular}{ l l l  }
 \toprule
 & Kernel & Expectation \\
 \midrule
 $h_A$ & $h^0_{A,1} - 2 h^0_{A,2}$ & $\lambda^3 \iint \bar{w}(\xi,\eta) (\bar{w}(\xi,\eta) - 2 f(\xi)g(\eta)) d\xi d\eta$ \\
 $h_{A,1}$ & $Y_{i_1j_1} (Y_{i_1j_1} - 1)  Y_{i_2j_2}$ & $\lambda^3 \iint \bar{w}(\xi,\eta)^2 d\xi d\eta$ \\
 $h_{A,2}$ & $Y_{i_1j_1} Y_{i_1j_2} Y_{i_2j_2}$ & $\lambda^3 \iint \bar{w}(\xi,\eta) f(\xi)g(\eta) d\xi d\eta$ \\
 $h_B$ & $Y_{i_1j_1} Y_{i_1j_2}$ & $\lambda^2 \int f(\xi)^2 d\xi$ \\
 $h_C$ & $Y_{i_1j_1} Y_{i_2j_1}$ & $\lambda^2 \int g(\eta)^2 d\eta$ \\
 $h_D$ & $\displaystyle{Y_{i_1j_1}}$ & $\lambda$ \\
 \bottomrule
\end{tabular}
\caption{Kernel functions (before symmetrization) used for the estimation of $d(\bar{w})$. The expectations are given by Lemma~\ref{lem:graphon_kernels} in Appendix.}
\label{tab:kernels_1}
\end{table}

We define the estimator $\widehat{d}_N$ of $d(\bar{w})$ as
\begin{equation*}
    \widehat{d}_N = t(U^{h_A}_N,U^{h_B}_N,U^{h_C}_N,U^{h_D}_N) := U^{h_A}_N/(U^{h_D}_N)^3 + U^{h_B}_N U^{h_C}_N/(U^{h_D}_N)^4.
\end{equation*}
We have 
\begin{equation*}
\begin{split}
    t(U^{h_A}_\infty,U^{h_B}_\infty,U^{h_C}_\infty,U^{h_D}_\infty) &= \iint \bar{w}(\xi,\eta) (\bar{w}(\xi,\eta) - 2 f(\xi)g(\eta)) d\xi d\eta \\
    &\quad + \int f(\xi)^2 d\xi \times \int g(\eta)^2 d\eta \\
    &= d(\bar{w}).
\end{split}
\end{equation*}
Then, Proposition~\ref{prop:delta_method_plugin} ensures that the studentized statistic converges to a standard normal distribution
\begin{equation*}
\begin{split}
    Z^d_N := \sqrt{\frac{N}{\widehat{V}^d_N}} \left(\widehat{d}_N - d(\bar{w})\right) \xrightarrow[N \rightarrow \infty]{\mathcal{D}} \mathcal{N}(0, 1),
\end{split}
\end{equation*}
where 
\begin{equation*}
    \widehat{V}^d_N = \nabla t(U^{h_A}_N,U^{h_B}_N,U^{h_C}_N,U^{h_D}_N)^T \widehat{\Sigma}^{h_A,h_B,h_C,h_D}_N \nabla t(U^{h_A}_N,U^{h_B}_N,U^{h_C}_N,U^{h_D}_N),
\end{equation*}
and $\widehat{\Sigma}^{h_A,h_B,h_C,h_D}_N$ is defined as in Proposition~\ref{prop:delta_method_plugin}.

\paragraph{Heterogeneity of the rows of a network.}
\label{sub:example_f2}

The degree distributions of a network hold significant information about its topology. For a binary network, denote $D_i := \sum_j Y_{ij}$ the degree of the $i$-th row, then $\E[D_i \mid \xi_i] = f(\xi_i) \E[D_i]$, where $f$ is given by~\eqref{eq:def_marginals}. Therefore, the marginals $f$ and $g$ account for the expected relative degree distributions of the binary network, which is known to characterize networks. In the BEDD model, these distributions (and the density of the network) fully characterize the model. 

Although the sum of the edge weights and the number of edges stemming from a node are equivalent for binary networks, for weighted networks, they are two different quantities. The sum of the edge weights stemming from a node is sometimes called its strength~\cite{barrat2004architecture}, which does not depend entirely on the number of edges. In a network with weighted edges, $f(\xi_i)$ corresponds to the expected relative strength of the $i$-th row-node, instead of its expected relative degree.

One way to characterize the degree/strength distributions of a network is to calculate their variance. In particular, the variance of the degree distribution of a network quantifies its heterogeneity (i.e. the unbalance between strongly interacting nodes and the others) and can be used as an index to characterize networks~\cite{snijders1981degree}. In our framework with random graph models, rather than directly study the empirical distribution of row degrees/strengths, one would like to retrieve information on $f$ and $g$, the distributions of the expected degrees/strengths specified by the model. For the BEDD model, $f$ and $g$ are directly given by the model. The variance of the row expected relative degree/strength distribution is $F_2 - 1$, where $F_2 := \int f^2(\xi) d\xi$.

$F_2$ can be estimated using the estimator $\widehat{F}_{2,N} := \kappa(U^{h_1}_N,U^{h_2}_N) = U^{h_1}_N/U^{h_2}_N$, using the $U$-statistics associated to the submatrix kernels functions $h_1$ and $h_2$ defined as 
\begin{equation*}
    h_1(Y_{(i;j_1,j_2)}) = Y_{ij_1} Y_{ij_2}
\end{equation*}
and 
\begin{equation*}
    h_2(Y_{(i_1,i_2;j_1,j_2)}) = \frac{1}{2} (Y_{i_1j_1} Y_{i_2j_2} + Y_{i_1j_2} Y_{i_2j_1}).
\end{equation*}
By Lemma~\ref{lem:f2_kernels} in Appendix, we have $U^{h_1}_\infty = \lambda^2 F_2$ and $U^{h_2}_\infty = \lambda^2$. Given that $\nabla \kappa(U^{h_1}_N, U^{h_2}_N) = (1/U^{h_2}_N, -U^{h_1}_N/(U^{h_2}_N)^2)$, the application of Proposition~\ref{prop:delta_method_plugin} yields 
\begin{equation}
    Z^{F_2}_N := \sqrt{\frac{N}{\widehat{V}^{F_2}_N}} \bigg(\widehat{F}_{2,N} - F_2\bigg) \xrightarrow[N \rightarrow \infty]{\mathcal{D}} \mathcal{N}(0, 1),
    \label{eq:convergence_vdt}
\end{equation}
where
\begin{equation*}
    \widehat{V}^{F_2}_N = \frac{1}{(U^{h_2}_N)^2} \widehat{V}^{h_1}_N - \frac{2 \widehat{F}_{2,N}}{(U^{h_2}_N)^2} \widehat{C}^{h_1,h_2}_N + \frac{(\widehat{F}_{2,N})^2}{(U^{h_2}_N)^2} \widehat{V}^{h_2}_N.
\end{equation*}

In a similar way, one can define the analogous estimator $\widehat{G}_{2,N}$ for $G_2 := \int g^2(\eta) d\eta$ and its associated variance estimator to quantify the heterogeneity of the columns of networks.

\section{Simulations}
\label{sec:simulations}

The methodology proposed in the previous section relies on Propositions~\ref{prop:asymptotic_normality_plugin} and~\ref{prop:delta_method_plugin}, which are asymptotic results. In this section, we illustrate the behavior of the statistics defined in the previous section on synthetic networks, investigating the effect of network size. For each example, we check the asymptotic normality of the studentized statistic for networks simulated under different configurations (different models, values of $N$ and $\rho$). We give the resulting Q-Q plots. We also examine the coverage probabilities of the confidence intervals built with the estimates or the standard deviations of the statistics, depending on whether they are fitted for estimation or statistical testing. Finally, we illustrate the computational efficiency of our variance estimator, compared to naive approaches.

\subsection{Motif counts}

The following simulation results illustrate the application to network motif counting developed in Section~\ref{sub:example_motif}.

\paragraph{Model I:} This graph model is an LBM with two row groups and two column groups of equal proportion. Using the notations of~\eqref{eq:lbm}, this means $\boldsymbol{\alpha} = \boldsymbol{\beta} = (0.5, 0.5)$. The probability matrix $\boldsymbol{\pi}$ has size $2 \times 2$. We set $\pi_{k\ell} = 0.5$ for all $1 \le k, \ell \le 2$ except $\pi_{11} = 0.95$.

Under Model I, for each value $N \in \{ 2^k : k \in \llbracket 5, 12 \rrbracket \}$ and $\rho \in \{ 1/8, 1/2 \}$, we simulated $K = 500$ networks of size $m_N \times n_N$ where $m_N = \lfloor \rho N \rfloor$ and $n_N = N - m_N$. The relative frequencies of motifs 6 and 10 of \cite{simmons2019motifs} are respectively given by $U^{h_6}_N$ and $U^{h_{14}}_N$ where $h_6$ and $h_{14}$ were defined by~\eqref{eq:motif_kernels}. 

The Q-Q plots of the studentized statistic associated to $h_6$ and $h_{14}$ are given in Figure~\ref{fig:motif6_qq} and~\ref{fig:motif14_qq}. For $\rho = 1/2$, we observe that the empirical distributions of both statistics converge and become close to a normal distribution as $N \gtrsim 128$. Figure~\ref{fig:motif_coverage} gives the respective coverage probabilities for $U^{h_6}_\infty$ and $U^{h_14}_\infty$ of the intervals $\left[U^{h_6}_N - \Phi(1-\frac{\alpha}{2})\sqrt{\frac{N}{\widehat{V}^{h_6}_N}}, U^{h_6}_N + \Phi(1-\frac{\alpha}{2})\sqrt{\frac{N}{\widehat{V}^{h_6}_N}}\right]$ and $\left[U^{h_{14}}_N - \Phi(1-\frac{\alpha}{2})\sqrt{\frac{N}{\widehat{V}^{h_{14}}_N}}, U^{h_{14}}_N + \Phi(1-\frac{\alpha}{2})\sqrt{\frac{N}{\widehat{V}^{h_{14}}_N}}\right]$ respectively, where $\Phi$ is the quantile function of the standard normal distribution. The coverage probabilities converge to $\alpha$ but with different speeds depending on the motif. We also observe that a larger number of nodes is needed to reach the target coverage probabilities with $\rho = 1/8$ (rectangular matrix) than $\rho = 1/2$ (square matrix).

\begin{figure}[!tb]
\includegraphics[width=0.99\linewidth]{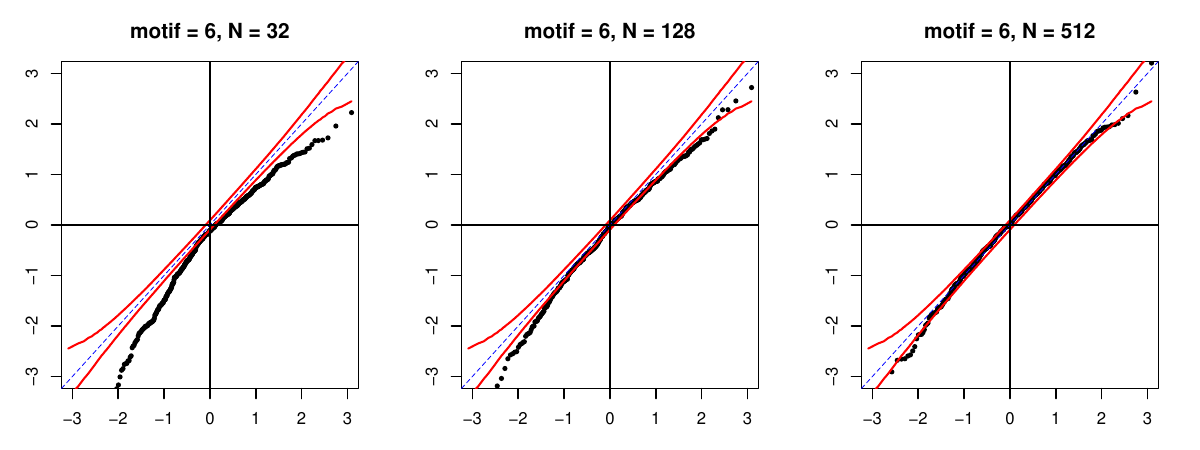}
\caption{Q-Q plots for $Z^6_N$ the studentized statistic associated with $U^{h_6}_N$, with $\rho = 0.5$.}
\label{fig:motif6_qq}
\end{figure}

\begin{figure}[!tb]
\includegraphics[width=0.99\linewidth]{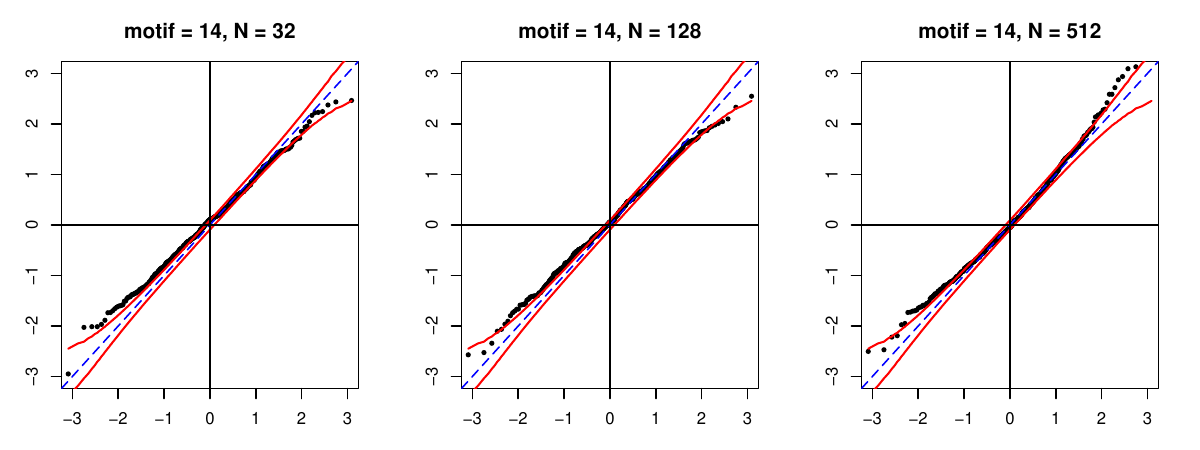}
\caption{Q-Q plots for $Z^{14}_N$ the studentized statistic associated with $U^{h_{14}}_N$, with $\rho = 0.5$.}
\label{fig:motif14_qq}
\end{figure}

\begin{figure}[!tb]
\includegraphics[width=0.99\linewidth]{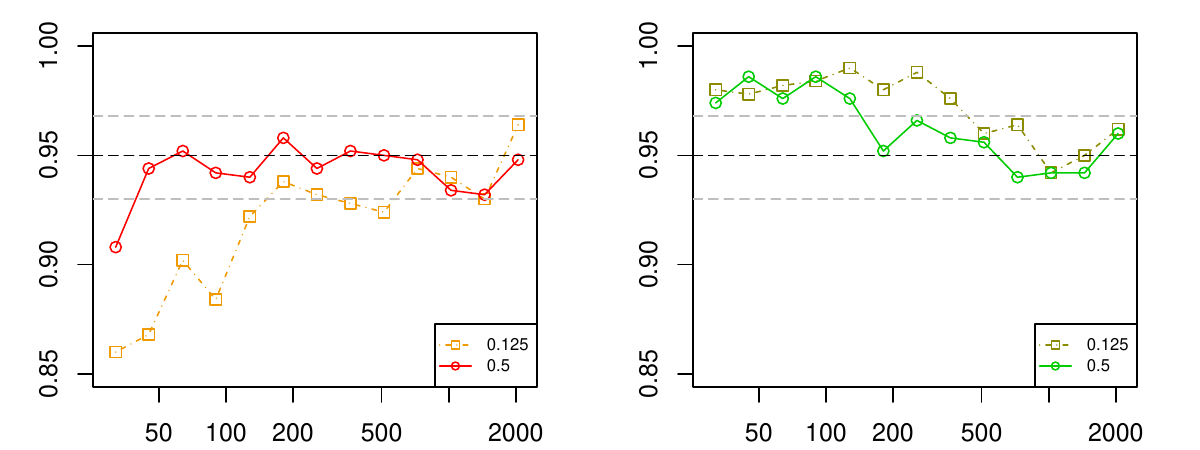}
\caption{Empirical coverage probabilities for the asymptotic confidence intervals at level $\alpha = 0.95$ of $U^{h_6}_N$ (left) and $U^{h_{14}}_N$ (right) for different values of $N$ (x-axis), $\rho \in \{1/8, 1/2 \}$. Grey dashed lines represent the confidence interval at level $0.95$ of the frequency $Z = X/K$, if $X$ follows the binomial distribution with parameters $K$ and $\alpha = 0.95$. }
\label{fig:motif_coverage}
\end{figure}

\subsection{Graphon product distance}

The following simulation results illustrate the application to graphon analysis developed in Section~\ref{sub:example_graphon}.

\paragraph{Model II($\epsilon$):} We consider a Poisson-LBM. As this is a weighted graph model, $\boldsymbol{\pi}$ is a matrix of weights rather than probabilities, i.e. $(\pi_{k \ell})_{1 \le k \le K, 1 \le \ell \le L}$ are real non-negative numbers. We consider two row groups and two column groups of equal proportion. Let $\boldsymbol{\pi}^0$ be a weight matrix. We set $\pi^0_{11} = 4$, $\pi^0_{12} = \pi^0_{21} = 2$ and $\pi^0_{22} = 1$. The LBM with parameters $\boldsymbol{\alpha}$, $\boldsymbol{\beta}$ and $\boldsymbol{\pi}$ is also a BEDD model. Indeed, its graphon has a product form and can be written $w_0(\xi, \eta) = \lambda \bar{w}_0(\xi, \eta) = \lambda f_0(\xi) g_0(\eta)$ where $\lambda = 9/4$, $f_0 = g_0$ both take values $4/3$ on $[0, 0.5]$, $2/3$ on $] 0.5, 1]$.

Now, let $\boldsymbol{\tau}$ be the $2 \times 2$ matrix where $\tau_{11} = \tau_{22} = 2$ and $\tau_{12} = \tau_{21} = 0$. For $\epsilon \ge 0$, we define Model II($\epsilon$) as a LBM with group probabilities $\boldsymbol{\alpha}$ and $\boldsymbol{\beta}$ and with weight matrix $\boldsymbol{\pi}^\epsilon = \frac{\lambda}{\lambda + \epsilon} (\boldsymbol{\pi}^0 + \epsilon \boldsymbol{\tau})$. Thus, Model II(0) is an LBM with product form (BEDD model) and for $\epsilon > 0$, Model II($\epsilon$) with graphon $w_\epsilon = \lambda \bar{w}_\epsilon$ is a perturbed version and rescaled to preserve the same density $\lambda$. As $\epsilon$ grows, the graphon of Model II($\epsilon$) strays further from a BEDD model. Indeed, one can show that $d(\bar{w}_\epsilon) = 64 \epsilon^2 (5 + 2 \epsilon)^2/(9 + 4 \epsilon)^4$, which is an increasing function for $\epsilon \ge 0$.

Under Model II($\epsilon$), for each value $N \in \{ 2^k : k \in \llbracket 5, 12 \rrbracket \}$, $\rho \in \{ 1/8, 1/2 \}$, $\epsilon \in \{0.5, 1, 1.5, 2, 2.5, 3\}$, we simulated $K = 500$ networks of size $m_N \times n_N$ where $m_N = \lfloor \rho N \rfloor$ and $n_N = N - m_N$. For each network, we computed $\widehat{d}_N$ and $\widehat{V}^d_N$ as estimates of $d(\bar{w}_\epsilon)$ and $V^d$ respectively. The Q-Q plots of the studentized statistic $Z^d_N$ are given in Figure~\ref{fig:graphon_qq}. It is apparent that $Z^d_N$ is not centered, especially for $N \lesssim 512$. This is due to the fact that $\widehat{d}_N$ is obtained via the delta method, so it is a biased estimator of $d(\bar{w}_\epsilon)$ for finite values of $N$. However, the bias converges to $0$ when $N$ grows, so we find that the statistic achieves normality for $N \gtrsim 2048$. Figure~\ref{fig:graphon_coverage} gives for different values of $\epsilon$, the coverage probability for $d(\bar{w}_\epsilon)$ of the interval $\left[\widehat{d}_N - \Phi(1-\frac{\alpha}{2})\sqrt{\frac{N}{\widehat{V}^d_N}}, \widehat{d}_N + \Phi(1-\frac{\alpha}{2})\sqrt{\frac{N}{\widehat{V}^d_N}}\right]$ where $\Phi$ is the quantile function of the standard normal distribution. We observe that convergence is fastest when $\rho = 1/2$ but also when $\epsilon$ is larger, which seems to indicate that estimation of $d(\bar{w})$ is more precise for square matrices and when $w$ is further away from a product model.

\begin{figure}[!tb]
\includegraphics[width=0.99\linewidth]{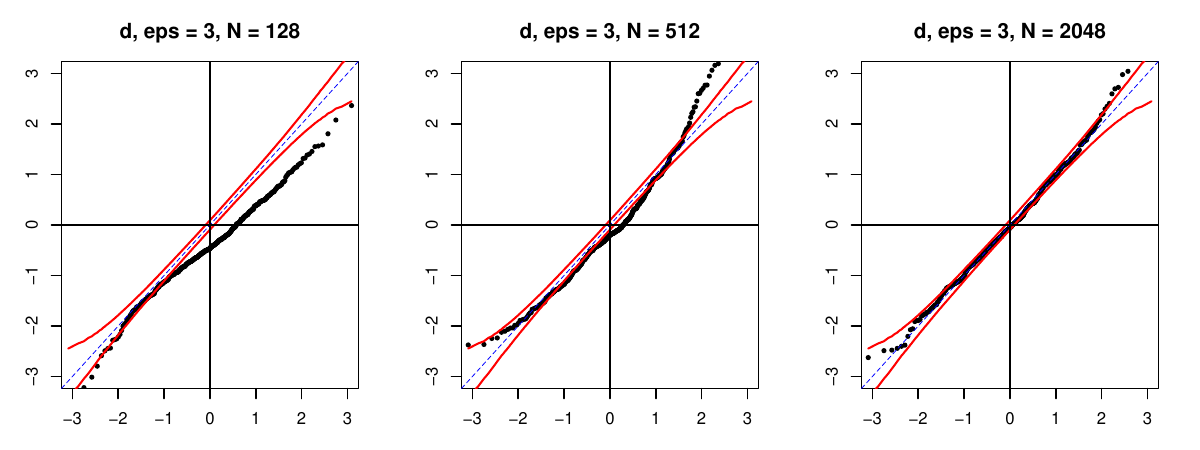}
\caption{Q-Q plot for $Z^d_N$ the studentized statistic associated with $\widehat{d}_N$, $\epsilon = 3$, $\rho = 0.5$}
\label{fig:graphon_qq}
\end{figure}

\begin{figure}[!tb]
\includegraphics[width=0.99\linewidth]{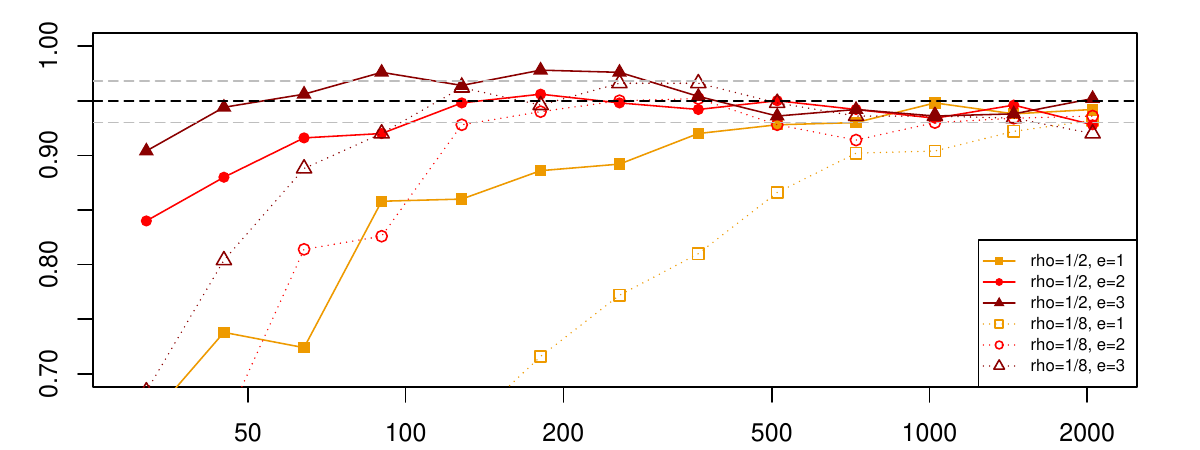}
\caption{Empirical coverage probabilities for the asymptotic confidence intervals at level $\alpha = 0.95$ of $\widehat{d}_N$ for different values of $N$ (x-axis), $\rho \in \{1/8, 1/2 \}$, $\epsilon \in \{1,2,3\}$. Grey dashed lines represent the confidence interval at level $0.95$ of the frequency $Z = X/K$, if $X$ follows the binomial distribution with parameters $K$ and $\alpha = 0.95$. }
\label{fig:graphon_coverage}
\end{figure}

\subsection{Heterogeneity in the row weights of a network}

The following simulation results illustrate the application to network row heterogeneity analysis developed in Section~\ref{sub:example_f2}.

\paragraph{Model III:} In this example, we consider a weighted BEDD model with power-law strength distributions, i.e. the marginals $f$ and $g$ have the form $f(\xi) = (\alpha_f + 1) \xi^{\alpha_f}$ and $g(\eta) = (\alpha_g + 1) \eta^{\alpha_g}$, where $\alpha_f$ and $\alpha_g$ are real non-negative numbers. $\alpha_f$ is directly related to $F_2 = \int f(\xi)^2 d\xi = (\alpha_f + 1)^2/(2 \alpha_f + 1)$. We set $\lambda = 1$, $F_2 = 3$ and $G_2 = 2$.

Under Model III, for each value $N \in \{ 2^k : k \in \llbracket 5, 12 \rrbracket \}$ and $\rho \in \{ 1/8, 1/2, 7/8 \}$, we simulated $K = 500$ networks of size $m_N \times n_N$ where $m_N = \lfloor \rho N \rfloor$ and $n_N = N - m_N$. We estimated $F_2$ and $G_2$ using $\widehat{F}_{2,N}$ and its symmetric counterpart $\widehat{G}_{2,N}$. We also computed $\widehat{V}^{F_2}_N$ and $\widehat{V}^{G_2}_N$ to obtain the studentized statistics $Z^{F_2}_N$ and $Z^{G_2}_N$. Q-Q plots of $Z^{F_2}_N$ are given in Figure~\ref{fig:f2_qq}. Figure~\ref{fig:f2_coverage} gives the respective coverage probabilities for $F_2$ and $G_2$ of the intervals $$\left[\widehat{F}_{2,N} - \Phi(1-\frac{\alpha}{2})\sqrt{\frac{N}{\widehat{V}^{F_2}_N}}, \widehat{F}_{2,N} + \Phi(1-\frac{\alpha}{2})\sqrt{\frac{N}{\widehat{V}^{F_2}_N}}\right]$$ and $$\left[\widehat{G}_{2,N} - \Phi(1-\frac{\alpha}{2})\sqrt{\frac{N}{\widehat{V}^{G_2}_N}}, \widehat{G}_{2,N} + \Phi(1-\frac{\alpha}{2})\sqrt{\frac{N}{\widehat{V}^{G_2}_N}}\right]$$ respectively, where $\Phi$ is the quantile function of the standard normal distribution. Despite the bias due to the delta method, apparent for $N=64$ in Figure~\ref{fig:f2_qq}, we find that the coverage probabilities fall in the confidence intervals even for small values of $N$. We would expect that for $F_2$, the worst case corresponds to $\rho = 1/8$ (there are fewer rows) and that for $G_2$, the worst case corresponds to $\rho = 7/8$ (there are fewer columns), but Figure~\ref{fig:f2_coverage} does not show a clear difference between the three values of $\rho$.

\begin{figure}[!tb]
\includegraphics[width=0.99\linewidth]{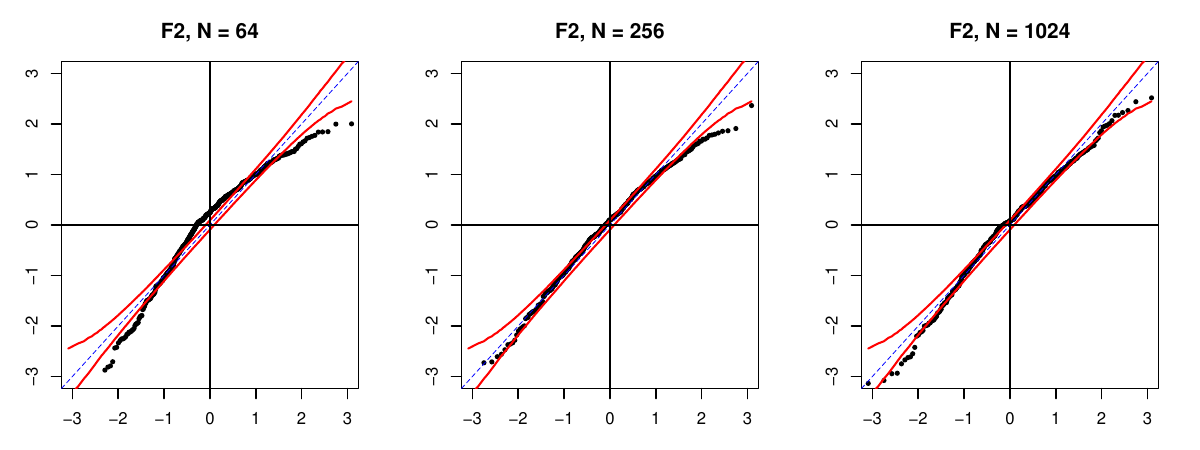}
\caption{Q-Q plot for $Z^{F_2}_N$ the studentized statistic associated with $\widehat{F}_{2,N}$, $\rho = 0.5$}
\label{fig:f2_qq}
\end{figure}

\begin{figure}[!tb]
\includegraphics[width=0.99\linewidth]{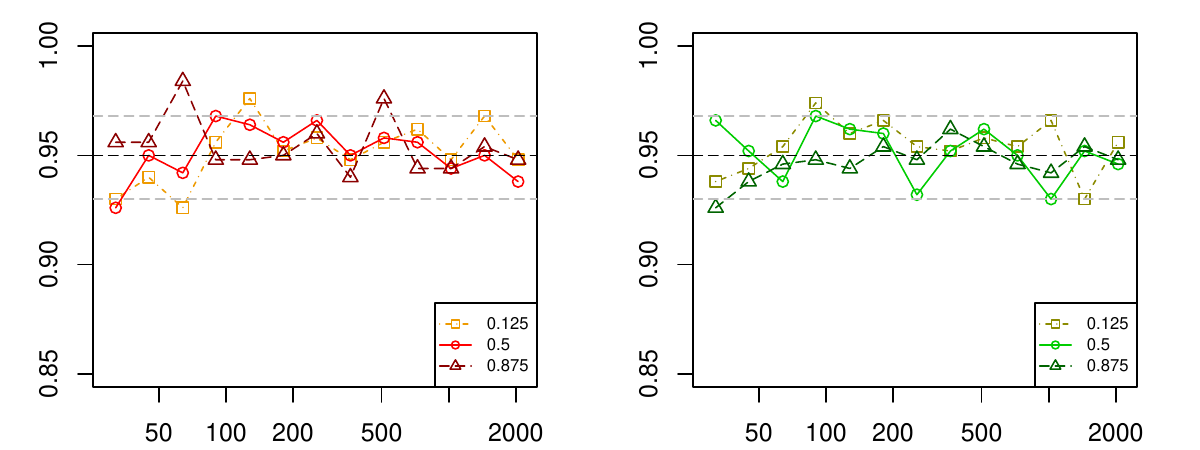}
\caption{Empirical coverage probabilities for the asymptotic confidence intervals at level $\alpha = 0.95$ of $Z^{F_2}_N$ (left) and $Z^{G_2}_N$ (right) for different values of $N$ (x-axis), $\rho \in \{1/8, 1/2, 7/8 \}$. Grey dashed lines represent the confidence interval at level $0.95$ of the frequency $Z = X/K$, if $X$ follows the binomial distribution with parameters $K$ and $\alpha = 0.95$. }
\label{fig:f2_coverage}
\end{figure}

\subsection{Computational efficiency of variance estimators}
\label{sub:simu_variance}

In the following simulation study, we assess the computational performance of variance estimation using the techniques developed in this paper. Specifically, we consider the $U$-statistics $U^{h_1}_N$ and $U^{h_2}_N$ defined in Section~\ref{sub:example_f2}. To estimate their asymptotic variances $V^{h_i}$ for $i \in \{ 1,2\}$, we compare three algorithms:
\begin{itemize}
    \item Algorithm A: We estimate $V^{h_i}$ using empirical covariance estimators for $v^{1,0}_{h_i} = \gamma_h^{1,0}$ and $v^{0,1}_{h_i} = \gamma_h^{0,1}$, as described in the beginning of Section~\ref{sec:variance_estimator}. This is the direct approach to estimating the asymptotic variance of a $U$-statistic on RCE matrices.
    \item Algorithm B: We compute $\widehat{V}^{h_i}_N$ defined in Section~\ref{sub:variance_estimator} by directly applying the kernel function to all subgraphs in the summations of equations~\eqref{eq:est_v10} and~\eqref{eq:est_v01}.
    \item Algorithm C: We compute $\widehat{V}^{h_i}_N$ defined in Section~\ref{sub:variance_estimator}, using an alternative approach that evaluates $U$-statistics on smaller networks and then combines them using the expressions of~\eqref{eq:est_v10_alternate} and~\eqref{eq:est_v01_alternate}. This method leverages matrix operations for efficient computation:
\begin{align}
    U^{h_1}_{N} &= \frac{1}{n_N m_N (m_N-1)} \left[ |Y_N^TY_N|_1 - \Trace(Y_N^TY_N)\right], \label{eq:matrix_op_h1} \\ 
    U^{h_2}_{N} &= \frac{1}{m_N(m_N-1)n_N(n_N-1)} \left[ (|Y_N|_1)^2 - |Y_N^TY_N|_1 + \Trace(Y_N^TY_N)\right. \label{eq:matrix_op_h2} \\ 
    & \quad\left.- |Y_NY_N^T|_1 + \Trace(Y_NY_N^T) - |Y^{\odot2}_N|_1 \right], \nonumber
\end{align}
\end{itemize}
where $Y_N$ is the observed network, with $m_N$ rows and $n_N$ columns.

For each value $N \in \{ 2^{k/2} : k \in \llbracket 6, 16 \rrbracket \}$, we simulated $K = 100$ networks of size $m_N \times n_N$ where $m_N = \lfloor N/2 \rfloor$ and $n_N = N - m_N$, under Model III. We then estimated $V^{h_1}$ and $V^{h_2}$ using Algorithms~A, B and C. Due to its excessive computational cost, Algorithm A has been used only for $N < 32$. Appendix~\ref{app:computation_times} presents the average computation times for each variance estimate. They indicate that, while Algorithm A is by far the slowest among the three, Algorithm C is significantly faster than Algorithm B due to the matrix operation trick, despite both computing the same estimator.

Figure~\ref{fig:mean_var_est} presents the average values of the estimates across network sizes. They show quick convergence for the estimates $\widehat{V}^{h_1}_N$ and $\widehat{V}^{h_2}_N$ provided by Algorithms B and C. In contrast, Algorithm A not only requires a higher computational cost but also produces substantially less precise estimates over the considered network size range. Furthermore, Algorithm A occasionally yields negative variance estimates, which is highly undesirable. These results highlight the advantages of the proposed approach compared to the classic method.

\begin{figure}[!tb]
\centering
\includegraphics[width=0.45\linewidth]{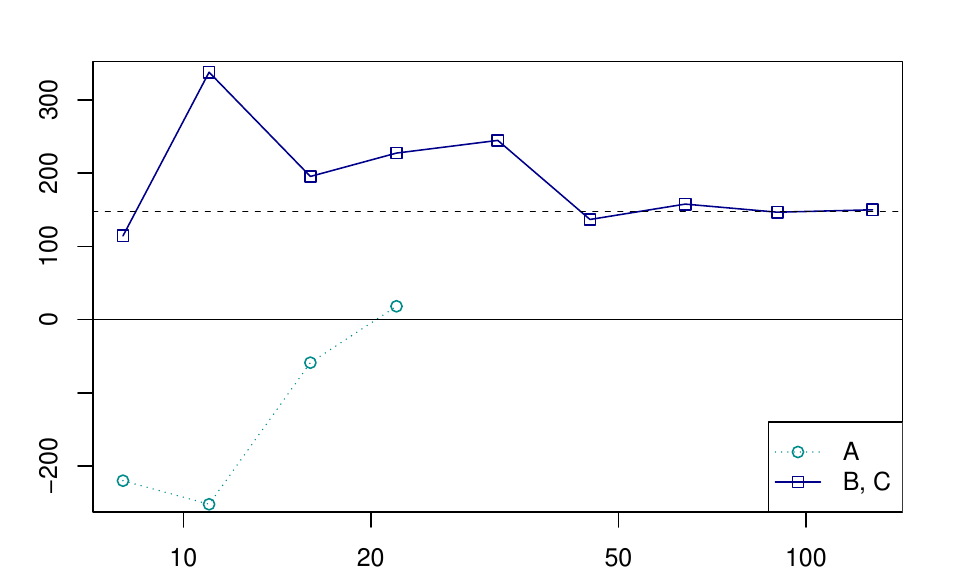}
\includegraphics[width=0.45\linewidth]{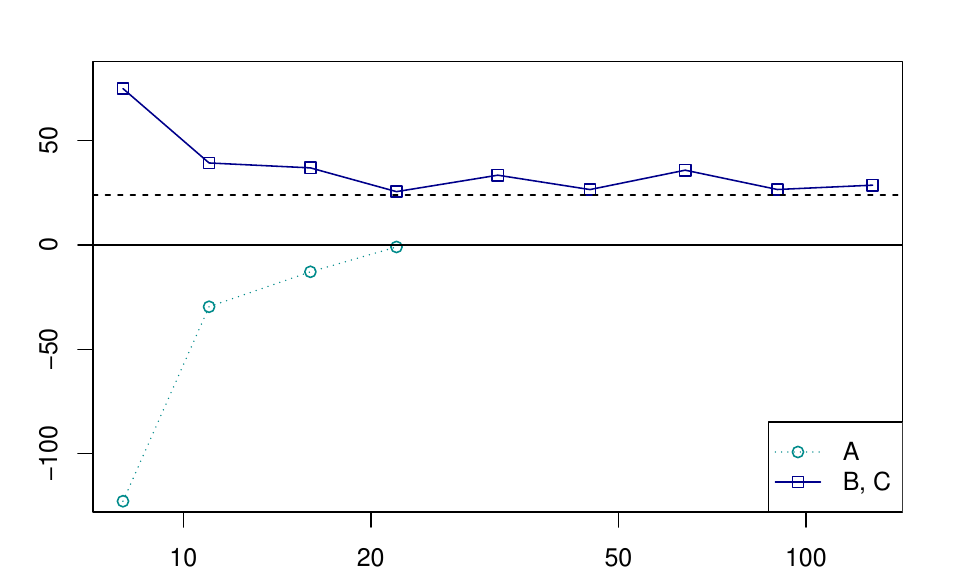}
\caption{Empirical average of the estimates for $V^{h_1}$ (left) and $V^{h_2}$ (right) given by Algorithms A, B and C for different network sizes $N$ (in log-scale). The dashed lines represent the true values of $V^{h_1}$ and $V^{h_2}$. Algorithm A was not tested for $N \ge 32$ due to its excessively high computation time.}
\label{fig:mean_var_est}
\end{figure}

\section{Illustrations}
\label{sec:illustrations}

To illustrate our methodology and the interpretation of some of the $U$-statistics introduced in this paper, we considered the set of law-makers networks compiled by \cite{MSA19}. The database contains networks arising from different fields (ecology, social sciences, life sciences). We focused on the subset of so-called ``legislature'' networks both because of their sizes and because network comparison is of interest for this dataset.

\paragraph{Data description.}
Four law-maker assemblies were considered: the European Parliament (`EP'), the General Assembly of the United Nations (`UN'), the US House of Representatives (`USH') and the US Senate (`USS'). One network has been recorded each year for each parliament; we considered the 26 years from 1979 to 2004, for which the data are available for all four assemblies. The network recorded for a given assembly in a given year consists of the votes (yes or no) of the different members (rows of the adjacency matrix) for the different proposed laws (columns of the adjacency matrix). \\
Figure \ref{fig:legDims} gives the dimensions and densities of the 26 networks collected in each assembly: the main difference is that the European Parliament is both larger (both in terms of members and laws) and sparser than the three others.

\begin{figure}[tb]
  \begin{center}
    \includegraphics[width=.32\textwidth]{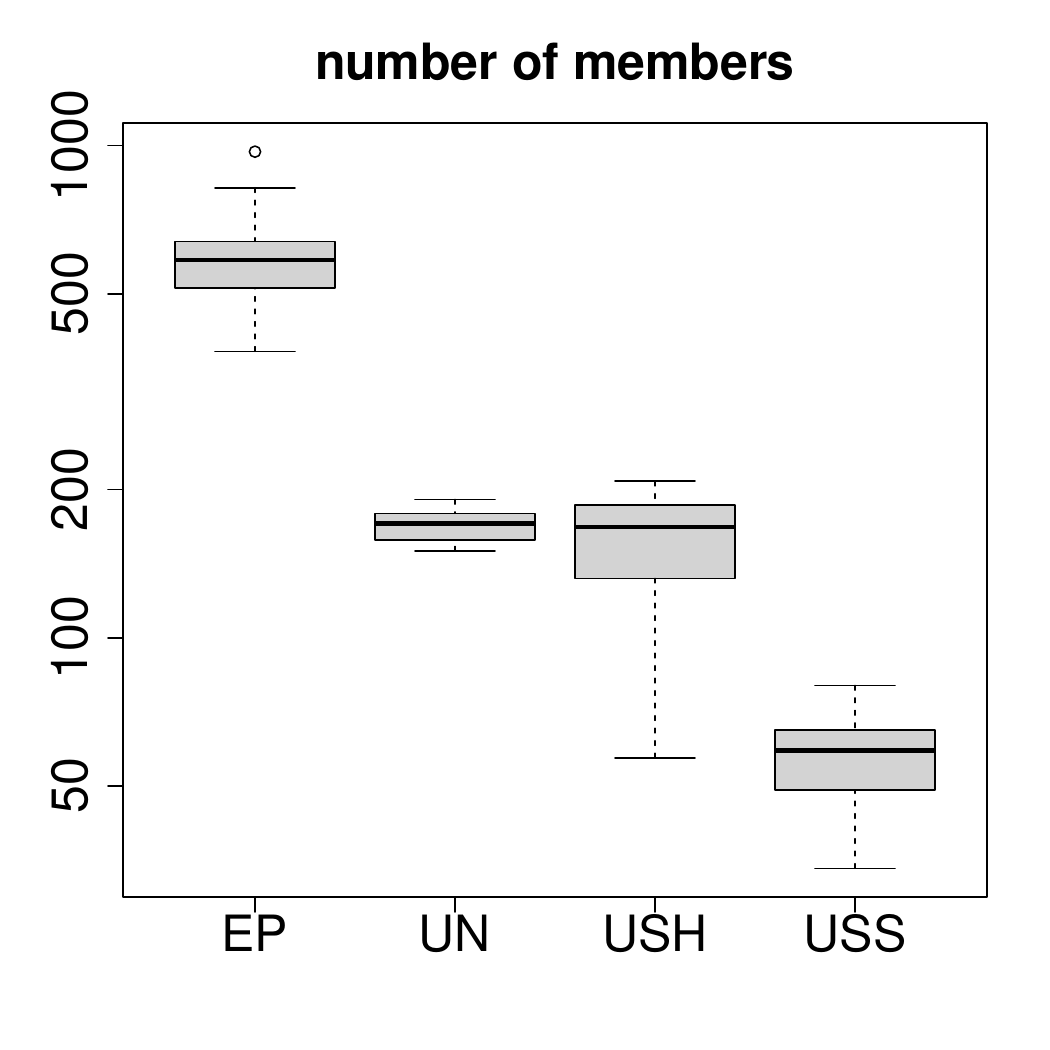}
    \includegraphics[width=.32\textwidth]{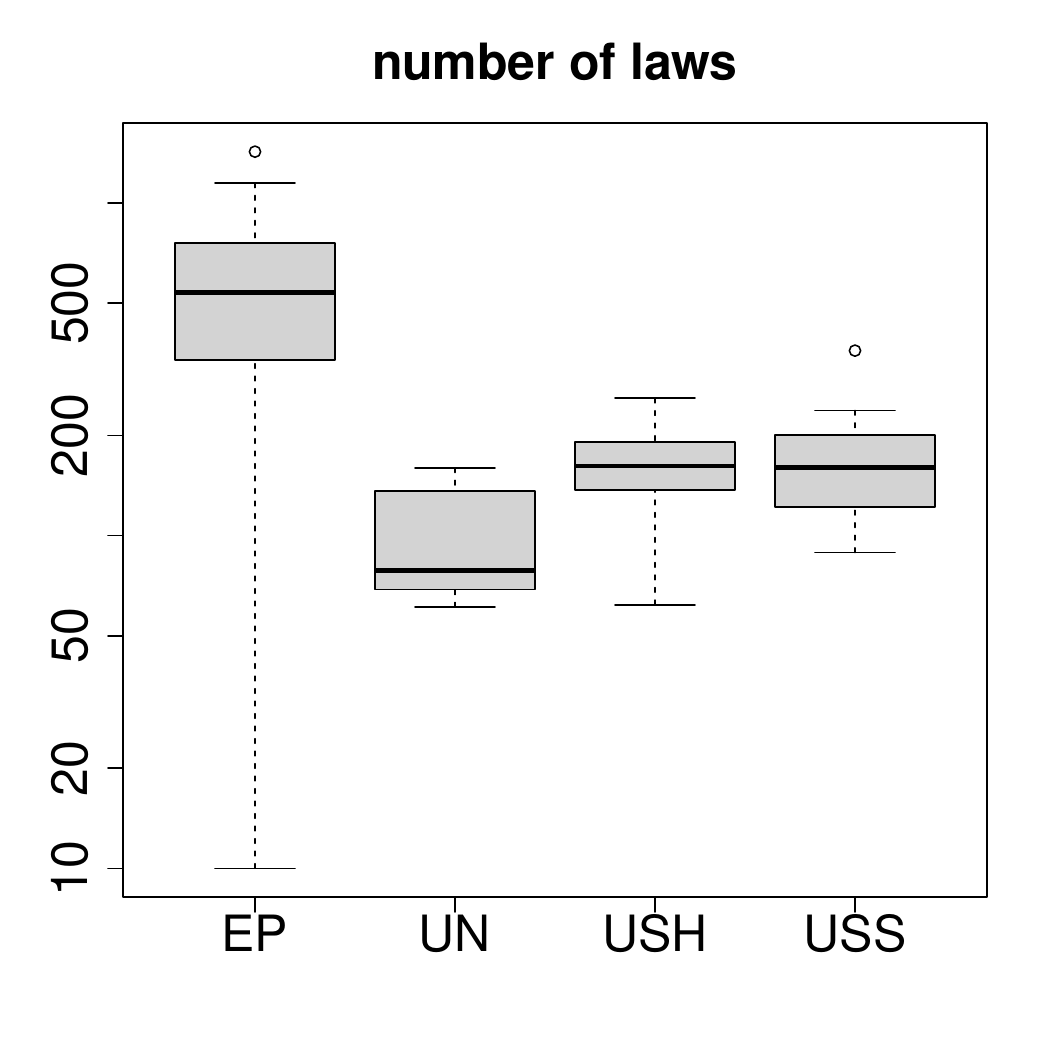}
    \includegraphics[width=.32\textwidth]{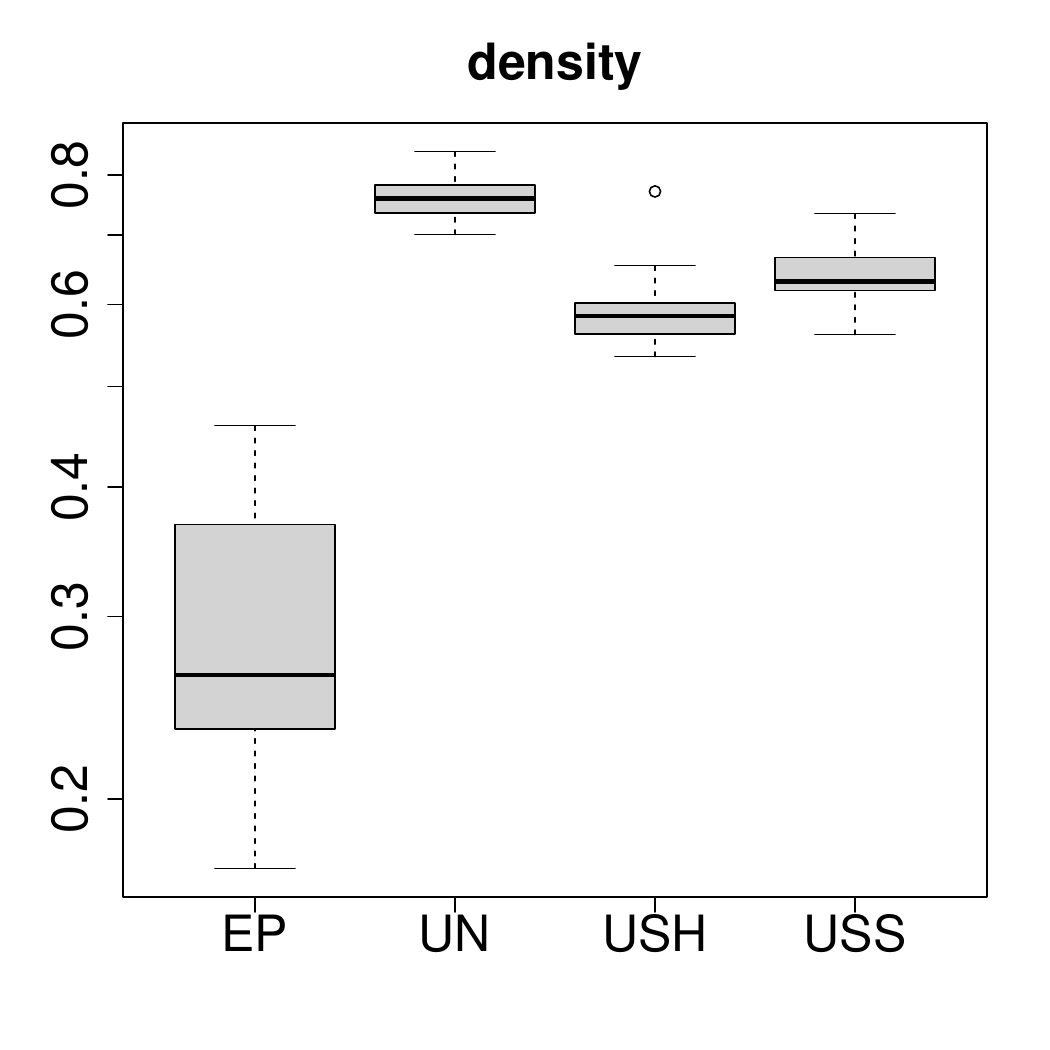}
  \end{center}
  \caption{Distribution of the number of members (left), number of laws (center), and density (right) of the four lawmakers networks across the 26 years (in log-scale). \label{fig:legDims}}
\end{figure}

\paragraph{Degree imbalance.}
We then focused on the degree of imbalance among the rows (resp. columns), which, under the weighted BEDD model defined in Equation \eqref{eq:wbedd}, can be measured by the $U$-statistic $F_2$ (resp. $G_2$). 
Figure \ref{fig:legStatIC} gives the evolution of each of the two indicators over the years for each parliament. We observe that, for each of them, both $F_2$ and $G_2$ remain above 1, all along the period: as expected no uniformity exists, neither among the members ($F_2 > 1$), nor among the laws ($G_2 > 1$). Regarding the US networks (USH and USS), the imbalance is more marked among the laws than between the members. As expected also, the confidence intervals are narrower for the largest networks (EP). No systematic pattern is observed, except the shift in the imbalance among resolutions voted at the General Assembly of the United Nations (UN) that is observed in 1992 (and which happens to coincide with the UN membership of former soviet republics).

\begin{figure}[tb]
  \begin{center}
  \begin{tabular}{cccc}
  \includegraphics[width=.22\textwidth, trim=20 20 20 20, clip=]{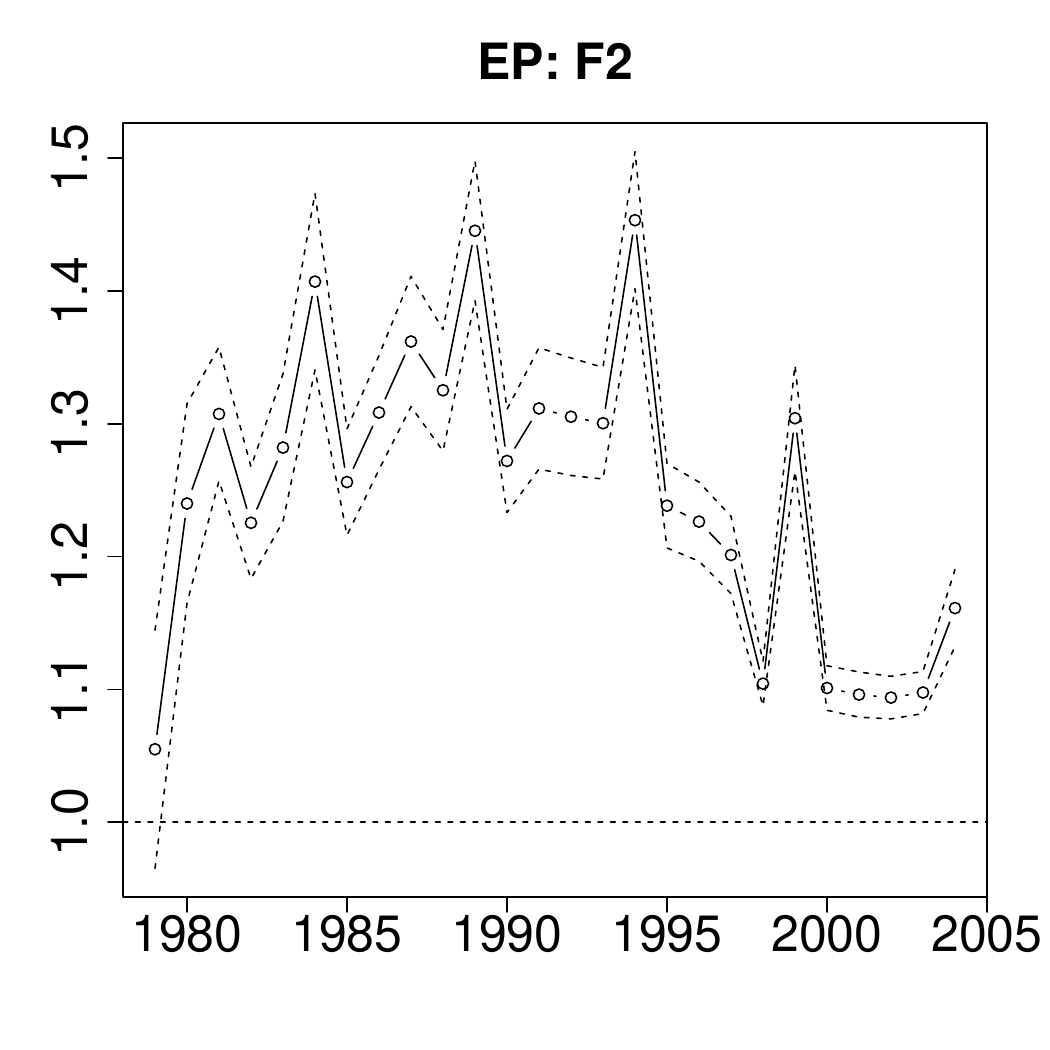} & 
  \includegraphics[width=.22\textwidth, trim=20 20 20 20, clip=]{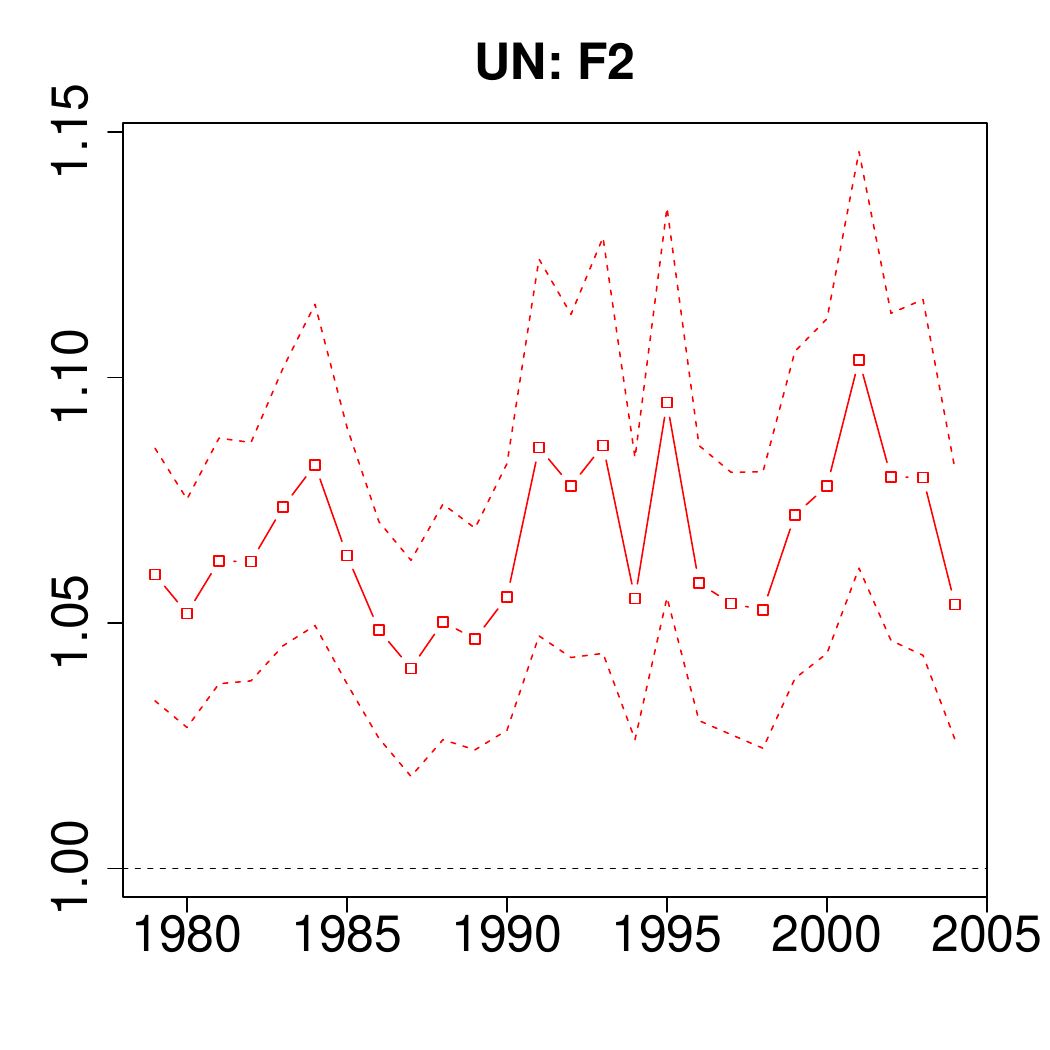} & 
  \includegraphics[width=.22\textwidth, trim=20 20 20 20, clip=]{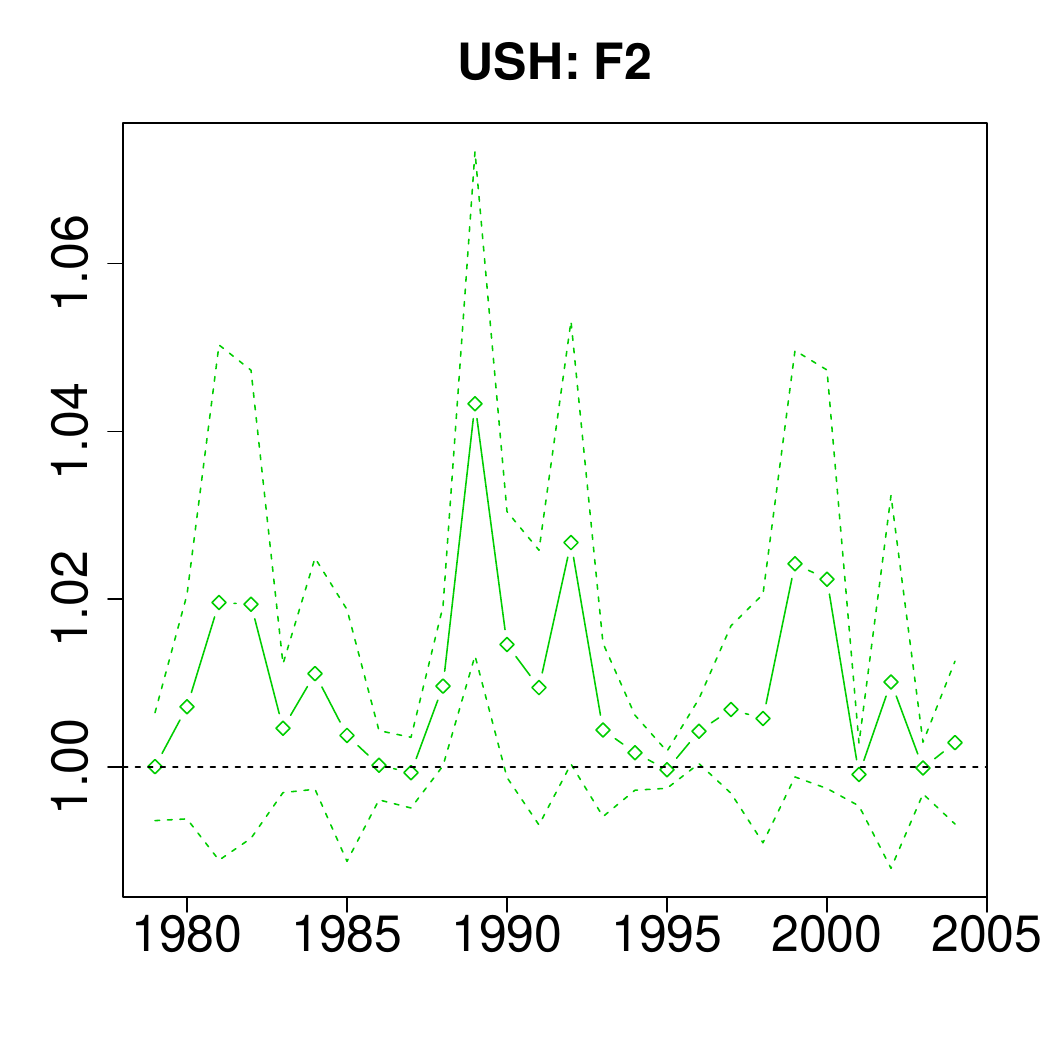} & 
  \includegraphics[width=.22\textwidth, trim=20 20 20 20, clip=]{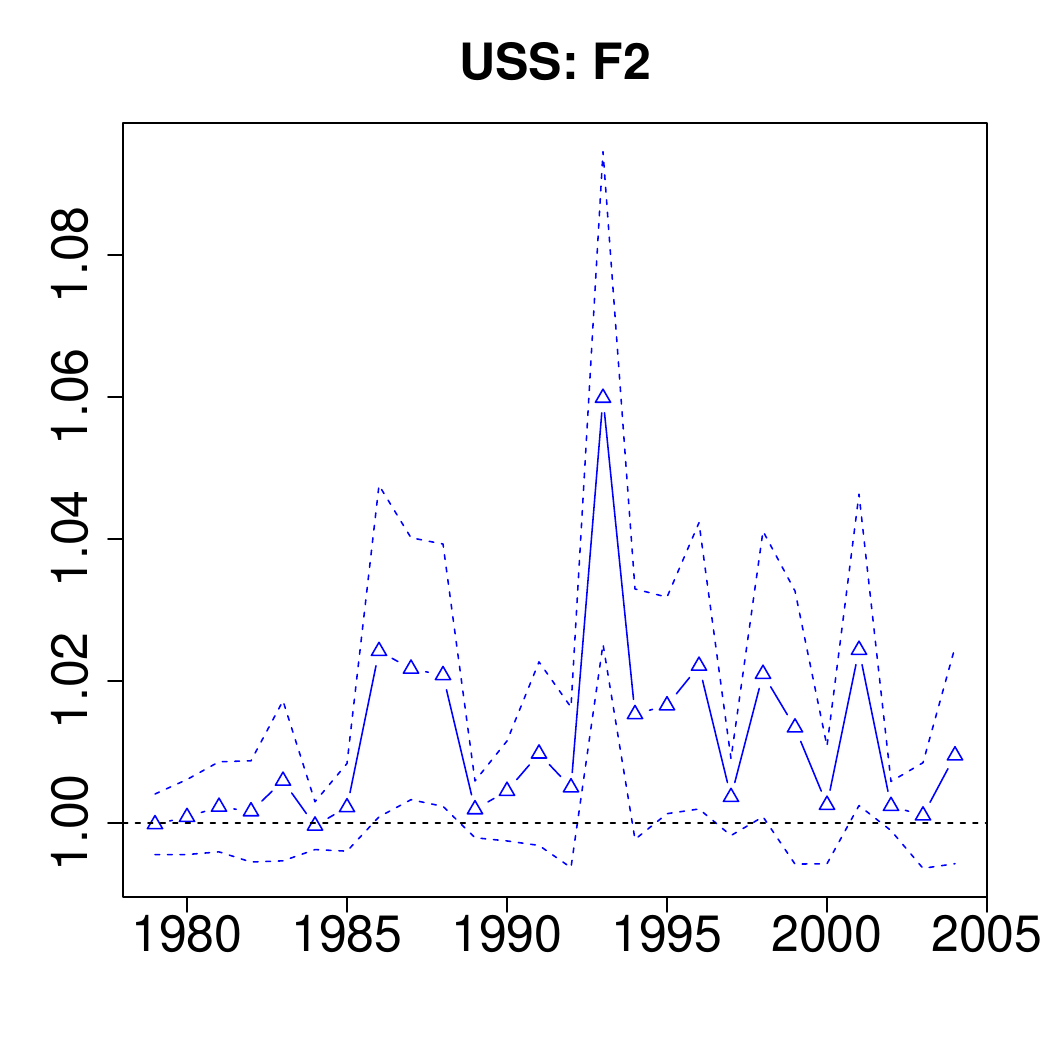} \\
  \includegraphics[width=.22\textwidth, trim=20 20 20 20, clip=]{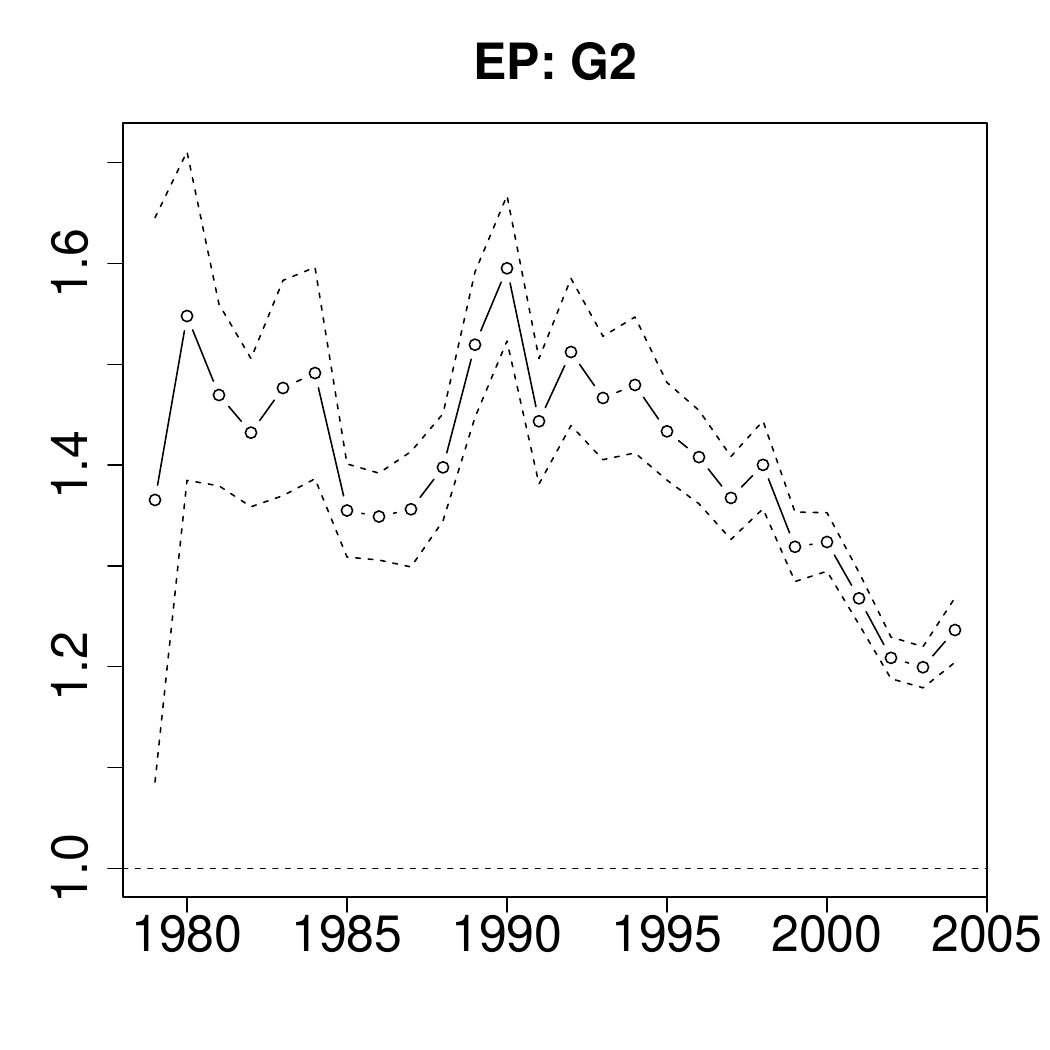} & 
  \includegraphics[width=.22\textwidth, trim=20 20 20 20, clip=]{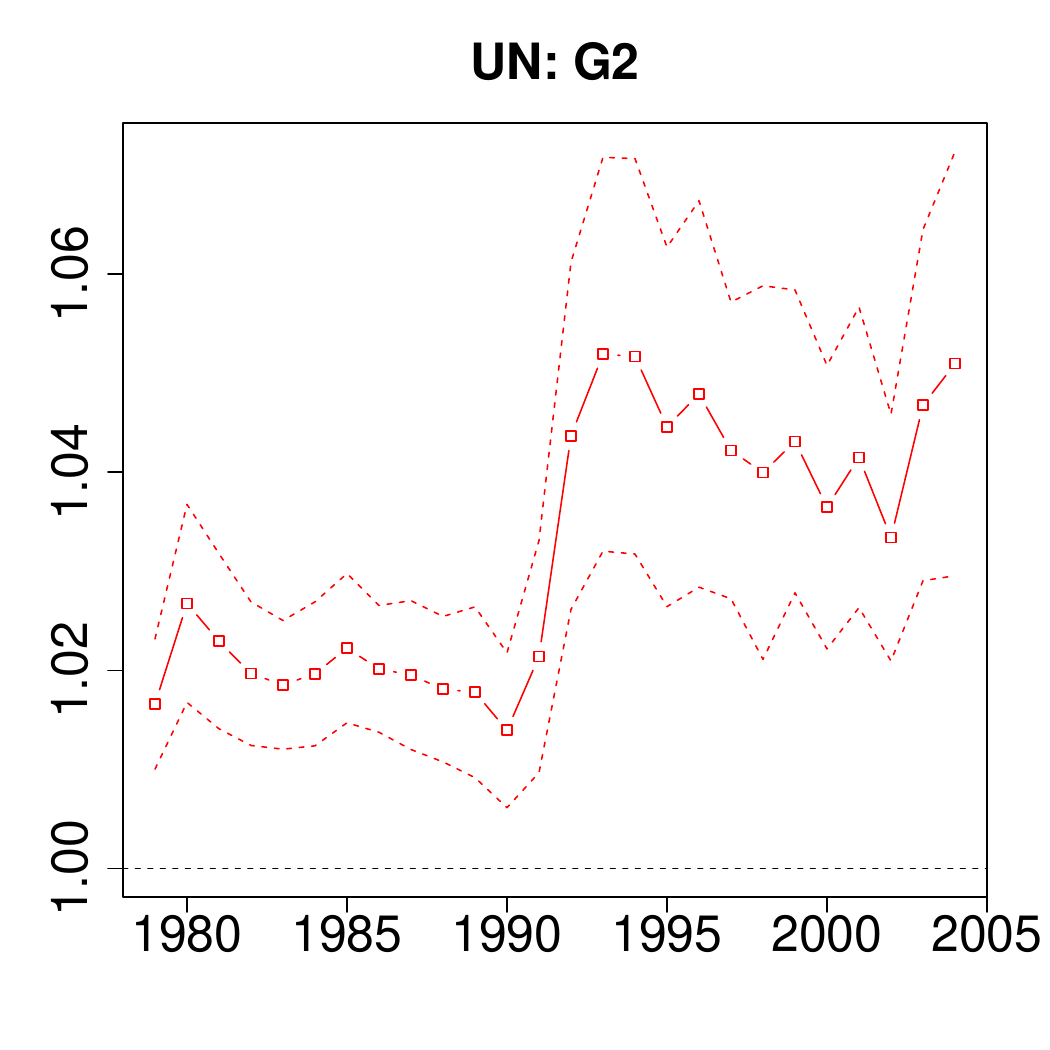} & 
  \includegraphics[width=.22\textwidth, trim=20 20 20 20, clip=]{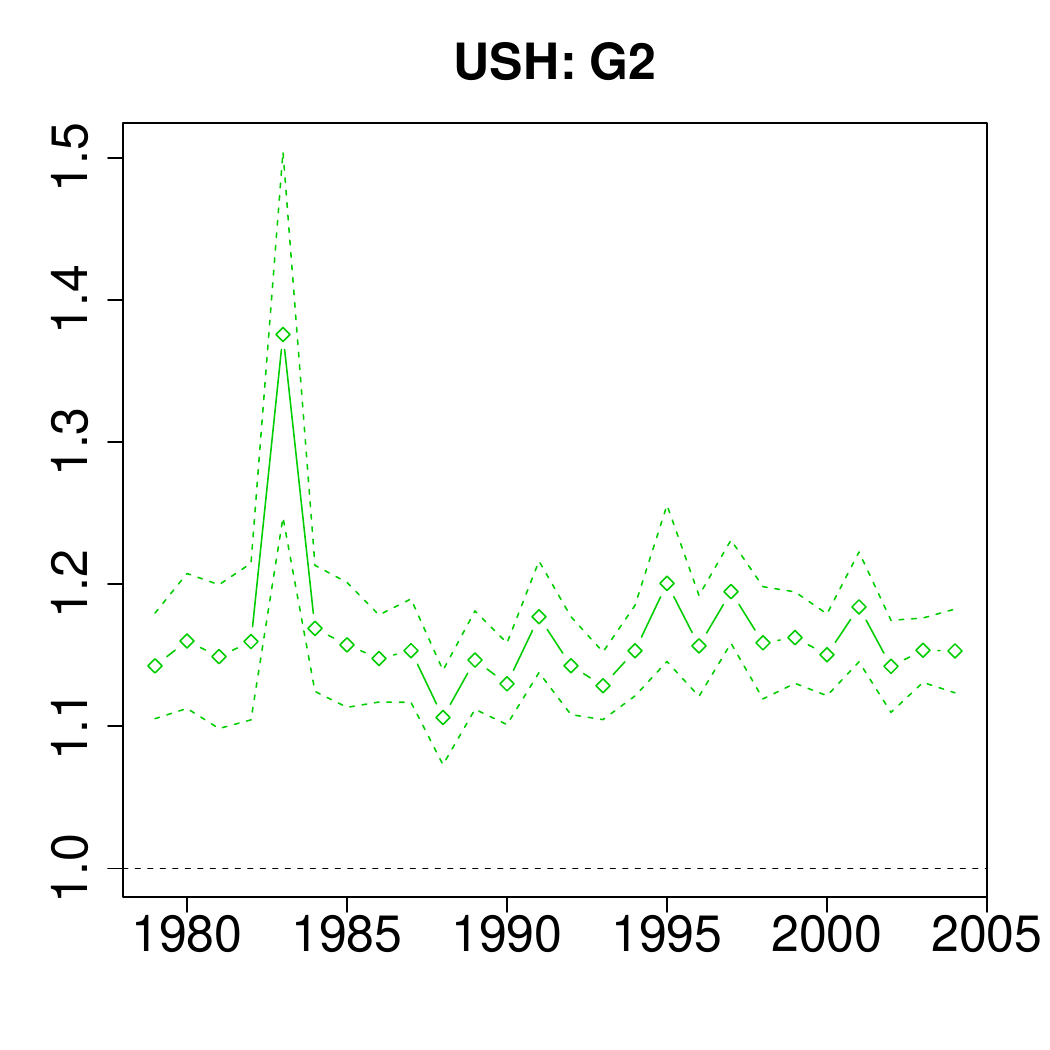} & 
  \includegraphics[width=.22\textwidth, trim=20 20 20 20, clip=]{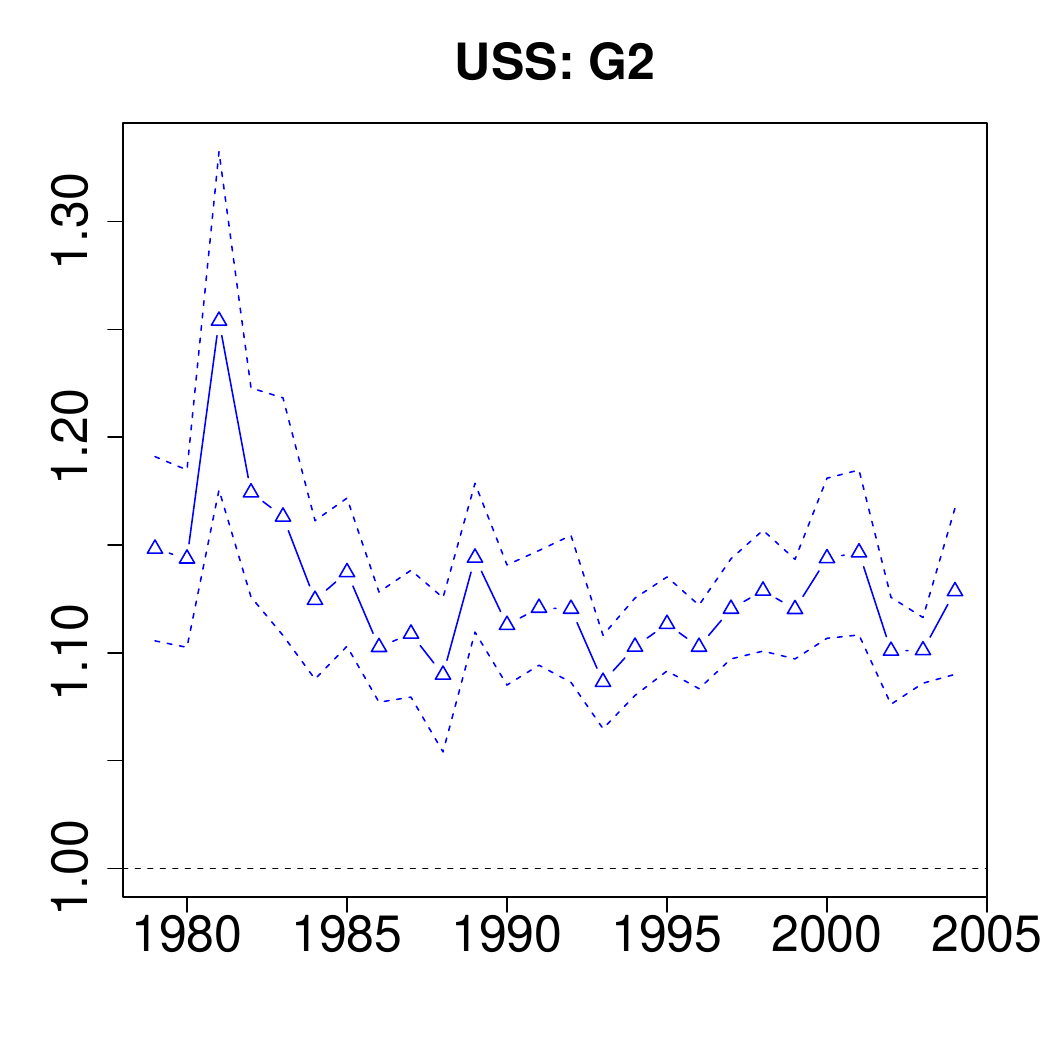} 
  \end{tabular}
  \end{center}
  \caption{Evolution of degree imbalance in each assembly over the years. Top: $F_2$ $U$-statistics, bottom: $G_2$. From left to right: European Parliament, General Assembly of the United Nations, US House of Representatives, and US Senate. Solid line: $U$-statistic as an estimate of $F_2$ (resp. $G_2$). Dotted line: 95\%-confidence interval for $F_2$ (resp. $G_2$).
  \label{fig:legStatIC}}
\end{figure}

\paragraph{Network comparison.}
For each available year, we then compared the networks of the four assemblies in terms of degree imbalance ($F_2$ and $G_2$) and frequency of topological motifs 6 (as given in Figure \ref{fig:motif_6_10}). We chose this motif as it constitutes a clique, characterizing a group behavior, in which close members are in favor of the same laws. For each of these parameters, we use the comparison test procedure described in Section \ref{sec:netComp}. \\
Figure \ref{fig:legCompStat} displays the results. We observe no significant difference between the two US assemblies, which are also the smallest ones: the absence of significant differences can therefore result from a weak power of the tests when considering small networks. We also observe a higher heterogeneity among the members of the European Parliament with respect to all other assemblies, as well as a higher heterogeneity among the members of the United Nations Assembly, with respect to the two US chambers. A different picture is obtained for the heterogeneity among the laws, which is significantly higher in the UN assembly and significantly lower in the EP assembly. \\
The frequency of motif 6 is interesting, as it may reveal a specific socio-political behavior. In this respect, the group structure turns out to be much stronger in the UN than in the EP, the members of which represent both different political orientations and different nations.

\begin{figure}[tb]
  \begin{center}
  \includegraphics[width=.32\textwidth]{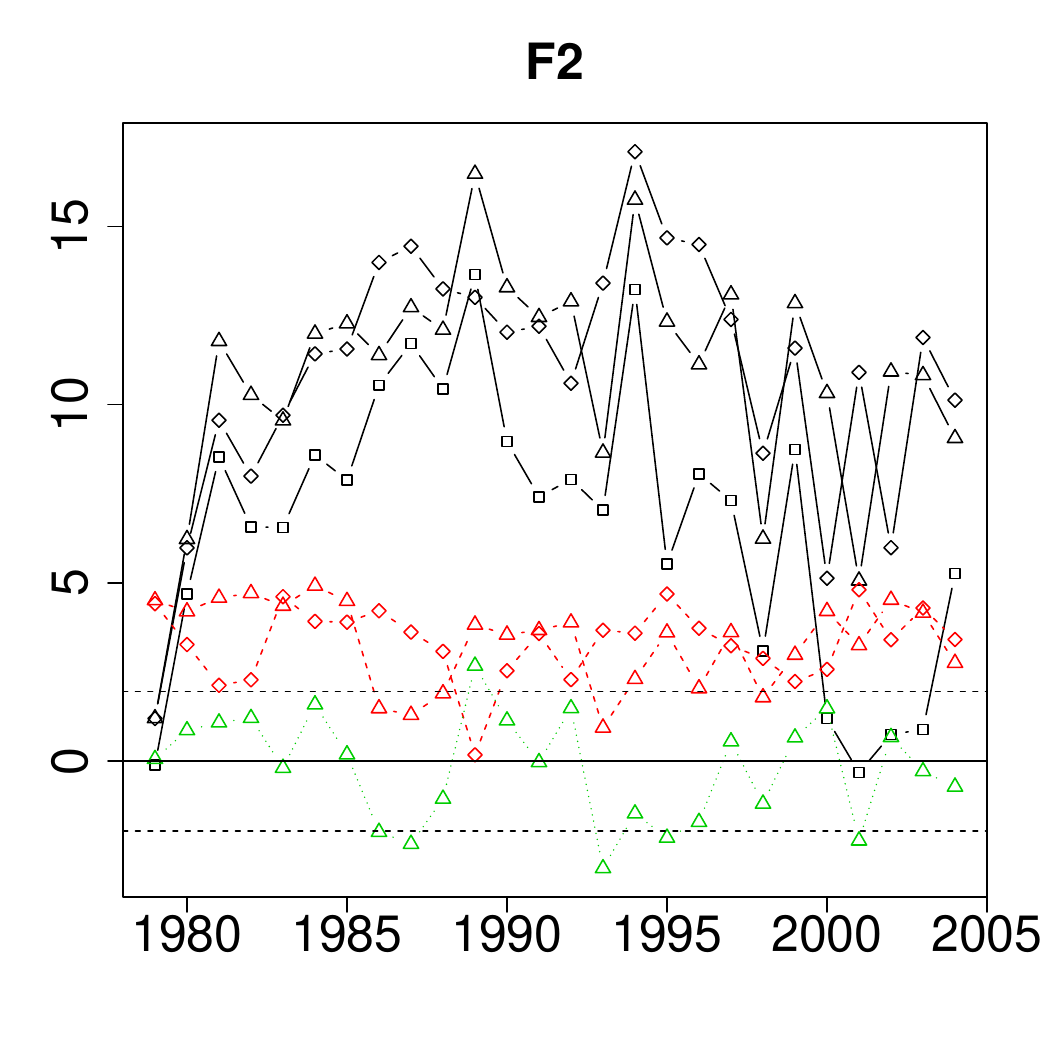} 
  \includegraphics[width=.32\textwidth]{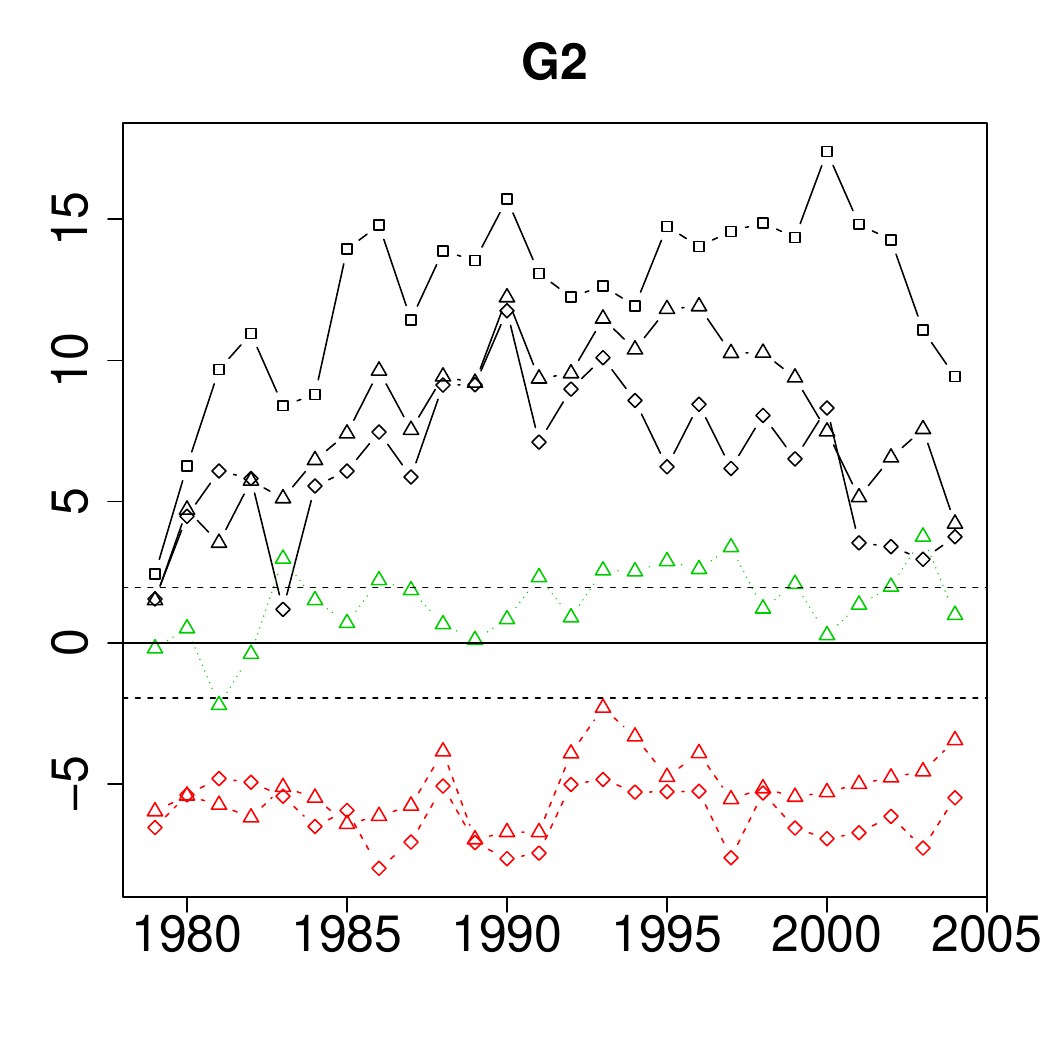}
  \includegraphics[width=.32\textwidth]{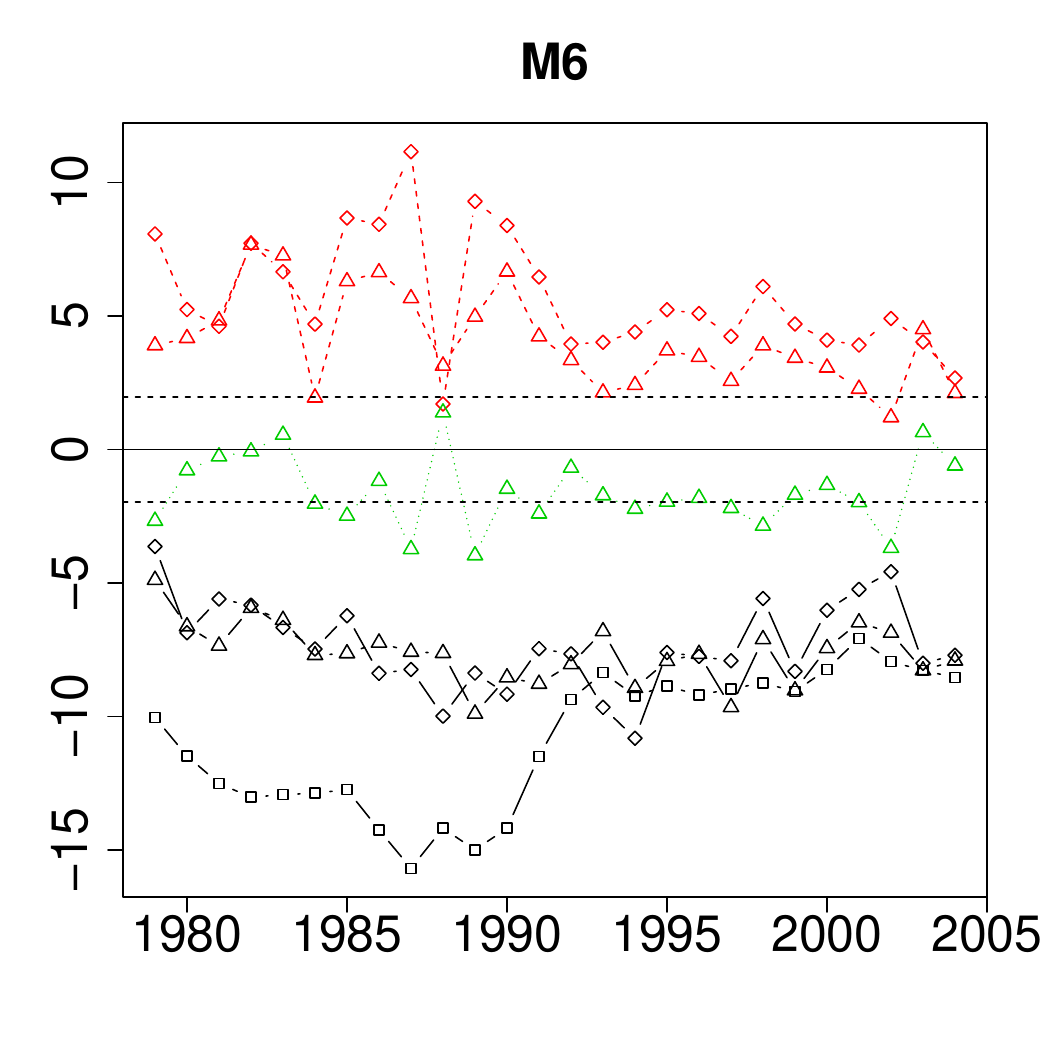}
  \end{center}
  \caption{Network comparison for the $U$-statistics, $F_2$, $G_2$ and for the count the motif 6 (M6). 
  EP-UN= ---$\Box$---, 
  EP-USS= ---$\Diamond$---, 
  EP-USH= ---$\triangle$---, 
  UN-USS= \textcolor{red}{- -$\Diamond$- -}, 
  UN-USH= \textcolor{red}{- -$\triangle$- -}, 
  USS-USH= \textcolor{green}{$\cdots\triangle\cdots$}.
  Horizontal lines = standard normal quantiles with levels .025 and .975.
  \label{fig:legCompStat}}
\end{figure}

\section{Illustrations}
\label{sec:illustrations}

To illustrate our methodology and the interpretation of some of the $U$-statistics introduced in this paper, we considered the set of law-makers networks compiled by \cite{MSA19}. The database contains networks arising from different fields (ecology, social sciences, life sciences). We focused on the subset of so-called ``legislature'' networks both because of their sizes and because network comparison is of interest for this dataset.

\paragraph{Data description.}
Four law-maker assemblies were considered: the European Parliament (`EP'), the General Assembly of the United Nations (`UN'), the US House of Representatives (`USH') and the US Senate (`USS'). One network has been recorded each year for each parliament; we considered the 26 years from 1979 to 2004, for which the data are available for all four assemblies. The network recorded for a given assembly in a given year consists of the votes (yes or no) of the different members (rows of the adjacency matrix) for the different proposed laws (columns of the adjacency matrix). \\
Figure \ref{fig:legDims} gives the dimensions and densities of the 26 networks collected in each assembly: the main difference is that the European Parliament is both larger (both in terms of members and laws) and sparser than the three others.

\begin{figure}[tb]
  \begin{center}
    \includegraphics[width=.32\textwidth]{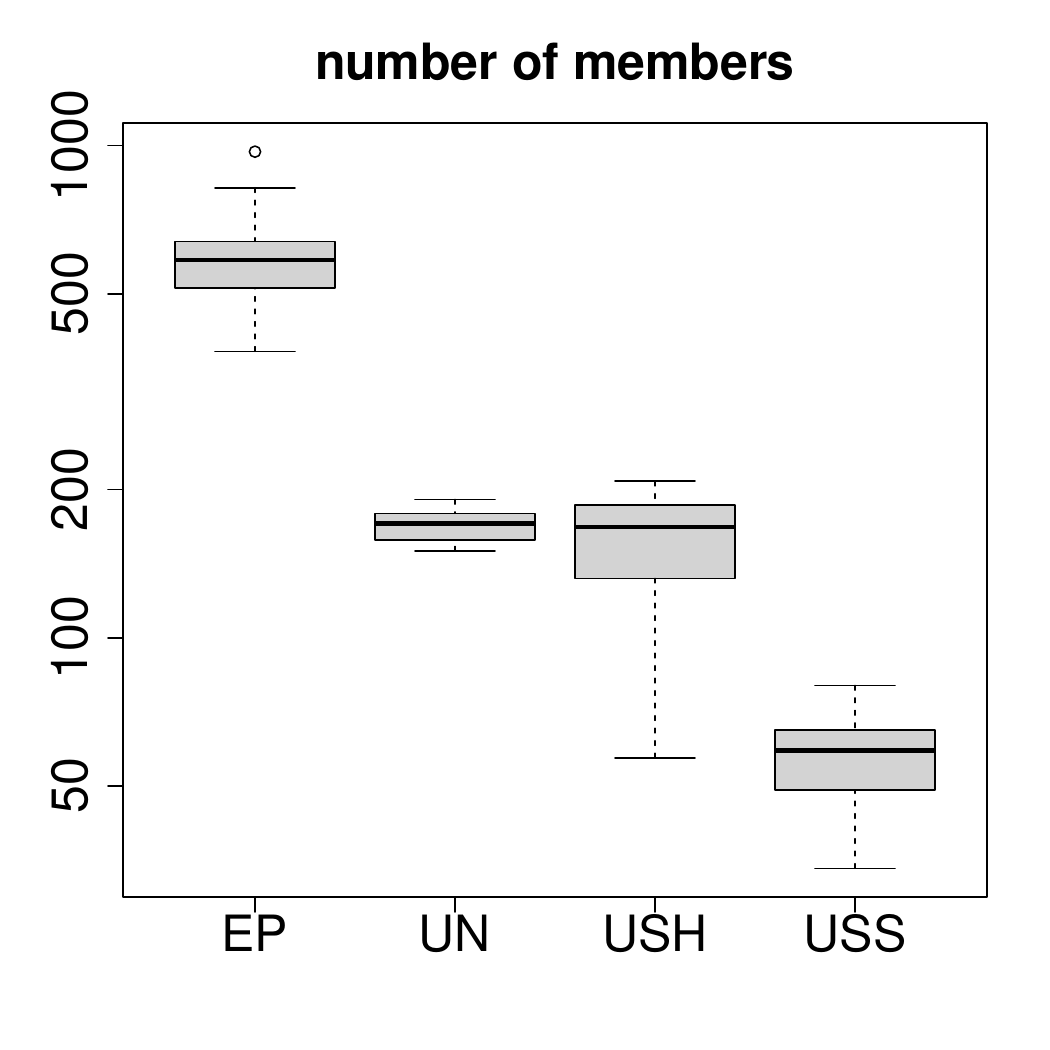}
    \includegraphics[width=.32\textwidth]{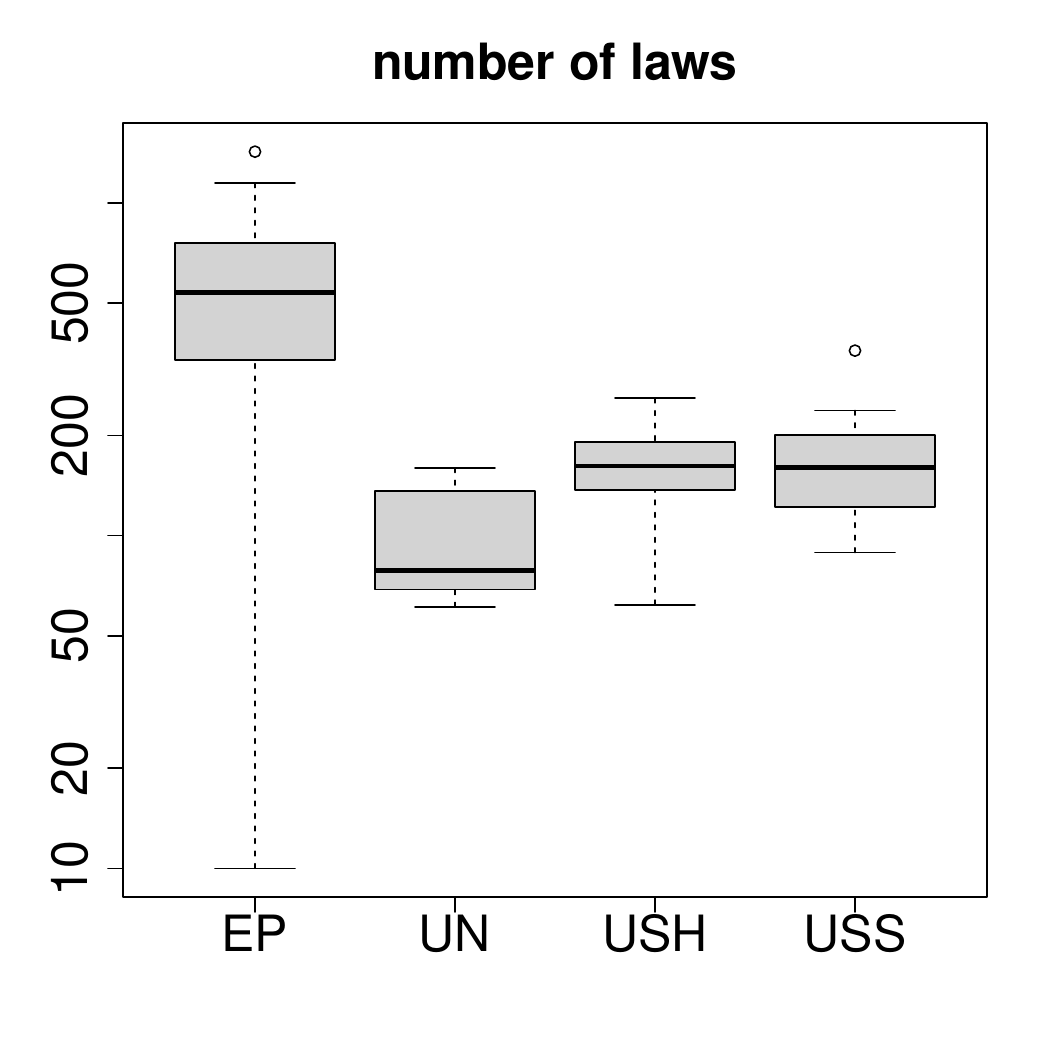}
    \includegraphics[width=.32\textwidth]{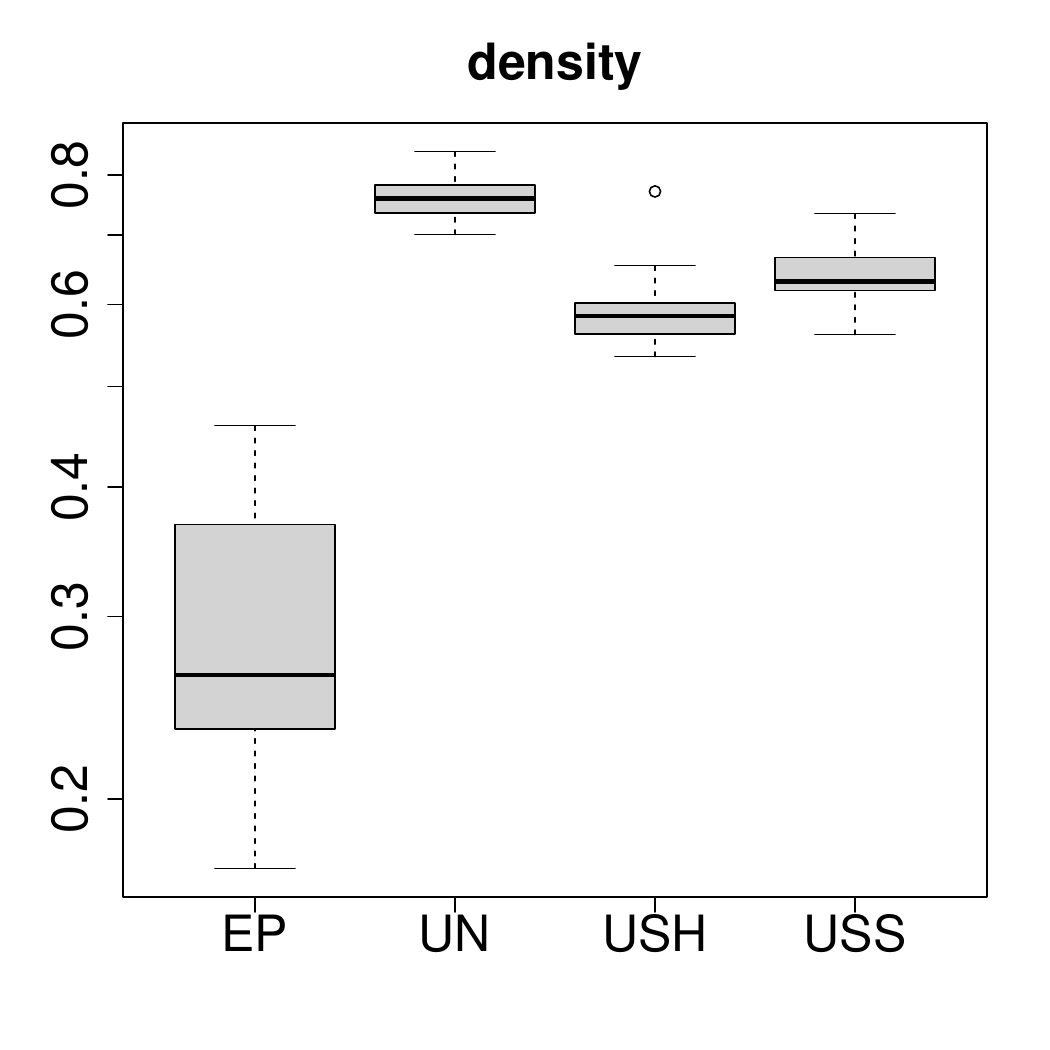}
  \end{center}
  \caption{Distribution of the number of members (left), number of laws (center), and density (right) of the four lawmakers networks across the 26 years (in log-scale). \label{fig:legDims}}
\end{figure}

\paragraph{Degree imbalance.}
We then focused on the degree of imbalance among the rows (resp. columns), which, under the weighted BEDD model defined in Equation \eqref{eq:wbedd}, can be measured by the $U$-statistic $F_2$ (resp. $G_2$). 
Figure \ref{fig:legStatIC} gives the evolution of each of the two indicators over the years for each parliament. We observe that, for each of them, both $F_2$ and $G_2$ remain above 1, all along the period: as expected no uniformity exists, neither among the members ($F_2 > 1$), nor among the laws ($G_2 > 1$). Regarding the US networks (USH and USS), the imbalance is more marked among the laws than between the members. As expected also, the confidence intervals are narrower for the largest networks (EP). No systematic pattern is observed, except the shift in the imbalance among resolutions voted at the General Assembly of the United Nations (UN) that is observed in 1992 (and which happens to coincide with the UN membership of former soviet republics).

\begin{figure}[tb]
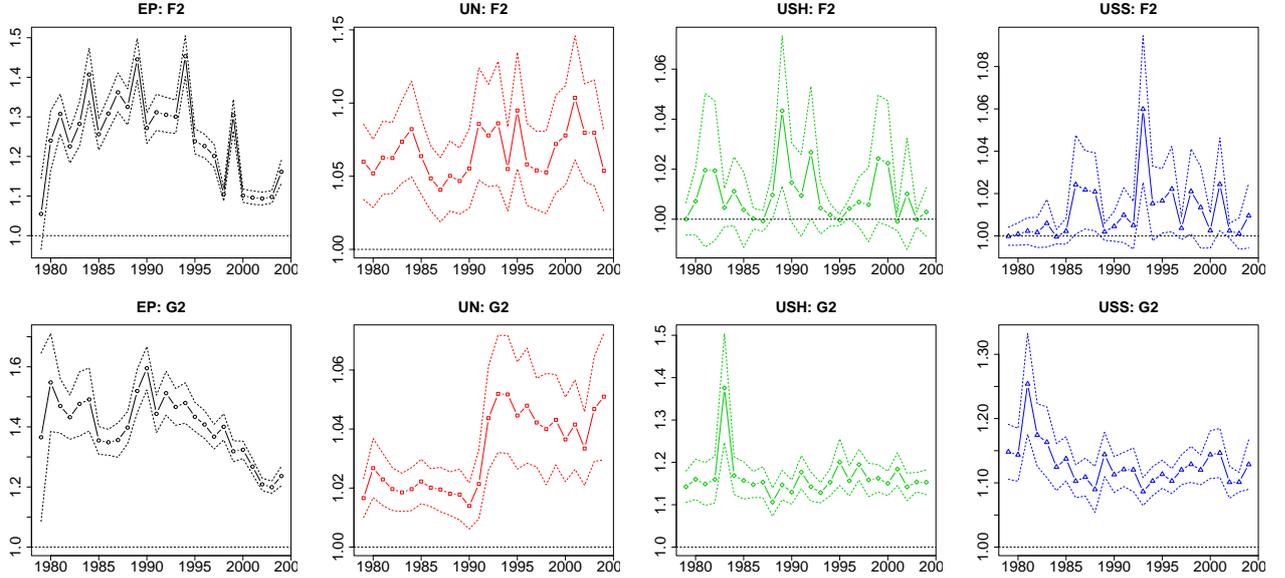

  \begin{center}
  \begin{tabular}{cccc}
  \includegraphics[width=.22\textwidth, trim=20 20 20 20, clip=]{legislature-EP-F2-statIC} & 
  \includegraphics[width=.22\textwidth, trim=20 20 20 20, clip=]{legislature-UN-F2-statIC} & 
  \includegraphics[width=.22\textwidth, trim=20 20 20 20, clip=]{legislature-USH-F2-statIC} & 
  \includegraphics[width=.22\textwidth, trim=20 20 20 20, clip=]{legislature-USS-F2-statIC} \\
  \includegraphics[width=.22\textwidth, trim=20 20 20 20, clip=]{legislature-EP-G2-statIC} & 
  \includegraphics[width=.22\textwidth, trim=20 20 20 20, clip=]{legislature-UN-G2-statIC} & 
  \includegraphics[width=.22\textwidth, trim=20 20 20 20, clip=]{legislature-USH-G2-statIC} & 
  \includegraphics[width=.22\textwidth, trim=20 20 20 20, clip=]{legislature-USS-G2-statIC} 
  \end{tabular}
  \end{center}
  \caption{Evolution of degree imbalance in each assembly over the years. Top: $F_2$ $U$-statistics, bottom: $G_2$. From left to right: European Parliament, General Assembly of the United Nations, US House of Representatives, and US Senate. Solid line: $U$-statistic as an estimate of $F_2$ (resp. $G_2$). Dotted line: 95\%-confidence interval for $F_2$ (resp. $G_2$).
  \label{fig:legStatIC}}
\end{figure}

\paragraph{Network comparison.}
For each available year, we then compared the networks of the four assemblies in terms of degree imbalance ($F_2$ and $G_2$) and frequency of topological motifs 6 (as given in Figure \ref{fig:motif_6_10}). We chose this motif as it constitutes a clique, characterizing a group behavior, in which close members are in favor of the same laws. For each of these parameters, we use the comparison test procedure described in Section \ref{sec:netComp}. \\
Figure \ref{fig:legCompStat} displays the results. We observe no significant difference between the two US assemblies, which are also the smallest ones: the absence of significant differences can therefore result from a weak power of the tests when considering small networks. We also observe a higher heterogeneity among the members of the European Parliament with respect to all other assemblies, as well as a higher heterogeneity among the members of the United Nations Assembly, with respect to the two US chambers. A different picture is obtained for the heterogeneity among the laws, which is significantly higher in the UN assembly and significantly lower in the EP assembly. \\
The frequency of motif 6 is interesting, as it may reveal a specific socio-political behavior. In this respect, the group structure turns out to be much stronger in the UN than in the EP, the members of which represent both different political orientations and different nations.

\begin{figure}[tb]
  \begin{center}
  \includegraphics[width=.32\textwidth]{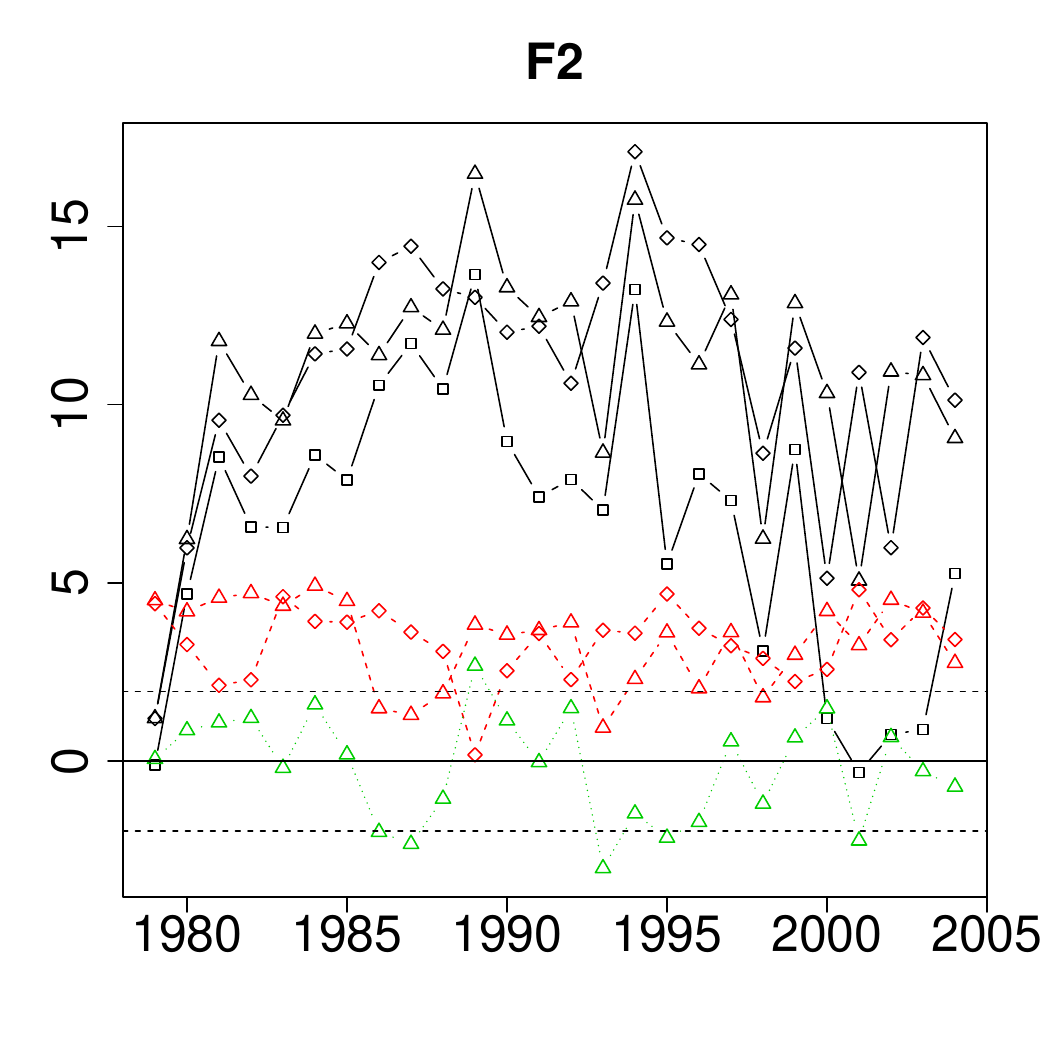} 
  \includegraphics[width=.32\textwidth]{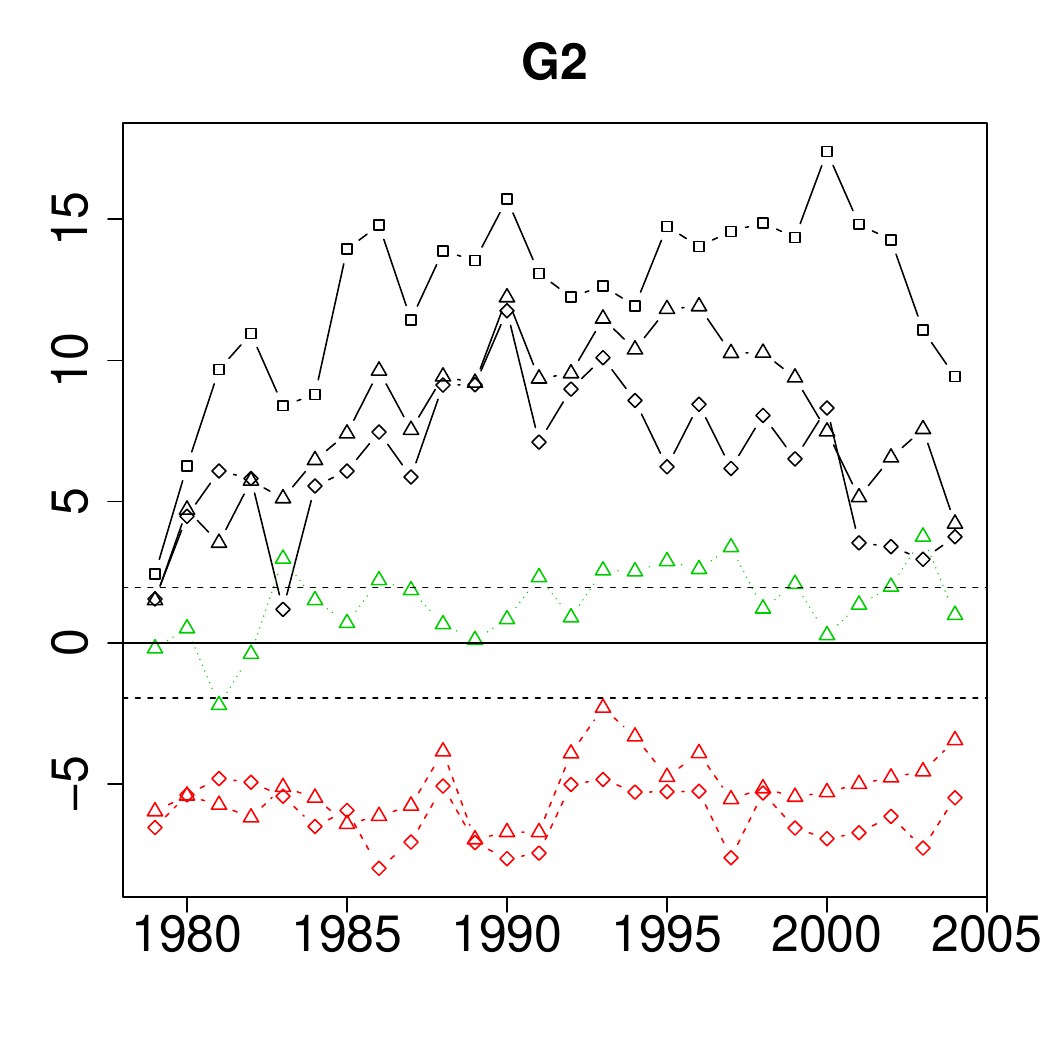}
  \includegraphics[width=.32\textwidth]{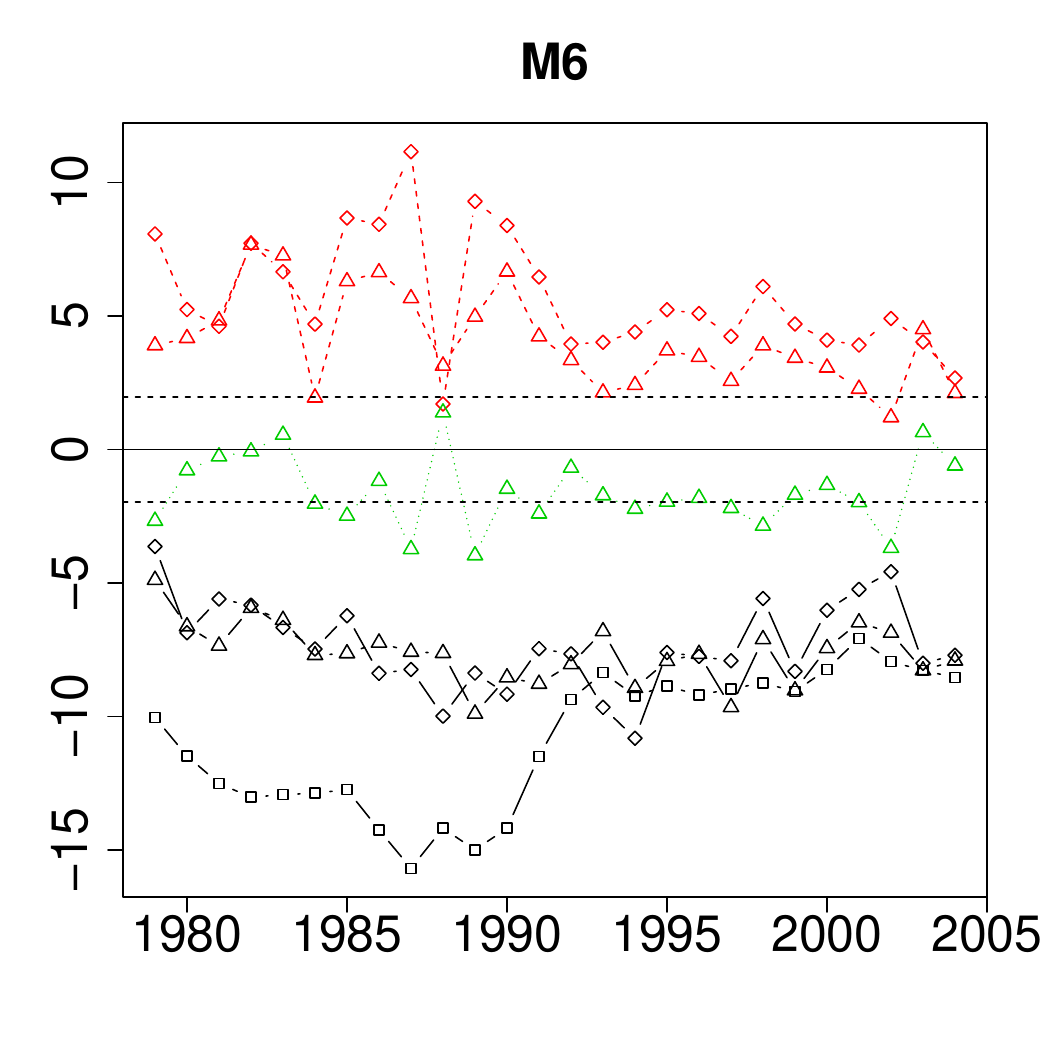}
  \end{center}
  \caption{Network comparison for the $U$-statistics, $F_2$, $G_2$ and for the count the motif 6 (M6). 
  EP-UN= ---$\Box$---, 
  EP-USS= ---$\Diamond$---, 
  EP-USH= ---$\triangle$---, 
  UN-USS= \textcolor{red}{- -$\Diamond$- -}, 
  UN-USH= \textcolor{red}{- -$\triangle$- -}, 
  USS-USH= \textcolor{green}{$\cdots\triangle\cdots$}.
  Horizontal lines = standard normal quantiles with levels .025 and .975.
  \label{fig:legCompStat}}
\end{figure}

\section*{Acknowledgements}

This work was funded by the grant ANR-18-CE02-0010-01 of the French National Research Agency ANR (project EcoNet) and a grant from R\'egion \^Ile-de-France.

\newpage

\bibliographystyle{plain}
\bibliography{biblio}

\newpage

\begin{appendix}

\section{Backward martingales}
\label{app:martingales}

Here, we present the backward martingales and their convergence theorem, which is used to prove the convergence of some estimators. The proof of Theorem~\ref{th:martingale_convergence} can be found in~\cite{doob1953stochastic}, Section 7, Theorem 4.2. We recall beforehand the definition of a decreasing filtration.

\begin{definition}
    A decreasing filtration is a decreasing sequence of $\sigma$-fields $\mathcal{F} = (\mathcal{F}_n)_{n \ge 1}$, i.e. such that for all $n \ge 1$, $\mathcal{F}_{n+1} \subset \mathcal{F}_n$.
\end{definition}

\begin{definition}
    Let $\mathcal{F} = (\mathcal{F}_n)_{n \ge 1}$ be a decreasing filtration and $M = (M_n)_{n \ge 1}$ a sequence of integrable random variables adapted to $\mathcal{F}$. $(M_n, \mathcal{F}_n)_{n \ge 1}$ is a backward martingale if and only if for all $n \ge 1$, $\mathbb{E}[M_n \mid \mathcal{F}_{n+1}] = M_{n+1}$.
\end{definition}

\begin{theorem}
    Let $(M_n, \mathcal{F}_n)_{n \ge 1}$ be a backward martingale. Then, $(M_n)_{n \ge 1}$ is uniformly integrable, and,  denoting $M_\infty = \mathbb{E}[M_1 \mid \mathcal{F}_\infty]$ where $\mathcal{F}_\infty = \bigcap_{n = 1}^{\infty} \mathcal{F}_n$, we have
    \begin{equation*}
        M_n \xrightarrow[n \rightarrow \infty]{a.s., L_1} M_\infty.
    \end{equation*}
    \label{th:martingale_convergence}
\end{theorem}

\section{Additional results presented in Section~\ref{sec:hoeffding} and proofs}
\label{app:hoeffding}

\subsection{Orthogonality of projection spaces}

First, we prove Proposition~\ref{prop:ortho_proj}. This proposition relies on the fact that the projections are \textit{conditionally centered}. This property is given by the following lemma, the proof of which is provided in Appendix~\ref{app:hoeffding}.

\begin{lemma} Let $h$ be a kernel function of size $p \times q$. Let $(\mathbf{i},\mathbf{j}) \in \mathcal{P}_r(\mathbb{N}) \times \mathcal{P}_c(\mathbb{N})$, where $(0,0) < (r,c) \le (p,q)$. For all $\underline{\mathbf{i}} \subset \mathbf{i}$ and $\underline{\mathbf{j}} \subset \mathbf{j}$, we have
\begin{equation*}
    \E[p^{r,c}h(Y_{\mathbf{i}, \mathbf{j}}) \mid \mathcal{A}_{\underline{\mathbf{i}}, \underline{\mathbf{j}}}] = 0.
\end{equation*}
\label{lem:proj_smaller_cond_exp}
\end{lemma}

\begin{proof}
We prove this lemma by induction on $(r,c)$ the sizes of $\mathbf{i}$ and $\mathbf{j}$. For $(r,c) = (1,0)$ and $(r,c) = (0,1)$, we have 
\begin{equation*}
\E[p^{1,0} h(Y_{\mathbf{i}, \emptyset}) \mid \mathcal{A}_{\emptyset, \emptyset}] = \E[\psi^{1,0}h(Y_{\mathbf{i}, \emptyset})] - \E[h(Y_{\llbracket p \rrbracket, \llbracket q \rrbracket})] = 0
\end{equation*}
and 
\begin{equation*}
\E[p^{0,1} h(Y_{\emptyset, \mathbf{j}}) \mid \mathcal{A}_{\emptyset, \emptyset}] = \E[\psi^{0,1}h(Y_{\emptyset, \mathbf{i}})] - \E[h(Y_{\llbracket p \rrbracket, \llbracket q \rrbracket})] = 0.
\end{equation*}

Suppose that the lemma is true for all $(0,0) < (r',c') < (r,c)$. Let $(\mathbf{i},\mathbf{j}) \in \mathcal{P}_r(\mathbb{N}) \times \mathcal{P}_c(\mathbb{N})$, $\underline{\mathbf{i}} \subset \mathbf{i}$ and $\underline{\mathbf{j}} \subset \mathbf{j}$. Denote $\underline{r} = \Card(\underline{\mathbf{i}})$ and $\underline{c} = \Card(\underline{\mathbf{j}})$. We can write
    \begin{equation*}
    \begin{split}
        \E[p^{r,c}h(Y_{\mathbf{i}, \mathbf{j}}) \mid \mathcal{A}_{\underline{\mathbf{i}}, \underline{\mathbf{j}}}] &= 
        \E[\psi^{r,c} h(Y_{\mathbf{i}, \mathbf{j}}) \mid \mathcal{A}_{\underline{\mathbf{i}}, \underline{\mathbf{j}}}] - \E[p^{\underline{r},\underline{c}}h(Y_{\underline{\mathbf{i}}, \underline{\mathbf{j}}}) \mid \mathcal{A}_{\underline{\mathbf{i}}, \underline{\mathbf{j}}}] \\
        &\quad - \sum_{(0,0) < (r',c') < (r,c)} \sum_{\substack{\mathbf{i}' \in \mathcal{P}_{r'}(\mathbf{i}), \mathbf{j}' \in \mathcal{P}_{c'}(\mathbf{j}) \\ (\mathbf{i}',\mathbf{j}') \neq (\underline{\mathbf{i}}, \underline{\mathbf{j}})}} \E[p^{r',c'}h(Y_{\mathbf{i}', \mathbf{j}'}) \mid \mathcal{A}_{\underline{\mathbf{i}}, \underline{\mathbf{j}}}] \\
        &= \sum_{(0,0) \le (r',c') < (\underline{r},\underline{c})} \sum_{\substack{\mathbf{i}' \in \mathcal{P}_{r'}(\underline{\mathbf{i}}), \mathbf{j}' \in \mathcal{P}_{c'}(\underline{\mathbf{j}})}} \E[p^{r',c'}h(Y_{\mathbf{i}', \mathbf{j}'}) \mid \mathcal{A}_{\underline{\mathbf{i}}, \underline{\mathbf{j}}}]\\
        &\quad - \sum_{(0,0) < (r',c') < (r,c)} \sum_{\substack{\mathbf{i}' \in \mathcal{P}_{r'}(\mathbf{i}), \mathbf{j}' \in \mathcal{P}_{c'}(\mathbf{j}) \\ (\mathbf{i}',\mathbf{j}') \neq (\underline{\mathbf{i}}, \underline{\mathbf{j}})}} \E[p^{r',c'}h(Y_{\mathbf{i}', \mathbf{j}'}) \mid \mathcal{A}_{\underline{\mathbf{i}}, \underline{\mathbf{j}}}] \\
        &= - \sum_{(0,0) < (r',c') < (r,c)} \sum_{\substack{\mathbf{i}' \in \mathcal{P}_{r'}(\mathbf{i}), \mathbf{j}' \in \mathcal{P}_{c'}(\mathbf{j}) \\ \mathbf{i}' \not\subset \underline{\mathbf{i}}, \mathbf{j}' \not\subset \underline{\mathbf{j}}}} \E[p^{r',c'}h(Y_{\mathbf{i}', \mathbf{j}'}) \mid \mathcal{A}_{\mathbf{i}' \cap \underline{\mathbf{i}}, \mathbf{j}' \cap \underline{\mathbf{j}}}].
    \end{split}
    \end{equation*}
    where we have used the fact that the $p^{r',c'}h(Y_{\mathbf{i}', \mathbf{j}'})$ are measurable by their respective $\mathcal{A}_{\mathbf{i}', \mathbf{j}'}$ so that $\E[p^{r',c'}h(Y_{\mathbf{i}', \mathbf{j}'}) \mid \mathcal{A}_{\underline{\mathbf{i}}, \underline{\mathbf{j}}}] = \E[p^{r',c'}h(Y_{\mathbf{i}', \mathbf{j}'}) \mid \mathcal{A}_{\mathbf{i}' \cap \underline{\mathbf{i}}, \mathbf{j}' \cap \underline{\mathbf{j}}}]$. 

    Since the last sum excludes the case $\mathbf{i}' = \underline{\mathbf{i}}$ and $\mathbf{j}' = \underline{\mathbf{j}}$, then the induction hypothesis ensures that all the terms are equal to $0$, so $\E[p^{r,c}h(Y_{\mathbf{i}, \mathbf{j}}) \mid \mathcal{A}_{\underline{\mathbf{i}}, \underline{\mathbf{j}}}] = 0$, which concludes the proof by induction.
\end{proof}

\begin{proof}[Proof of Proposition~\ref{prop:ortho_proj}]
    The two properties are proven similarly and derive from the fact that $(\mathbf{i}_1, \mathbf{j}_1) \neq (\mathbf{i}_2, \mathbf{j}_2)$. This is true for both properties.
    
    Consider any (possibly equal) $(r_1,c_1)$ and $(r_2,c_2)$ and associated $(\mathbf{i}_1, \mathbf{j}_1) \neq (\mathbf{i}_2, \mathbf{j}_2)$. Then $\mathbf{i}_1 \neq \mathbf{i}_2$ or $\mathbf{j}_1 \neq \mathbf{j}_2$. Without loss of generality, assume that $\mathbf{i}_1 \neq \mathbf{i}_2$ so there is an element $i_2 \in \mathbf{i}_2$ which is not included in $\mathbf{i}_1$. Then 
    \begin{equation*}
    \begin{split}
        \E[p^{r_1,c_1}h_1(Y_{\mathbf{i}_1,\mathbf{j}_1}) p^{r_2,c_2}h_2(Y_{\mathbf{i}_2,\mathbf{j}_2})] &= \E[\E[p^{r_1,c_1}h_1(Y_{\mathbf{i}_1,\mathbf{j}_1}) p^{r_2,c_2}h_2(Y_{\mathbf{i}_2,\mathbf{j}_2}) \mid \mathcal{A}_{\{i_2\}, \emptyset}]] \\
        &= \E[p^{r_1,c_1}h_1(Y_{\mathbf{i}_1,\mathbf{j}_1}) \E[p^{r_2,c_2}h_2(Y_{\mathbf{i}_2,\mathbf{j}_2}) \mid \mathcal{A}_{\{i_2\}, \emptyset}]] \\
        &= 0,
    \end{split}
    \end{equation*}
    since $\E[p^{r_2,c_2}h_2(Y_{\mathbf{i}_2,\mathbf{j}_2}) \mid \mathcal{A}_{\{i_2\}, \emptyset}] = 0$ from Lemma~\ref{lem:proj_smaller_cond_exp}.

    We have proven both properties, since the case $(r_1,c_1) \neq (r_2,c_2)$ corresponds to the first property and the case $(r_1,c_1) = (r_2,c_2)$ corresponds to the second property.
\end{proof}

\subsection{Variance decomposition}

Here, we state and prove Corollary~\ref{cor:ortho_ustats}, of which Corollary~\ref{cor:cov_ustats}, providing a variance decomposition for $U$-statistics, is a direct consequence. 

\begin{corollary}
    Let $h_1$ and $h_2$ two kernel functions of respective sizes $p_1 \times q_1$ and $p_2 \times q_2$. Let $(0,0) \le (r_1,c_1) < (p_1,q_1) $ and $ (0,0) \le (r_2,c_2) < (p_2,q_2)$.
    \begin{enumerate}
        \item If $(r_1,c_1) \neq (r_2,c_2)$, then 
        \begin{equation*}
            \Cov(P^{r_1,c_1}_{m,n}h_1(Y), P^{r_2,c_2}_{m,n}h_2(Y)) = 0.
        \end{equation*}
        \item If $(r_1,c_1) = (r_2,c_2) = (r,c)$, then 
        \begin{equation*}
        \begin{split}
            \Cov&(P^{r,c}_{m,n}h_1(Y), P^{r,c}_{m,n}h_2(Y)) \\
            &= \binom{m}{r}^{-1} \binom{n}{c}^{-1} \Cov(p^{r,c}h_1(Y_{\llbracket r \rrbracket,\llbracket c \rrbracket}), p^{r,c}h_2(Y_{\llbracket r \rrbracket,\llbracket c \rrbracket})).
        \end{split}
        \end{equation*}
    \end{enumerate}
    \label{cor:ortho_ustats}
\end{corollary}

\begin{proof}
First, we see that for some $(r_1, c_1)$ and $(r_2, c_2)$,
    \begin{equation*}
    \begin{split}
        \Cov(&P^{r_1,c_1}_{m,n}h_1(Y), P^{r_2,c_2}_{m,n}h_2(Y)) = \left[\binom{m}{r_1}\binom{m}{r_2}\binom{n}{c_1}\binom{n}{c_2}\right]^{-1} \\
        &\sum_{\substack{\mathbf{i}_1 \in \mathcal{P}_{r_1}(\llbracket m \rrbracket) \\ \mathbf{j}_1 \in \mathcal{P}_{c_1}(\llbracket n \rrbracket)}} \sum_{\substack{\mathbf{i}_2 \in \mathcal{P}_{r_2}(\llbracket m \rrbracket) \\ \mathbf{j}_2 \in \mathcal{P}_{c_2}(\llbracket n \rrbracket)}} \Cov(p^{r_1,c_1} h_1(Y_{\mathbf{i}_1, \mathbf{j}_1}), p^{r_2,c_2} h_2(Y_{\mathbf{i}_2, \mathbf{j}_2})).
    \end{split}
    \end{equation*}
    If $(r_1,c_1) \neq (r_2,c_2)$, then from Proposition~\ref{prop:ortho_proj}, all the covariance terms $$\Cov(p^{r_1,c_1} h_1(Y_{\mathbf{i}_1, \mathbf{j}_1}), p^{r_2,c_2} h_2(Y_{\mathbf{i}_2, \mathbf{j}_2})) = 0,$$ so $\Cov(P^{r_1,c_1}_{m,n}h_1(Y), P^{r_2,c_2}_{m,n}h_2(Y)) = 0$ and that is the first part of the corollary.

    If $(r_1,c_1) = (r_2,c_2) = (r,c)$, then from Proposition~\ref{prop:ortho_proj}, the terms $$\Cov(p^{r,c} h_1(Y_{\mathbf{i}_1, \mathbf{j}_1}), p^{r,c} h_2(Y_{\mathbf{i}_2, \mathbf{j}_2})) = 0$$ if $(\mathbf{i}_1,\mathbf{j}_1) \neq (\mathbf{i}_2,\mathbf{j}_2)$. Using this fact and the exchangeability of $Y$, we have
    \begin{equation*}
    \begin{split}
            \Cov(&P^{r,c}_{m,n}h_1(Y), P^{r,c}_{m,n}h_2(Y)) \\
            &= \left[\binom{m}{r}\binom{n}{c}\right]^{-2} \sum_{\substack{\mathbf{i}_1 \in \mathcal{P}_{r}(\llbracket m \rrbracket) \\ \mathbf{j}_1 \in \mathcal{P}_{c}(\llbracket n \rrbracket)}} \sum_{\substack{\mathbf{i}_2 \in \mathcal{P}_{r}(\llbracket m \rrbracket) \\ \mathbf{j}_2 \in \mathcal{P}_{c}(\llbracket n \rrbracket)}} \Cov(p^{r_1,c_1} h_1(Y_{\mathbf{i}_1, \mathbf{j}_1}), p^{r_2,c_2} h_2(Y_{\mathbf{i}_2, \mathbf{j}_2})) \\
            &= \left[\binom{m}{r}\binom{n}{c}\right]^{-2} \sum_{\substack{\mathbf{i}_1 \in \mathcal{P}_{r}(\llbracket m \rrbracket) \\ \mathbf{j}_1 \in \mathcal{P}_{c}(\llbracket n \rrbracket)}} \Cov(p^{r_1,c_1} h_1(Y_{\mathbf{i}_1, \mathbf{j}_1}), p^{r_2,c_2} h_2(Y_{\mathbf{i}_1, \mathbf{j}_1})) \\
            &= \left[\binom{m}{r}\binom{n}{c}\right]^{-1} \Cov(p^{r,c}h_1(Y_{\llbracket r \rrbracket,\llbracket c \rrbracket}), p^{r,c}h_2(Y_{\llbracket r \rrbracket,\llbracket c \rrbracket})),
    \end{split}
    \end{equation*}
    which proves the second part of the corollary.
\end{proof}

\section{Additional results presented in Section~\ref{sec:asymptotic_normality} and proofs}
\label{app:asymptotic_normality}

\subsection{Proofs for the asymptotic normality of $U$-statistics}

\begin{proof}[Proof of Lemma~\ref{lem:asymptotic_normality_of_main_terms}]
    We only prove the first convergence result, since the second can be deduced by analogy.

    By definition,
    \begin{equation*}
        \hproj{1,0}{\{i\},\emptyset}h = \psi^{1,0}_{(\{i\},\emptyset)}h - \mathbb{E}[h(Y_{\llbracket p \rrbracket, \llbracket q \rrbracket)})].
    \end{equation*}

    First, we see that the $\hproj{1,0}{\{i\},\emptyset}h(Y)$ only depends on $\xi_i$, so all the $\hproj{1,0}{\{i\},\emptyset}h(Y)$ are independent. Because $h$ is symmetric and $Y$ is RCE, they are also identically distributed.

    By the tower rule, we also have $\E[\psi^{1,0}_{(\{i\},\emptyset)}h] = \mathbb{E}[h(Y_{\llbracket p \rrbracket, \llbracket q \rrbracket})]$ so $$\E[\hproj{1,0}{\{i\},\emptyset}h] = 0.$$
    
    Finally, we have, from the Aldous-Hoover representation theorem, then Cauchy-Schwarz's inequality, and the fact that $\E[h(Y_{\llbracket p \rrbracket, \llbracket q \rrbracket})^2] < \infty$,
    \begin{equation*}
    \begin{split}
        \V[\hproj{1,0}{\{i\},\emptyset}h] &= v^{1,0}_h \\
        &= \V[\E[h(Y_{\llbracket p \rrbracket, \llbracket q \rrbracket}) \mid \xi_i]] \\
        &= \Cov(h(Y_{(1,...,p;1,...,q)}), h(Y_{(1,p+1,...,2p-1;1,...,q)})) \\
        &< \V[h(Y_{\llbracket p \rrbracket, \llbracket q \rrbracket})] \\
        &< \infty.
    \end{split}
    \end{equation*}

    The classical CLT gives the desired result.
\end{proof}

\begin{proof}[Proof of Lemma~\ref{lem:clt_residue}]
For all $\mathbf{i} \in \mathcal{P}_{r}(\llbracket m \rrbracket)$ and $\mathbf{j} \in \mathcal{P}_{c}(\llbracket n \rrbracket)$ and $(r,c) \in \mathbb{N}^2$, we have $\E[\psi^{r,c}_{(\mathbf{i}, \mathbf{j})} h] = \E[h(Y_{\llbracket p \rrbracket, \llbracket q \rrbracket})]$.

By recursion, we have $\E[A_N] = \E[P^{r,c}_{N}h] = 0$ for all $(r,c) > (0,0)$ since
    \begin{equation*}
        \begin{split}
            \E[P^{r,c}_{N}h] &= \E[\hproj{r,c}{\mathbf{i}, \mathbf{j}} h] \\
            &= \E[\psi^{r,c}_{(\mathbf{i}, \mathbf{j})} h] - \sum_{(0,0) \le (r',c') < (r,c)} \sum_{\substack{\mathbf{i}' \in \mathcal{P}_{r'}(\mathbf{i}) \\ \mathbf{j}' \in \mathcal{P}_{c'}(\mathbf{j})}} \E[\hproj{r',c'}{\mathbf{i}', \mathbf{j}'}h].
        \end{split}
    \end{equation*}

    Now, Corollary~\ref{cor:cov_ustats} imply that $\Cov(P^{r,c}_{N}h, P^{r',c'}_{N}h) = 0$ unless $(r,c) = (r',c')$. So
    \begin{equation}
    \begin{split}
        \V[A_N] &= N \sum_{\substack{(0,0) < (r,c) \le (p,q) \\ (r,c) \neq (1,0) \neq (0,1)}} \sum_{\substack{(0,0) < (r',c') \le (p,q) \\ (r',c') \neq (1,0) \neq (0,1)}} \binom{p}{r} \binom{p}{r'} \binom{q}{c} \binom{q}{c'} \Cov(P^{r,c}_{N}h, P^{r',c'}_{N}h) \\ 
        &= N \sum_{\substack{(0,0) < (r,c) \le (p,q) \\ (r,c) \neq (1,0) \neq (0,1)}} \binom{p}{r}^2 \binom{q}{c}^2 \V[P^{r,c}_{N}h] \\
        &= N \sum_{\substack{(0,0) < (r,c) \le (p,q) \\ (r,c) \neq (1,0) \neq (0,1)}} \binom{p}{r}^2 \binom{q}{c}^2 \binom{m_N}{r}^{-1} \binom{n_N}{c}^{-1} \V[p^{r,c}_{(\llbracket r \rrbracket,\llbracket c \rrbracket)}h].
    \end{split}
    \end{equation}
    From equation~\ref{eq:def_projection_U}, we see that $h(Y_{\llbracket p \rrbracket, \llbracket q \rrbracket})$ is a linear combination of all the $p^{r,c}_{(\mathbf{i}, \mathbf{j})}h$, for $0 \le r \le p, 0 \le c \le q, \mathbf{i} \in \mathcal{P}_{r}(\llbracket p \rrbracket), \mathbf{j} \in \mathcal{P}_{c}(\llbracket q \rrbracket)$. So $\E[h(Y_{\llbracket p \rrbracket, \llbracket q \rrbracket})^2] < \infty$ ensures that $\V[p^{r,c}_{(\llbracket r \rrbracket,\llbracket c \rrbracket)}h] < \infty$. Therefore $\V[A_N] = O(N m_N^{-r} n_N^{-c}) = O(N^{1-r-c})$.

    Thus, we deduce from Markov's inequality that $A_N \xrightarrow[N \rightarrow \infty]{\mathbb{P}} 0$.
\end{proof}

\subsection{The asymptotic normality result for functions of $U$-statistics}

In this section, we prove that functions of $U$-statistics are also asymptotically normal under non-degeneracy conditions. First, we show that the limiting distribution of a vector of $U$-statistics is a multivariate normal distribution under the condition that all the kernel functions are linearly independent. However, if the kernel functions are of different sizes, the notion of linear independence is unclear. We need to define the concept of kernel extension to enunciate the corollary.

\begin{definition}
    Let $h$ be a kernel function of size $p \times q$. Let $p' \ge p$ and $q' \ge q$. We define the extension of $h$ to the size $p' \times q'$ by $\tilde{h}$ such that for all $\mathbf{i}' \in \mathcal{P}_{p'}(\mathbb{N})$ and $\mathbf{j}' \in \mathcal{P}_{q'}(\mathbb{N})$,
    \begin{equation*}
        \tilde{h}(Y_{\mathbf{i}',\mathbf{j}'}) = \left[\binom{p'}{p} \binom{q'}{q}\right]^{-1} \sum_{\substack{\mathbf{i} \subset \mathcal{P}_p(\mathbf{i}') \\ \mathbf{j} \subset \mathcal{P}_q(\mathbf{j}')}} h(Y_{\mathbf{i},\mathbf{j}}).
    \end{equation*}
\end{definition}

The extension of a kernel shares some properties with its kernel, as shown in the following lemma. 
\begin{lemma}
    \begin{enumerate}
        \item $U^{\tilde{h}}_N = U^{h}_N$,
        \item $\E[\tilde{h}(Y_{\llbracket p' \rrbracket, \llbracket q' \rrbracket})] = \E[h(Y_{\llbracket p \rrbracket, \llbracket q \rrbracket})]$,
        \item $\E[\tilde{h}(Y_{\llbracket p' \rrbracket, \llbracket q' \rrbracket})^2] \le \E[h(Y_{\llbracket p \rrbracket, \llbracket q \rrbracket})^2]$.
    \end{enumerate}
    \label{lem:kernel_extension}
\end{lemma}

\begin{proof}
    The two first properties are straightforward. The third property stems from
    \begin{equation*}
        \begin{split}
            \E[&\tilde{h}(Y_{\llbracket p' \rrbracket, \llbracket q' \rrbracket})^2] \\
            &= \left[\binom{p'}{p} \binom{q'}{q}\right]^{-2} \sum_{\substack{\mathbf{i}_1 \subset \llbracket p' \rrbracket \\ \mathbf{j}_1 \subset \llbracket q' \rrbracket }} \sum_{\substack{\mathbf{i}_2 \subset \llbracket p' \rrbracket  \\ \mathbf{j}_2 \subset \llbracket q' \rrbracket}} \E[h(Y_{\mathbf{i}_1,\mathbf{j}_1})h(Y_{\mathbf{i}_2,\mathbf{j}_2})] \\
            &= \left[\binom{p'}{p} \binom{q'}{q}\right]^{-2} \sum_{\substack{\mathbf{i}_1 \subset \llbracket p' \rrbracket \\ \mathbf{j}_1 \subset \llbracket q' \rrbracket}} \sum_{\substack{\mathbf{i}_2 \subset \llbracket p' \rrbracket \\ \mathbf{j}_2 \subset \llbracket q' \rrbracket}} \left(\E[h(Y_{\llbracket p \rrbracket, \llbracket q \rrbracket})^2] \right. \\
            &\quad \left. - \frac{1}{2} \E[(h(Y_{\mathbf{i}_1,\mathbf{j}_1}) - h(Y_{\mathbf{i}_2,\mathbf{j}_2}))^2]\right) \\
            &\le \left[\binom{p'}{p} \binom{q'}{q}\right]^{-2} \sum_{\substack{\mathbf{i}_1 \subset \llbracket p' \rrbracket \\ \mathbf{j}_1 \subset \llbracket q' \rrbracket}} \sum_{\substack{\mathbf{i}_2 \subset \llbracket p' \rrbracket \\ \mathbf{j}_2 \subset \llbracket q' \rrbracket}} \E[h(Y_{\llbracket p \rrbracket, \llbracket q \rrbracket})^2] \\
            &\le \E[h(Y_{\llbracket p \rrbracket, \llbracket q \rrbracket})^2].
        \end{split}
    \end{equation*}
\end{proof}

With this definition, we can define the linear independence of kernel functions needed for the following corollary as the linear independence of their kernel extensions. Denote $c^{r,c}_{h_k,h_\ell} := \Cov\left(\psi^{r,c}_{(\llbracket r \rrbracket, \llbracket c \rrbracket)} h_k, \psi^{r,c}_{(\llbracket r \rrbracket, \llbracket c \rrbracket)} h_\ell \right)$.

\begin{corollary}
    Let $Y$ be a dissociated RCE matrix. Let $(h_1, h_2, ..., h_D)$ be a vector of kernel functions of respective sizes $p_1 \times q_1, p_2 \times q_2, ..., p_D \times q_D$ such that 
  \begin{enumerate}
      \item Theorem~\ref{th:asymptotic_normality} applies for each kernel function, i.e. $\mathbb{E}[h_k(Y_{\llbracket p_k \rrbracket, \llbracket q_k \rrbracket})^2] < \infty$ and $U^{h_k}_\infty$ and $V^{h_k}$ are as defined in Theorem~\ref{th:asymptotic_normality} for each kernel $h_k$, $1 \le k \le D$,
      \item for $t \in \mathbb{R}^D$, $\sum_{k=1}^D t_k \tilde{h}_k \equiv 0$ if and only if $t = (0, ..., 0)$, where for $1 \le k \le D$, $\tilde{h}_k$ is the extension of $h_k$ to   size $\max_k(p_k) \times \max_k(q_k)$.
  \end{enumerate} 
  Then
  \begin{equation*}
      \sqrt{N}\left(\begin{pmatrix} U^{h_1}_N  \\ U^{h_2}_N \\ ... \\ U^{h_D}_N 
    \end{pmatrix}-\begin{pmatrix} U^{h_1}_\infty  \\ U^{h_2}_\infty \\ ... \\ U^{h_D}_\infty 
    \end{pmatrix}\right) \xrightarrow[N \rightarrow \infty]{\mathcal{D}} \mathcal{N}(0, \Sigma^{h_1,...,h_D}),
  \end{equation*}
  with
  \begin{equation*}
      \Sigma^{h_1,...,h_D} = \left(C^{h_k, h_\ell} \right)_{1 \le k,\ell \le D},
  \end{equation*}
  where $C^{h_k,h_\ell} = \frac{p^2}{\rho} c^{1,0}_{h_k,h_\ell} + \frac{q^2}{1-\rho} c^{0,1}_{h_k,h_\ell}$ for all $1 \le k,\ell \le D$ (and $C^{h_k,h_k} = V^{h_k}$).
  \label{cor:joint_asymptotic_normality}
\end{corollary}

\begin{proof}
First, we show that  $\lim_{N \rightarrow +\infty} N \Cov(U^{h_k}_N, U^{h_\ell}_N)$

Let $(Z^{h_k})_{1 \le k \le D}$ be a vector of random variables following a centered multivariate Gaussian distribution with covariance matrix $\Sigma^{h_1,...,h_D}$ defined in the theorem. Then $Z^{h_k} \sim \mathcal{N}(0, V^{h_k})$ for all $1 \le k \le D$ and $\Cov(Z^{h_k}, Z^{h_\ell}) = C^{h_k, h_\ell}$ for all $1 \le k \le D$ and $1 \le \ell \le D$.
  
  Denote $p := \max_k(p_k)$ and $q := \max_k(q_k)$. For some $t = (t_1, t_2, ..., t_n) \in \mathbb{R}^n$, we set $\tilde{h}_t := t_1 \tilde{h}_1 + t_2 \tilde{h}_2 + ... + t_n \tilde{h}_D$. $\tilde{h}_t$ is a kernel function of size $p \times q$. 
  
  First, assume that $t \neq (0,...,0)$. Then by hypothesis, $\tilde{h}_t \not\equiv 0$, therefore Lemma~\ref{lem:kernel_extension} implies that $\sum_{k=1}^D t_k U^{h_k}_N = \sum_{k=1}^D t_k U^{\tilde{h}_k}_N = U^{\tilde{h}_t}_N$, the $U$-statistic with quadruplet kernel $\tilde{h}_t$ (of size $p \times q$). Using Cauchy-Schwarz inequality and the fact that from Lemma~\ref{lem:kernel_extension}, $\mathbb{E}[\tilde{h}_k(Y_{\llbracket p \rrbracket, \llbracket q \rrbracket})^2] \le \mathbb{E}[h_k(Y_{\llbracket p_k \rrbracket, \llbracket q_k \rrbracket})^2] < \infty$ for all $1 \le k \le D$, we have furthermore 
  \begin{equation*}
  \begin{split}
      \mathbb{E}[&\tilde{h}_t(Y_{\llbracket p \rrbracket, \llbracket q \rrbracket})^2] \\
      =&~ \sum_{k=1}^n t_k^2 \mathbb{E}[\tilde{h}_k(Y_{\llbracket p \rrbracket, \llbracket q \rrbracket})^2] + 2 \sum_{1\le k\neq\ell \le D} t_k t_\ell \mathbb{E}[\tilde{h}_k(Y_{\llbracket p \rrbracket, \llbracket q \rrbracket})\tilde{h}_\ell(Y_{\llbracket p \rrbracket, \llbracket q \rrbracket})], \\
      \le&~ \sum_{k=1}^n t_k^2 \mathbb{E}[\tilde{h}_k(Y_{\llbracket p \rrbracket, \llbracket q \rrbracket})^2] + 2 \sum_{1\le k\neq\ell \le D} t_k t_\ell \sqrt{\mathbb{E}[\tilde{h}_k(Y_{\llbracket p \rrbracket, \llbracket q \rrbracket})^2]\mathbb{E}[\tilde{h}_\ell(Y_{\llbracket p \rrbracket, \llbracket q \rrbracket})^2]}, \\
      <&~ \infty.
  \end{split}
  \end{equation*}
  Therefore, Theorem~\ref{th:asymptotic_normality} also applies for $U^{\tilde{h}_t}_N$ and $\sqrt{N} (U^{\tilde{h}_t}_N - U^{\tilde{h}_t}_\infty) \xrightarrow[N \rightarrow \infty]{\mathcal{D}} \mathcal{N}(0, V^{\tilde{h}_t})$, where
  \begin{itemize}
      \item $U^{\tilde{h}_t}_\infty = \sum_{k=1}^D t_k U^{h_k}_\infty$, 
      \item $V^{\tilde{h}_t} = t^T \Sigma^{h_1,...,h_D} t$.
  \end{itemize} 
  The second point comes from the fact that $$V^{\tilde{h}_t} = \lim_{N \rightarrow +\infty} N \sum_{k=1}^D \sum_{\ell=1}^D t_k t_\ell \Cov(U^{h_k}_N, U^{h_\ell}_N)$$ and by Corollary~\ref{cor:cov_ustats},  $\lim_{N \rightarrow +\infty} N \Cov(U^{h_k}_N, U^{h_\ell}_N) = \frac{p^2}{c} c^{1,0}_{h_k, h_\ell} + \frac{q^2}{1-c} c^{0,1}_{h_k, h_\ell} = \Sigma^{h_1,...,h_D}_{k\ell}$. Therefore, we can conclude that $\sqrt{N} \sum_{k=1}^D t_k (U^{h_k}_N - U^{h_k}_\infty) = \sqrt{N} (U^{\tilde{h}_t}_N - U^{\tilde{h}_t}_\infty) \xrightarrow[N \rightarrow \infty]{\mathcal{D}} \sum_{k=1}^D t_k Z^{h_k}$.
  
  Now assume that $t = (0,...,0)$. Then $\tilde{h}_t \equiv 0$ so $U^{\tilde{h}_t}_N = 0 = \sum_{k=1}^D t_k Z^{h_k}$. Therefore, $\sqrt{N} \sum_{k=1}^D t_k (U^{h_k}_N - U^{h_k}_\infty) = \sqrt{N} (U^{\tilde{h}_t}_N - U^{\tilde{h}_t}_\infty) \xrightarrow[N \rightarrow \infty]{\mathcal{D}} \sum_{k=1}^D t_k Z^{h_k}$ is still true. 
  
  We have proven that $\sqrt{N} \sum_{k=1}^D t_k (U^{h_k}_N - U^{h_k}_\infty) \xrightarrow[N \rightarrow \infty]{\mathcal{D}} \sum_{k=1}^D t_k Z^{h_k}$ for all $t \in \mathbb{R}^n$, so we can finally apply the Cramér-Wold theorem (Theorem 29.4 of~\cite{billingsley1995probability}) which states that $\sqrt{N}\left(U^{h_k}_N - U^{h_k}_\infty\right)_{1\le k \le D}$ converges jointly in distribution to $(Z^{h_k})_{1 \le k \le D}$, which is a centered multivariate Gaussian with covariance matrix $\Sigma^{h_1,...,h_D}$.
\end{proof}

Although, this corollary requires the kernel functions of $(h_1, h_2, ..., h_D)$ to be linearly independent, the corresponding $U$-statistics are not independent random variables, even asymptotically, because $\Sigma^{h_1,...,h_D}$ is not a diagonal matrix. One consequence of this corollary is that Theorem~\ref{th:asymptotic_normality} can be extended to differentiable functions of $U$-statistics.

\begin{corollary}
    Let $h_1, ..., h_D$ be $D$ kernel functions such that Corollary~\ref{cor:joint_asymptotic_normality} applies and $\Sigma^{h_1,...,h_D} = \left(C^{h_k, h_\ell} \right)_{1 \le k,\ell \le D}$. Denote $\theta = (U^{h_1}_\infty, ..., U^{h_D}_\infty)$. Let $g : \mathbb{R}^d \rightarrow \mathbb{R}$ be a differentiable function at $\theta$. Denote $\nabla g$ the gradient of g and $V^\delta := \nabla g(\theta)^T \Sigma^{h_1,...,h_D} \nabla g(\theta)$. If $V^\delta > 0$, then
    \begin{equation*}
        \sqrt{N} (g(U^{h_1}_N, ..., U^{h_D}_N) - g(\theta)) \xrightarrow[N \rightarrow \infty]{\mathcal{D}} \mathcal{N}(0, V^\delta).
    \end{equation*}
  \label{cor:delta_method}
\end{corollary}

\begin{proof}
     The first-order Taylor expansion of $g$ at $\theta$ is written
     \begin{equation*}
         g(U^{h_1}_N, ..., U^{h_D}_N) - g(\theta) = \nabla g(\theta)^T \left((U^{h_1}_N, ..., U^{h_D}_N)-\theta\right) + o_P\left( \lvert \lvert (U^{h_1}_N, ..., U^{h_D}_N)-\theta \rvert \rvert \right).
     \end{equation*}
     From Corollary~\ref{cor:joint_asymptotic_normality}, $\sqrt{N} \left((U^{h_1}_N, ..., U^{h_D}_N)-\theta\right)$ converges to a multivariate normal distribution with asymptotic covariance matrix $\Sigma^{h_1,...,h_D}$, so the delta method (see Theorem 3.1 of~\cite{van2000asymptotic}) can be applied to prove this proposition.
\end{proof}

\section{Additional results presented in Section~\ref{sec:variance_estimator} and proofs}
\label{app:variance_estimator}

\subsection{Some useful notations and results}

Consider the sets $\mathcal{S}^{p,q}_{N}$ and $\mathcal{S}^{p,q}_{N , (\mathbf{\underline{i}}, \mathbf{\underline{j}}) }$ defined in Section~\ref{subsec:notations}. It is obvious that
\begin{equation*}
    \Card(\mathcal{S}^{p,q}_{N}) = \binom{m_N}{p} \binom{n_N}{q}
 \quad \quad \mbox{and} \quad \quad
    \Card(\mathcal{S}^{p,q}_{N , (\mathbf{\underline{i}}, \mathbf{\underline{j}}) }) = \binom{m_N - \underline{p}}{p-\underline{p}} \binom{n_N - \underline{q}}{q - \underline{q}}.
\end{equation*}

Let $I_K = (\mathbf{\underline{i}}_1, ..., \mathbf{\underline{i}}_K)$ be a $K$-uplet of sets of row indices and $J_K = (\mathbf{\underline{j}}_1, ..., \mathbf{\underline{j}}_K)$ a $K$-uplet of subsets of column indices. Denote
\begin{equation*}
    \mathcal{T}^{p,q}_{N,(I_K,J_K)} := \mathcal{S}^{p,q}_{N , (\mathbf{\underline{i}}_1, \mathbf{\underline{j}}_1) } \times \mathcal{S}^{p,q}_{N , (\mathbf{\underline{i}}_2, \mathbf{\underline{j}}_2) } \times ... \times \mathcal{S}^{p,q}_{N , (\mathbf{\underline{i}}_K, \mathbf{\underline{j}}_K) }.
\end{equation*}
In the rest of the paper, we will often see averages of the type 
\begin{equation}
    T^{p,q}_N(I_K,J_K) := \frac{1}{\prod_{k=1}^K \Card( \mathcal{T}^{p,q}_{N,(I_K,J_K)})} \sum_{\mathcal{T}^{p,q}_{N,(I_K,J_K)}} \E[X_{\mathbf{i}_1, \mathbf{j}_1}X_{\mathbf{i}_2, \mathbf{j}_2}...X_{\mathbf{i}_K, \mathbf{j}_K}].
    \label{eq:def_t}
\end{equation}
As a remark, $T^{p,q}_N((\emptyset),(\emptyset)) = U^h_N$.

By exchangeability, the quantities $\E[X_{(\mathbf{i}_1, \mathbf{j}_1)}X_{(\mathbf{i}_2, \mathbf{j}_2)}...X_{(\mathbf{i}_K, \mathbf{j}_K)}]$ do not depend on the row indices that do not belong to any pairwise intersection of the $(\mathbf{i}_k)_{1 \le k \le K}$. The same holds for column indices that do not belong to any pairwise intersection of the $(\mathbf{j}_k)_{1 \le k \le K}$. Therefore, assuming $m_N \ge \Card(\cup_{k=1}^{K} \mathbf{\underline{i}}_k)$ and $n_N \ge \Card(\cup_{k=1}^{K} \mathbf{\underline{j}}_k)$, we can define
\begin{equation}
    \alpha(I_K,J_K) := \E[X_{\mathbf{\bar{\underline{i}}}_1, \mathbf{\bar{\underline{j}}}_1}X_{\mathbf{\bar{\underline{i}}}_2,\mathbf{\bar{\underline{j}}}_2}...X_{\mathbf{\bar{\underline{i}}}_K, \mathbf{\bar{\underline{j}}}_K}]
    \label{eq:alpha}
\end{equation}
where for $1 \le k \le K$, the $p$-uplet $\mathbf{\bar{\underline{i}}}_k$ only consists of elements of $\mathbf{\underline{i}}_k$ and elements that are not in any of the other $\mathbf{\underline{i}}_{k'}$, i.e. the $\mathbf{\bar{\underline{i}}}_k$ are of the form $\mathbf{\bar{\underline{i}}}_k = \mathbf{\underline{i}}_k \cup \tilde{\mathbf{i}}_k$ where $\cap_{k=1}^K \tilde{\mathbf{i}}_k = \emptyset$ and $(\cup_{k=1}^K \tilde{\mathbf{i}}_k) \cap (\cup_{k=1}^K \underline{\mathbf{i}}_k) = \emptyset$. 

The following lemma will be helpful in later proofs as it provides the asymptotic behavior of $T^{p,q}_N(I_K,J_K)$. It shows that these averages can be reduced to one dominant expectation term given by $\alpha(I_K,J_K)$ and a remainder vanishing as $N$ grows.

\begin{lemma}
Let $I_K = (\mathbf{\underline{i}}_1, ..., \mathbf{\underline{i}}_K)$ and $J_K = (\mathbf{\underline{j}}_1, ..., \mathbf{\underline{j}}_K)$ be $K$-uplets of respectively row and column indices. Let $\alpha(I_K, J_K)$ defined by~\eqref{eq:alpha}. We have
    \begin{equation*}
        T^{p,q}_N(I_K,J_K) = \alpha(I_K,J_K) + O\left(m_N^{-1} + n_N^{-1} \right).
    \end{equation*}
    \label{lem:expectation_combi}
\end{lemma}

\begin{proof}
    Let:
\begin{itemize}
    \item $\underline{p}_k := \Card(\mathbf{\underline{i}}_k)$ and $\underline{q}_k := \Card(\mathbf{\underline{j}}_k)$, for $1 \le k \le K$,
    \item $\underline{P} := \sum_{k=1}^K \underline{p}_k$ and $\underline{Q} := \sum_{k=1}^K \underline{q}_k$,
    \item $\bar{p} := \Card(\cup_{k=1}^K \mathbf{\underline{i}}_k)$ and $\bar{q} := \Card(\cup_{k=1}^K \mathbf{\underline{j}}_k)$.
\end{itemize}
The sum $T^{p,q}_N(I_K,J_K)$ defined by equation~\eqref{eq:def_t} is a sum over  $ T_N$ expectation terms where
    \begin{equation}
        T_N := \Card(\mathcal{T}^{p,q}_{N,(I_K,J_K)}) = \prod_{k=1}^K \Card( \mathcal{S}^{p,q}_{N , (\mathbf{\underline{i}}_k, \mathbf{\underline{j}}_k) }) = \prod_{k=1}^K \binom{m_N - \underline{p}_k}{p-\underline{p}_k} \binom{n_N - \underline{q}_k}{q - \underline{q}_k}
            \label{eq:lem_tn}.
    \end{equation}
These expectation terms $\E[X_{\mathbf{i}_1, \mathbf{j}_1}...X_{\mathbf{i}_K, \mathbf{j}_K}]$ only depend on the number of times elements appear in pairwise intersections between the $(\mathbf{i}_k, \mathbf{j}_k)$, $1 \le k \le K$. In particular, we have 
\begin{equation*}
    \E[X_{\mathbf{i}_1, \mathbf{j}_1}...X_{\mathbf{i}_K, \mathbf{j}_K}] = \alpha(I_K,J_K)
\end{equation*}
when all the $\mathbf{i}_k$ contains the elements of $\mathbf{\underline{i}}_k$ and all the other elements do not appear in any other $\mathbf{\underline{i}}_{k'}$ (see equation~\ref{eq:alpha}). Denote $A_N$ the number of terms of $T^{p,q}_N(I_K,J_K)$ where $\E[X_{\mathbf{i}_1, \mathbf{j}_1}...X_{\mathbf{i}_K, \mathbf{j}_K}] = \alpha(I_K,J_K)$, then
    
\begin{equation}
\begin{split}
    T^{p,q}_N(&I_K,J_K) \\
    &= \frac{A_N}{T_N} \alpha(I_K,J_K) + \frac{1}{T_N} \sum_{\substack{\mathcal{T}^{p,q}_{N,(I_K,J_K)} \\\E[X_{\mathbf{i}_1, \mathbf{j}_1}...X_{\mathbf{i}_K, \mathbf{j}_K}] \neq\alpha(I_K,J_K) }} \E[X_{\mathbf{i}_1, \mathbf{j}_1}X_{\mathbf{i}_2, \mathbf{j}_2}...X_{\mathbf{i}_K, \mathbf{j}_K}].
\end{split}
    \label{eq:lem_tnpq}
\end{equation}

    Using Jensen's inequality, H{\"o}lder's inequality, and the exchangeability of the submatrices, we have for all the expectation terms,
    \begin{equation*}
    \begin{split}
        0 \le \lvert \E[X_{\mathbf{i}_1, \mathbf{j}_1}...X_{\mathbf{i}_K, \mathbf{j}_K}]\rvert \le \prod_{k=1}^K \E[\lvert X_{\mathbf{i}_k, \mathbf{j}_k}\rvert^K]^{\frac{1}{K}} = \E[\lvert X_{\llbracket p \rrbracket, \llbracket q \rrbracket}\rvert^K].
    \end{split}
    \end{equation*}
    In particular, this holds for all $(T_N -A_N)$ terms of the remaining sum in Equation~\eqref{eq:lem_tnpq} when $\E[X_{\mathbf{i}_1, \mathbf{j}_1}...X_{\mathbf{i}_K, \mathbf{j}_K}] \neq \alpha(I_K,J_K)$, so 
    \begin{equation}
        0 \le \lvert T^{p,q}_N(I_K,J_K) - \frac{A_N}{T_N} \alpha(I_K,J_K) \rvert \le \left(1-\frac{A_N}{T_N}\right) \E[\lvert X_{\llbracket p \rrbracket, \llbracket q \rrbracket}\rvert^K].
        \label{eq:lem_tnpq_inequality}
    \end{equation}
    We need to calculate $T_N$ and $A_N$ to use this inequality and conclude the proof. We organize the rest of the proof in 3 parts.
    \begin{enumerate}
        \item We find an expression for $T_N$. Since for all $1 \le k \le K$, we have
        \begin{equation*}
            \binom{m_N - \underline{p}_k}{p-\underline{p}_k} \binom{n_N - \underline{q}_k}{q - \underline{q}_k} =  \frac{m_N^{p-\underline{p}_k} n_N^{q-\underline{q}_k}}{(p-\underline{p}_k)! (q - \underline{q}_k)!}  \left(1 +  O\left( m_N^{-1} +n_N^{-1}\right) \right),
        \end{equation*}
        then we have from Equation~\eqref{eq:lem_tn}
        \begin{equation}
            T_N = \left( \prod_{k=1}^K(p-\underline{p}_k)! (q-\underline{q}_k)!  \right)^{-1} m_N^{Kp - \underline{P}} n_N^{Kq - \underline{Q}} \left(1 +  O\left( m_N^{-1} +n_N^{-1}\right) \right).
            \label{eq:lem_exptn}
        \end{equation}
    \item We find an expression for $A_N$. The summation over $(\mathbf{i}_k, \mathbf{j}_k) \in \mathcal{S}^{p,q}_{N , (\mathbf{\underline{i}}_k, \mathbf{\underline{j}}_k) }$ is in fact a sum over the $p - \underline{p}_k$ elements of $\mathbf{i}_k$ and $q - \underline{q}_k$ elements of $\mathbf{j}_k$ that are not restricted by $\mathbf{\underline{i}}_k$ and $\mathbf{\underline{j}}_k$. 
    \begin{itemize}
        \item We can pick the first $\mathbf{i}_1$ choosing the $p - \underline{p}_1$ unrestricted indices among the $m_N - \bar{p}$ values of $\{1,...,m_N\}$ excluding $\cup_{k=1}^K \mathbf{\underline{i}}_k$. The same follows for the pick of $\mathbf{j}_1$, so there are $\binom{m_N-\bar{p}}{p-\underline{p}_1} \binom{n_N-\bar{q}}{q-\underline{q}_1}$ possible picks for $(\mathbf{i}_1, \mathbf{j}_1)$. 
        \item The pick of $\mathbf{i}_2$ consists in choosing the $p - \underline{p}_2$ unrestricted indices among the $m_N - \bar{p} - (p - \underline{p}_1)$ values of $\{1,...,m_N\}$ excluding $\cup_{k=1}^K \mathbf{\underline{i}}_k$ and the elements already taken by $\mathbf{i}_1$. We deduce that there are $\binom{m_N-\bar{p} - (p - \underline{p}_1)}{p-\underline{p}_2} \binom{n_N-\bar{q} - (q - \underline{q}_1)}{q-\underline{q}_2}$ possible picks for $(\mathbf{i}_2, \mathbf{j}_2)$.
        \item Iteratively, for all $1 \le k \le K$, we find that there are $$\binom{m_N-\bar{p} - \sum_{k'= 1}^{k-1} (p - \underline{p}_{k'})}{p-\underline{p}_k} \binom{n_N-\bar{q} - \sum_{k'= 1}^{k-1} (q - \underline{q}_{k'}) }{q-\underline{q}_k}$$ possible picks for $(\mathbf{i}_k, \mathbf{j}_k)$.
    \end{itemize} 
    We deduce that the number of possible picks for all the $(\mathbf{i}_k, \mathbf{j}_k)$, $1 \le k \le K$ so that $\E[X_{\mathbf{i}_1, \mathbf{j}_1}...X_{\mathbf{i}_K, \mathbf{j}_K}] = \alpha(I_K,J_K)$ is 
    \begin{equation*}
        A_N = \prod_{k=1}^K \binom{m_N-\bar{p}- \sum_{k'=1}^{k-1} (p-\underline{p}_{k'})}{p-\underline{p}_k} \binom{n_N-\bar{q}- \sum_{k'=1}^{k-1} (q-\underline{q}_{k'})}{q-\underline{q}_k}.
    \end{equation*}

    But we see that, for $1 \le k \le K$, we have
    \begin{equation*}
    \begin{split}
        \binom{m_N-\bar{p}- \sum_{k'=1}^{k-1} (p-\underline{p}_{k'})}{p-\underline{p}_k}&\binom{n_N-\bar{q}- \sum_{k'=1}^{k-1} (q-\underline{q}_{k'})}{q-\underline{q}_k} \\
        &= \frac{m_N^{p-\underline{p}_k} n_N^{q-\underline{q}_k}}{(p-\underline{p}_k)!(q-\underline{q}_k)!}  \left(1 + O\left(m_N^{-1} + n_N^{-1}\right)\right).
    \end{split}
    \end{equation*}
    So we have
    \begin{equation*}
        A_N = \left( \prod_{k=1}^K(p-\underline{p}_k)! (q-\underline{q}_k)!  \right)^{-1} m_N^{Kp - \underline{P}} n_N^{Kq - \underline{Q}} \left(1 +  O\left( m_N^{-1} +n_N^{-1}\right) \right).
    \end{equation*}
    \item Now with the expressions of $T_N$ and $A_N$, we can deduce that 
    \begin{equation*}
        T_N - A_N =  O\left( m_N^{Kp - \underline{P}} n_N^{Kq - \underline{Q}} (m_N^{-1} +n_N^{-1})\right),
    \end{equation*}
    so 
    \begin{equation*}
    \begin{split}
        1 - \frac{A_N}{T_N} &=  \frac{\left( \prod_{k=1}^K(p-\underline{p}_k)! (q-\underline{q}_k)!  \right) O\left( m_N^{Kp - \underline{P}} n_N^{Kq - \underline{Q}} (m_N^{-1} +n_N^{-1})\right)}{ m_N^{Kp - \underline{P}} n_N^{Kq - \underline{Q}} \left(1 +  O\left( m_N^{-1} +n_N^{-1}\right) \right)} \\
        &=  O\left( m_N^{-1} +n_N^{-1}\right).
    \end{split}
    \end{equation*}
    Therefore, we can finally conclude using Equation~\eqref{eq:lem_tnpq_inequality},
    \begin{equation*}
        \begin{split}
            T^{p,q}_N(I_K,J_K) &= \left(1+O\left( m_N^{-1} +n_N^{-1}\right)\right) \alpha(I_K,J_K) + O\left( m_N^{-1} +n_N^{-1}\right) \\
            &= \alpha(I_K,J_K) + O\left(m_N^{-1} + n_N^{-1} \right).
        \end{split}
    \end{equation*}
    \end{enumerate}
    
\end{proof}

\subsection{Proofs for the convergence of the conditional expectation estimators}

Here, we provide the proofs for Propositions~\ref{prop:cond_exp_unbiased} and~\ref{prop:cond_exp_consistent}, establishing the unbiasedness and asymptotic consistency of the conditional expectation estimators defined by~\eqref{eq:mu} and~\eqref{eq:nu}.

\begin{proof}[Proof of Proposition~\ref{prop:cond_exp_unbiased}]
The first result can be found directly by the definition of $\psi^{1,0}_{(\{i\}, \emptyset)} h$, since
    \begin{equation*}
    \begin{split}
        \E[\widehat{\mu}^{h,(i)}_N \mid \xi_i] &= \binom{m_N-1}{p-1}^{-1} \binom{n_N}{q}^{-1} \sum_{(\mathbf{i}, \mathbf{j}) \in \mathcal{S}^{p,q}_{N , (\{i\}, \emptyset) }} \E[h(Y_{\mathbf{i}, \mathbf{j}}) \mid \xi_i] \\
        &= \binom{m_N-1}{p-1}^{-1} \binom{n_N}{q}^{-1} \sum_{(\mathbf{i}, \mathbf{j}) \in \mathcal{S}^{p,q}_{N , (\{i\}, \emptyset) }} \psi^{1,0}_{(\{i\}, \emptyset)} h \\
        &= \psi^{1,0}_{(\{i\}, \emptyset)} h.
    \end{split}
    \end{equation*}
    The second result can be obtained analogously.
\end{proof}

\begin{proof}[Proof of Proposition~\ref{prop:cond_exp_consistent}]
    Let $N \in \mathbb{N}$ and $\mathcal{F}_N(Y) = \sigma((\widehat{\mu}^{h,(i)}_{K}(Y))_{K \ge N})$. Let $\Phi_N \in \mathbb{S}^{(i)}_{m_N} \times \mathbb{S}_{n_N}$ where $\mathbb{S}^{(i)}_{m_N}$ is the group of permutations $\sigma_{i}$ of $\llbracket m_N \rrbracket$ such that $\sigma_i(i) = i$. If $\Phi_N = (\sigma_{i}, \tau)$, denote $\Phi_N Y = (Y_{\sigma_{i}(k)\tau(j)})_{\substack{1 \le k \le m_N\\1 \le j \le n_N}}$. 
    
    First, we observe that $\widehat{\mu}^{h,(i)}_N(Y) = \widehat{\mu}^{h,(i)}_N(\Phi_N Y)$, so $\mathcal{F}_{N}(\Phi_N Y) = \mathcal{F}_{N}(Y)$. Therefore, by the exchangeability of $Y$, we have 
    \begin{equation*}
        Y \mid \mathcal{F}_N(Y) \overset{\mathcal{D}}{=} \Phi_N Y \mid \mathcal{F}_N(Y).
    \end{equation*}
    This assertion is true for all $N \in \mathbb{N}$ and $\Phi_N \in \mathbb{S}^{(i)}_{m_N} \times \mathbb{S}_{n_N}$. Now, note that for all $(\mathbf{i}, \mathbf{j})$ and $(\mathbf{i}', \mathbf{j}')$ elements of $\mathcal{S}^{p,q}_{N , (\{i\}, \emptyset) }$, we can always find a permutation $\Phi_N \in \mathbb{S}^{(i)}_{m_N} \times \mathbb{S}_{n_N}$ such that $h(\Phi_N Y_{\mathbf{i}, \mathbf{j}}) = h(Y_{\mathbf{i}', \mathbf{j}'})$. Thus, we have $\mathbb{E}[h(Y_{\mathbf{i}, \mathbf{j}}) \mid \mathcal{F}_{N}(Y)] = \mathbb{E}[h(Y_{\mathbf{i}', \mathbf{j}'}) \mid \mathcal{F}_{N}(Y)]$. Hence, for any $(\mathbf{i}, \mathbf{j}) \in \mathcal{S}^{p,q}_{N+1, (\{i\}, \emptyset) }$, we deduce that
    \begin{equation*}
    \begin{split}
        \mathbb{E}[\widehat{\mu}^{h,(i)}_N(Y) \mid \mathcal{F}_{N+1}(Y)] &= \mathbb{E}[h(Y_{(\mathbf{i}, \mathbf{j})}) \mid \mathcal{F}_{N+1}(Y)] \\ 
        &= \mathbb{E}[\widehat{\mu}^{h,(i)}_{N+1}(Y) \mid \mathcal{F}_{N+1}(Y)] \\
        &= \widehat{\mu}^{h,(i)}_{N+1}(Y).
    \end{split}
    \end{equation*}
    Therefore, $\widehat{\mu}^{h,(i)}_N(Y)$ is a backward martingale with respect to $\mathcal{F}_{N}(Y)$ (see Appendix~\ref{app:martingales}). By Theorem~\ref{th:martingale_convergence}, we have that $\widehat{\mu}^{h,(i)}_{N}(Y) \xrightarrow[N \rightarrow \infty]{a.s., L_1} \E[\widehat{\mu}^{h,(i)}_{1}(Y) \mid \mathcal{F}_{\infty}(Y)]$, where $\mathcal{F}_{\infty}(Y) = \bigcap_{N=1}^\infty \mathcal{F}_{N}(Y)$.

    Finally, $\mathcal{F}_{\infty}(Y) = \sigma(\xi_i)$ so Proposition~\ref{prop:cond_exp_unbiased} implies $\E[\widehat{\mu}^{h,(i)}_{1}(Y) \mid \mathcal{F}_{\infty}(Y)] = \psi^{1,0}_{(\{i\}, \emptyset)} h(Y)$ and thus, $\widehat{\mu}^{h,(i)}_{N}(Y) \xrightarrow[N \rightarrow \infty]{a.s., L_1} \psi^{1,0}_{(\{i\}, \emptyset)} h(Y)$.
\end{proof}

\subsection{Results for the convergence of the variance estimators and proofs}

In this section, we state and prove Propositions~\ref{prop:est_var_unbias} and~\ref{prop:est_var_variance} ensuring that the variance estimators defined by~\eqref{eq:est_v10} and~\eqref{eq:est_v01} are asymptotically unbiased and have vanishing variances.

\begin{proposition}
    We have $\E[\widehat{v}^{h;1,0}_N] = v^{1,0}_h + O\left(N^{-1}\right)$ and $\E[\widehat{v}^{h;0,1}_N] = v^{0,1}_h + O\left(N^{-1}\right)$.
    As a consequence, $\widehat{v}^{h;1,0}_N$ and $\widehat{v}^{h;0,1}_N$ are asymptotically unbiased estimators for $v^{1,0}_h$ and $v^{0,1}_h$.
    \label{prop:est_var_unbias}
\end{proposition}

\begin{proof}
    In this proof, for the estimator defined by equation~\eqref{eq:mu}, we write $\widehat{\mu}^{(i)}_N$ instead of $\widehat{\mu}^{h,(i)}_N$ to simplify the notation without ambiguity. Notice that
    \begin{equation}
            \E[\widehat{v}^{h;1,0}_N] = \frac{1}{2} \E[(\widehat{\mu}^{(1)}_N - \widehat{\mu}^{(2)}_N)^2] = \E[(\widehat{\mu}^{(1)}_N)^2] - \E[\widehat{\mu}^{(1)}_N\widehat{\mu}^{(2)}_N].
        \label{eq:expect_variance}
    \end{equation}

    \begin{itemize}
        \item First, we calculate:
    \begin{equation*}
        \mathbb{E}[\widehat{\mu}^{(1)}_N(Y)^2] = T^{p,q}_N(I_2,J_2) 
    \end{equation*}
  where $I_2 = (\{1\}, \{1\})$ and $J_2 = (\emptyset, \emptyset)$. Applying Lemma~\ref{lem:expectation_combi}, with $\underline{P} = \Card(\{1\}) + \Card(\{1\}) = 2$, $\underline{Q} = \Card(\emptyset) + \Card(\emptyset) = 0$ and
  \begin{equation*}
  \begin{split}
      \alpha(I_2, J_2) &= \E[h(Y_{(1,...,p;1,...,q)})h(Y_{(1,p+1,...,2p-1;q+1,...,2q)})] \\
      &= \E[\E[h(Y_{(1,...,p;1,...,q)})h(Y_{(1,p+1,...,2p-1;q+1,...,2q)}) \mid \xi_1]] \\
      &= \E[\E[h(Y_{\llbracket p \rrbracket, \llbracket q \rrbracket}) \mid \xi_1]^2] \\
      &= \E[\psi^{1,0}_{(\{1\}, \emptyset)} h^2],
  \end{split}
  \end{equation*}
 we find
    \begin{equation}
    \begin{split}
        \mathbb{E}[(\widehat{\mu}^{(1)}_N)^2] =  \E[\psi^{1,0}_{(\{1\}, \emptyset)} h^2] + O\left(N^{-1}\right).
    \end{split}
    \label{eq:mu_squared}
    \end{equation}

    \item Next, we calculate:
    \begin{equation*}
        \mathbb{E}[\widehat{\mu}^{(1)}_N \widehat{\mu}^{(2)}_N] = T^{p,q}_N(I_2',J_2')
    \end{equation*}
  where $I_2' = (\{1\}, \{2\})$ and $J_2' = (\emptyset, \emptyset)$. Applying Lemma~\ref{lem:expectation_combi} with
  \begin{equation*}
  \begin{split}
      \alpha(I_2', J_2') &= \E[h(Y_{(1,3,...,p+1;1,...,q)})h(Y_{(2,p+2,...,2p;q+1,...,2q)})] \\
      &= \E[h(Y_{\llbracket p \rrbracket, \llbracket q \rrbracket})]^2 \\
      &= \E[\psi^{1,0}_{(\{1\}, \emptyset)} h]^2,
  \end{split}
  \end{equation*}
we find
    \begin{equation}
    \begin{split}
        \mathbb{E}[\widehat{\mu}^{(1)}_N \widehat{\mu}^{(2)}_N] &= \E[\psi^{1,0}_{(\{1\}, \emptyset)} h]^2 + O\left(N^{-1}\right).
    \end{split}
    \label{eq:mu_cross}
    \end{equation}
    \end{itemize}

    Finally, we can combine Equations~\eqref{eq:expect_variance},~\eqref{eq:mu_squared} and~\eqref{eq:mu_cross} to obtain
    \begin{equation*}
        \E[\widehat{v}^{h;1,0}_N] = \E[\psi^{1,0}_{(\{1\}, \emptyset)} h^2] - \E[\psi^{1,0}_{(\{1\}, \emptyset)} h]^2 + O\left(N^{-1}\right) = v^{1,0}_h + O\left(N^{-1}\right),
    \end{equation*}
    which proves the result.
\end{proof}

\begin{proposition}
    We have $\V[\widehat{v}^{h;1,0}_N] = O\left(N^{-1}\right)$ and $\V[\widehat{v}^{h;0,1}_N] = O\left(N^{-1}\right)$.
    \label{prop:est_var_variance}
\end{proposition}

\begin{proof}

    In this proof, we write $\widehat{\mu}^{(i)}_N$ instead of $\widehat{\mu}^{h,(i)}_N$ to simplify the notation without ambiguity. Notice that
    \begin{align}
        \E[(\widehat{v}^{h;1,0}_N)^2] &= \binom{m_N}{2}^{-2} \sum_{1 \le i_1 < i_2 \le m_N} \sum_{1 \le i'_1 < i'_2 \le m_N} \E\left[\frac{(\widehat{\mu}^{(i_1)}_N - \widehat{\mu}^{(i_2)}_N)^2(\widehat{\mu}^{(i'_1)}_N - \widehat{\mu}^{(i'_2)}_N)^2}{4}\right] \nonumber \\
        &= \frac{1}{4} \E[(\widehat{\mu}^{(1)}_N - \widehat{\mu}^{(2)}_N)^2(\widehat{\mu}^{(3)}_N - \widehat{\mu}^{(4)}_N)^2] + O\left(N^{-1}\right) \nonumber \\
        &= \mathbb{E}[(\widehat{\mu}^{(1)}_N)^2(\widehat{\mu}^{(2)}_N)^2] - 2 \mathbb{E}[(\widehat{\mu}^{(1)}_N)^2 \widehat{\mu}^{(2)}_N \widehat{\mu}^{(3)}_N] + \mathbb{E}[\widehat{\mu}^{(1)}_N\widehat{\mu}^{(2)}_N \widehat{\mu}^{(3)}_N \widehat{\mu}^{(4)}_N] \nonumber \\
        &\quad + O\left(N^{-1}\right). 
    \label{eq:squared_v_expectation}
    \end{align}
    Now, we calculate each of the three expectation terms in this equation.

    \begin{itemize}
        \item First, we calculate:
    \begin{equation*}
            \mathbb{E}[(\widehat{\mu}^{(1)}_N)^2(\widehat{\mu}^{(2)}_N)^2] = T^{p,q}_N(I_4, J_4),
    \end{equation*}
where $I_4 = (\{1\},\{1\},\{2\},\{2\})$ and $J_4 = (\emptyset,\emptyset,\emptyset,\emptyset)$. Applying Lemma~\ref{lem:expectation_combi} with
\begin{equation*}
\begin{split}
    \alpha&(I_4, J_4)\\
    &= \E[h(Y_{(1,3,...,p+1;1,...,q)}) h(Y_{(1,p+2,...,2p;q+1,...,2q)}) \\
    &\quad h(Y_{(2,2p+1,...,3p-1;2q+1,...,3q)}) h(Y_{(2,3p,...,4p-2;3q+1,...,4q)})] \\
    &= \E[h(Y_{(1,...,p;1,...,q)}) h(Y_{(1,p+1,...,2p-1;q+1,...,2q)})]^2 \\
    &= \E[\E[h(Y_{(1,...,p;1,...,q)})h(Y_{(1,p+1,...,2p-1;q+1,...,2q)}) \mid \xi_1]]^2 \\
    &= \E[\E[h(Y_{\llbracket p \rrbracket, \llbracket q \rrbracket}) \mid \xi_1]^2]^2 \\
    &= \E[\psi^{1,0}_{(\{1\}, \emptyset)} h^2]^2,
\end{split}
\end{equation*}
we find
    \begin{equation}
        \begin{split}
            \mathbb{E}[(\widehat{\mu}^{(1)}_N)^2(\widehat{\mu}^{(2)}_N)^2] &= \E[\psi^{1,0}_{(\{1\}, \emptyset)} h^2]^2 + O\left(N^{-1}\right).
        \end{split}
        \label{eq:mu_4}
    \end{equation}

\item Next, we calculate:
\begin{equation*}
    \mathbb{E}[(\widehat{\mu}^{(1)}_N)^2 \widehat{\mu}^{(2)}_N \widehat{\mu}^{(3)}_N] = T^{p,q}_N(I_4', J_4'),
\end{equation*}
where $I_4' = (\{1\},\{1\},\{2\},\{3\})$ and $J_4' = (\emptyset,\emptyset,\emptyset,\emptyset)$. Applying Lemma~\ref{lem:expectation_combi} with
\begin{equation*}
\begin{split}
    \alpha&(I_4', J_4') \\
    &= \E[h(Y_{(1,4,...,p+2;1,...,q)}) h(Y_{(1,p+3,...,2p+1;q+1,...,2q)}) \\
    &\quad h(Y_{(2,2p+2,...,3p;2q+1,...,3q)}) h(Y_{(3,3p+1,...,4p-1;3q+1,...,4q)})] \\
    &= \E[h(Y_{(1,...,p;1,...,q)}) h(Y_{(1,p+1,...,2p-1;q+1,...,2q)})]^2 \E[h(Y_{(1,...,p;1,...,q)})]^2 \\
    &= \E[\E[h(Y_{(1,...,p;1,...,q)})h(Y_{(1,p+1,...,2p-1;q+1,...,2q)}) \mid \xi_1]] \E[\psi^{1,0}_{(\{1\}, \emptyset)} h(Y)]^2\\
    &= \E[\E[h(Y_{\llbracket p \rrbracket, \llbracket q \rrbracket}) \mid \xi_1]^2] \E[\psi^{1,0}_{(\{1\}, \emptyset)} h]^2\\
    &= \E[\psi^{1,0}_{(\{1\}, \emptyset)} h(Y)^2] \E[\psi^{1,0}_{(\{1\}, \emptyset)} h]^2,
\end{split}
\end{equation*}
we find
    \begin{equation}
        \begin{split}
            \mathbb{E}[(\widehat{\mu}^{(1)}_N(Y))^2 \widehat{\mu}^{(2)}_N(Y) \widehat{\mu}^{(3)}_N(Y)] &= \E[\psi^{1,0}_{(\{1\}, \emptyset)} h(Y)^2] \E[\psi^{1,0}_{(\{1\}, \emptyset)} h(Y)]^2 \\
            &\quad + O\left(N^{-1}\right).
        \end{split}
        \label{eq:mu_211}
    \end{equation}

\item Now, we calculate:
\begin{equation*}
    \mathbb{E}[\widehat{\mu}^{(1)}_N(Y) \widehat{\mu}^{(2)}_N(Y) \widehat{\mu}^{(3)}_N(Y) \widehat{\mu}^{(4)}_N(Y)] = T^{p,q}_N(I_4'', J_4''),
\end{equation*}
where $I_4' = (\{1\},\{2\},\{3\},\{4\})$ and $J_4' = (\emptyset,\emptyset,\emptyset,\emptyset)$. Applying Lemma~\ref{lem:expectation_combi} with
\begin{equation*}
\begin{split}
    \alpha&(I_4'', J_4'') \\
    &= \E[h(Y_{(1,5,...,p+3;1,...,q)}) h(Y_{(2,p+4,...,2p+2;q+1,...,2q)}) \\
    &\quad h(Y_{(3,2p+3,...,3p+1;2q+1,...,3q)}) h(Y_{(4,3p+2,...,4p;3q+1,...,4q)})] \\
    &= \E[h(Y_{\llbracket p \rrbracket, \llbracket q \rrbracket})]^4 \\
    &= \E[\psi^{1,0}_{(\{1\}, \emptyset)} h(Y)]^4,
\end{split}
\end{equation*}
we find
    \begin{equation}
        \begin{split}
            \mathbb{E}[\widehat{\mu}^{(1)}_N(Y) \widehat{\mu}^{(2)}_N(Y) \widehat{\mu}^{(3)}_N(Y) \widehat{\mu}^{(4)}_N(Y)] &= \E[\psi^{1,0}_{(\{1\}, \emptyset)} h(Y)]^4 + O\left(N^{-1}\right).
        \end{split}
        \label{eq:mu_1111}
    \end{equation}
    
    \end{itemize}

    Finally, injecting the calculated expressions~\eqref{eq:mu_4},~\eqref{eq:mu_211} and~\eqref{eq:mu_1111} in~\eqref{eq:squared_v_expectation}, we obtain
    \begin{equation*}
    \begin{split}
        \E[(\widehat{v}^{h;1,0}_N)^2] &= \E[\psi^{1,0}_{(\{1\}, \emptyset)} h(Y)^2]^2 - 2\E[\psi^{1,0}_{(\{1\}, \emptyset)} h(Y)^2] \E[\psi^{1,0}_{(\{1\}, \emptyset)} h(Y)]^2 \\
        &+ \E[\psi^{1,0}_{(\{1\}, \emptyset)} h(Y)]^4 + O\left(N^{-1}\right) \\
        &= \V[\psi^{1,0}_{(\{1\}, \emptyset)} h(Y)^2]^2 + O\left(N^{-1}\right) \\
        &= (v^{1,0}_h)^2 + O\left(N^{-1}\right) \\
        &= \E[\widehat{v}^{h;1,0}_N]^2 + O\left(N^{-1}\right),
    \end{split}
    \end{equation*}
    where we have applied Proposition~\ref{prop:est_var_unbias} in the last step. This proves that $\V[\widehat{v}^{h;1,0}_N] = O\left(N^{-1}\right)$, concluding the proof.
\end{proof}

\section{Proofs of the results presented in Section~\ref{sec:examples}}
\label{app:examples}

\begin{lemma}
    Let $Y$ be a matrix sampled from a Poisson $W$-graph model. Let $\bar{w}$, $f$ and $g$ be defined as in Section~\ref{sub:example_graphon}. For the kernel functions defined in Table~\ref{tab:kernels_1}, we have
    \begin{itemize}
        \item $\E[h_{A,1}(Y_{(i_1, i_2; j_1, j_2)})] = \lambda^3 \iint \bar{w}(\xi,\eta)^2 d\xi d\eta$,
        \item $\E[h_{A,2}(Y_{(i_1, i_2; j_1, j_2)})] = \lambda^3 \iint \bar{w}(\xi,\eta) f(\xi)g(\eta) d\xi d\eta$,
        \item $\E[h_B(Y_{(i_1; j_1, j_2)})] = \lambda^2 \int f(\xi)^2 d\xi$,
        \item $\E[h_C(Y_{(i_1, i_2; j_1)})] = \lambda^2 \int g(\eta)^2 d\eta$,
        \item $\E[h_D(Y_{(i_1; j_1)})] = \lambda$.
    \end{itemize}
    \label{lem:graphon_kernels}
\end{lemma}

\begin{proof}
    The result for $h_D$ is straightforward. For the other kernel functions:
    \begin{equation*}
        \begin{split}
            \E[h_{A,1}(Y_{(i_1, i_2; j_1, j_2)})] &= \E[\E[Y_{i_1j_1} (Y_{i_1j_1} - 1)  Y_{i_2j_2} \mid \boldsymbol{\xi}, \boldsymbol{\eta}]] \\
            &= \E[\E[Y_{i_1j_1} (Y_{i_1j_1} - 1) \mid \boldsymbol{\xi}, \boldsymbol{\eta}]\E[Y_{i_2j_2} \mid \boldsymbol{\xi}, \boldsymbol{\eta}]] \\
            &= \E[\lambda^2 \bar{w}(\xi_{i_1}, \eta_{j_1})^2 \times \lambda] \\
            &= \lambda^3 \iint \bar{w}(\xi,\eta)^2 d\xi d\eta.
        \end{split}
    \end{equation*}
    
    \begin{equation*}
        \begin{split}
            \E[&h_{A,2}(Y_{(i_1, i_2; j_1, j_2)})] \\
            &= \E[\E[Y_{i_1j_1} Y_{i_1j_2} Y_{i_2j_2} \mid \boldsymbol{\xi}, \boldsymbol{\eta}]] \\
            &= \E[\E[Y_{i_1j_1} \mid \boldsymbol{\xi}, \boldsymbol{\eta}]\E[ Y_{i_1j_2} \mid \boldsymbol{\xi}, \boldsymbol{\eta}]\E[Y_{i_2j_2} \mid \boldsymbol{\xi}, \boldsymbol{\eta}]] \\
            &= \E[\lambda \bar{w}(\xi_{i_1}, \eta_{j_1}) \times \lambda \bar{w}(\xi_{i_1}, \eta_{j_2}) \times \lambda \bar{w}(\xi_{i_2}, \eta_{j_2})] \\
            &= \lambda^3 \iint \left[\bar{w}(\xi_{i_1}, \eta_{j_2}) \left(\int \bar{w}(\xi_{i_1}, \eta_{j_1}) d\eta_{j_1} \right) \left(\int \bar{w}(\xi_{i_2}, \eta_{j_2}) d\xi_{i_2} \right)\right] d\xi_{i_1} d\eta_{j_2} \\
            &= \lambda^3 \iint \bar{w}(\xi,\eta) f(\xi)g(\eta) d\xi d\eta.
        \end{split}
    \end{equation*}
    
    \begin{equation*}
        \begin{split}
            \E[h_B(Y_{(i_1; j_1, j_2)})] &= \E[\E[Y_{i_1j_1} Y_{i_1j_2} \mid \boldsymbol{\xi}, \boldsymbol{\eta}]] \\
            &= \E[\E[Y_{i_1j_1} \mid \boldsymbol{\xi}, \boldsymbol{\eta}]\E[ Y_{i_1j_2} \mid \boldsymbol{\xi}, \boldsymbol{\eta}]] \\
            &= \E[\lambda \bar{w}(\xi_{i_1}, \eta_{j_1}) \times \lambda \bar{w}(\xi_{i_1}, \eta_{j_2})] \\
            &= \lambda^2 \int \left[ \left(\int \bar{w}(\xi_{i_1}, \eta_{j_1}) d\eta_{j_1} \right) \left(\int \bar{w}(\xi_{i_1}, \eta_{j_2}) d\eta_{j_2} \right) \right] d\xi_{i_1}  \\
            &= \lambda^2 \int f(\xi)^2 d\xi.
        \end{split}
    \end{equation*}
    
    \begin{equation*}
        \begin{split}
            \E[h_C(Y_{(i_1, i_2; j_1)})] &= \E[\E[Y_{i_1j_1} Y_{i_2j_1} \mid \boldsymbol{\xi}, \boldsymbol{\eta}]] \\
            &= \E[\E[Y_{i_1j_1} \mid \boldsymbol{\xi}, \boldsymbol{\eta}]\E[ Y_{i_2j_1} \mid \boldsymbol{\xi}, \boldsymbol{\eta}]] \\
            &= \E[\lambda \bar{w}(\xi_{i_1}, \eta_{j_1}) \times \lambda \bar{w}(\xi_{i_2}, \eta_{j_1})] \\
            &= \lambda^2 \int \left[ \left(\int \bar{w}(\xi_{i_1}, \eta_{j_1}) d\xi_{i_1} \right) \left(\int \bar{w}(\xi_{i_2}, \eta_{j_1}) d\xi_{i_2} \right) \right]  d\eta_{j_1} \\
            &= \lambda^2 \int g(\eta)^2 d\eta.
        \end{split}
    \end{equation*}
\end{proof}

\begin{lemma}
    Let $Y$ be a matrix sampled from a $W$-graph model. Let $h_1$ and $h_2$ be the kernel functions defined as in Section~\ref{sub:example_f2}. We have
    \begin{itemize}
        \item $\E[h_1(Y_{(i_1; j_1, j_2)})] = \lambda^2 \int f(\xi)^2 d\xi = \lambda^2 F_2$,
        \item $\E[h_2(Y_{(i_1, i_2; j_1, j_2)})] = \lambda^2$.
    \end{itemize}
    \label{lem:f2_kernels}
\end{lemma}

\begin{proof}
    The proof for $h_1$ is identical to that for $h_B$ in the proof of Lemma~\ref{lem:graphon_kernels}. For $h_2$,
    \begin{equation*}
        \begin{split}
            \E[h_2(Y_{(i_1, i_2; j_1, j_2)})] &= \E[\E[Y_{i_1j_1}  Y_{i_2j_2} \mid \boldsymbol{\xi}, \boldsymbol{\eta}]] \\
            &= \E[\E[Y_{i_1j_1} \mid \boldsymbol{\xi}, \boldsymbol{\eta}]\E[Y_{i_2j_2} \mid \boldsymbol{\xi}, \boldsymbol{\eta}]] \\
            &= \E[\lambda \times \lambda] \\
            &= \lambda^2.
        \end{split}
    \end{equation*}
\end{proof}

\section{Computation times for variance estimators}
\label{app:computation_times}

This section contains Tables~\ref{tab:comp_time_h1} and~\ref{tab:comp_time_h2} reporting the computation times for Algorithms A, B and C of~\ref{sub:simu_variance} for estimating the asymptotic variance of $U^{h_1}_N$ and $U^{h_2}_N$. 

As expected, Algorithm A is much slower than the proposed approach of this paper (Algorithm B and C). Notably, Algorithm C outspeeds Algorithm B by efficiently leveraging matrix operations, which are highly optimized in most computing libraries. While Algorithms A and B exhibit similar computational costs between estimating $V^{h_1}$ and $V^{h_2}$, Algorithm C is approximately twice as fast for $V^{h_1}$ compared to $V^{h_2}$. This reflects the fact that the matrix operations required for computing $U^{h_1}_N$ are simpler than for computing $U^{h_2}_N$. 

\begin{table}[tb!]
\centering
\begin{tabular}{ l c c c }
 \toprule
 $N$ & Algorithm A & Algorithm B & Algorithm C \\
 \midrule
 8 & $1.92 \times 10^{-2}$ & $1.24 \times 10^{-3}$ & $3.52 \times 10^{-4}$ \\
 11 & $5.49 \times 10^{-1}$ & $2.75 \times 10^{-3}$ & $3.23 \times 10^{-4}$ \\
 16 &  $7.82$ & $1.08 \times 10^{-2}$ & $4.48 \times 10^{-4}$ \\
 22 & $1.56 \times 10^{2}$ & $4.54 \times 10^{-2}$ & $6.33 \times 10^{-4}$ \\
 32 & - & $1.90 \times 10^{-1}$ & $8.19 \times 10^{-4}$ \\
 45 & - & $8.19 \times 10^{-1}$ & $1.27 \times 10^{-3}$ \\
 64 & - & $3.45$ & $2.46 \times 10^{-3}$ \\
 90 & - & $1.24 \times 10^{1}$ & $5.63 \times 10^{-3}$ \\
 128 & - & $2.52 \times 10^{1}$ & $1.66 \times 10^{-2}$ \\
 \bottomrule
\end{tabular}
\caption{Average time in seconds for estimating $V^{h_1}$ with the three algorithms.}
\label{tab:comp_time_h1}
\end{table}

\begin{table}[tb!]
\centering
\begin{tabular}{ l c c c }
 \toprule
 $N$ & Algorithm A & Algorithm B & Algorithm C \\
 \midrule
 8 & $3.40 \times 10^{-2}$ & $1.09 \times 10^{-3}$ & $ 4.63\times 10^{-4}$ \\
 11 & $6.48 \times 10^{-1}$ & $4.26 \times 10^{-3}$ & $5.26 \times 10^{-4}$ \\
 16 & $8.71$  & $1.15 \times 10^{-2}$ & $7.41 \times 10^{-4}$ \\
 22 & $1.25 \times 10^{2}$ & $4.57 \times 10^{-2}$ & $1.01 \times 10^{-3}$ \\
 32 & - & $1.94 \times 10^{-1}$ & $1.45 \times 10^{-3}$ \\
 45 & - & $8.99 \times 10^{-1}$ & $2.53 \times 10^{-3}$ \\
 64 & - & $3.47$ & $4.38 \times 10^{-3}$ \\
 90 & - & $1.47 \times 10^{1}$ & $1.05 \times 10^{-2}$ \\
 128 & - & $2.81 \times 10^{1}$ & $3.28 \times 10^{-2}$ \\
 \bottomrule
\end{tabular}
\caption{Average time in seconds for estimating $V^{h_2}$ with the three algorithms.}
\label{tab:comp_time_h2}
\end{table}

\end{appendix}

\end{document}